\documentclass{article}
\usepackage{amssymb}
\usepackage{amsmath}
\usepackage{graphicx}
\usepackage{epstopdf}
\usepackage{subfigure}
\usepackage{subfloat}

\newtheorem{theorem}{Theorem}

\newtheorem{definition}[theorem]{Definition}

\newtheorem{lemma}[theorem]{Lemma}

\newenvironment{proof}[1][Proof]{\noindent\textbf{#1.} }{\ \rule{0.5em}{0.5em}}
\input{tcilatex}
\begin{document}

\title{Approximate discrete-time schemes for the estimation of diffusion
processes from complete observations}
\author{J.C. Jimenez \\
%EndAName
{\small Instituto de Cibern\'{e}tica, Matem\'{a}tica y F\'{\i}sica}\\
{\small Departamento de Matem\'{a}tica Interdisciplinaria}\\
{\small Calle 15, no. 551, Vedado, La Habana, Cuba}\\
{\small e-mail: jcarlos@icimaf.cu}} \maketitle

\begin{abstract}
In this paper, a modification of the conventional approximations to
the quasi-maximum likelihood method is introduced for the parameter
estimation of diffusion processes from discrete observations. This
is based on a convergent approximation to the first two conditional
moments of the diffusion process through discrete-time schemes. It
is shown that, for finite samples, the resulting approximate
estimators converge to the quasi-maximum likelihood one when the
error between the discrete-time approximation and the diffusion
process decreases. For an increasing number of observations, the
approximate estimators are asymptotically normal distributed and
their bias decreases when the mentioned error does it. A simulation
study is provided to illustrate the performance of the new
estimators. The results show that, with respect to the conventional
approximate estimators, the new ones significantly enhance the
parameter estimation of the test equations. The proposed estimators
are intended for the recurrent practical situation where a nonlinear
stochastic system should be identified from a reduced number of
complete observations distant in time.
\end{abstract}

\section{Introduction}
The statistical inference for diffusion processes described by
Stochastic Differential Equations (SDEs) is currently a subject of
intensive researches. A basic difficulty of this statistical problem
is that, except for a few simple examples, the joint distribution of
the discrete-time observations of the process has unknown
closed-form. To overcome this, a number of estimators based on
analytical and simulated approximations have been developed during
more than three decades. Such methods are the focus of a growing
literature. See, for instance, the review papers by Bibby and
Sorensen (1996), Prakasa-Rao (1999), Nielsen et al. (2000) and
Jimenez et al. (2006).

In particular, the present paper deals with the class of
quasi-maximum likelihood (QML) estimators for the parameter
estimation of SDEs given a time series of complete observations.
These are the estimators obtained by maximizing a normal
log-likelihood function when the assumption of normality is not
satisfied and all the components of the diffusion process are
discretely observed. The simplest approximations to this class of
estimators are derived, for SDEs with additive noise, from the
likelihood of the discrete-time process defined by a numerical
integrator. Typically, they are derived from the Euler-Maruyama
scheme (Prakasa Rao 1983, Yoshida 1992, Florens-Zmirou 1989) or from
the Local Linearization (LL) one (Ozaki 1985, 1992; Shoji \& Ozaki
1997, 1998). It have been demonstrated that, when the distance
between two consecutive observations remains fixed, these
approximate QML estimators are asymptotically biased when the number
of observations increases (Florens-Zmirou 1989). However, a number
of comparative studies among different estimation methods have shown
that, in practical situations in which the distance between
observations is small enough, the approximate QML estimators based
on LL integrators display the better performance due to their
simplicity, computational efficiency and negligible bias (see, e.g.,
Shoji \& Ozaki 1997, Durham \& Gallant 2002, Singer 2002, Hurn et
al. 2007). Therefore, any modification to these approximate QML
methods that yields a bias reduction will be useful. In this
direction, two methods have early been proposed. For the estimators
based on the Euler-Maruyama integrator, Clement (1995) introduced a
correction to the bias by means of simulations. Whereas, on the
basis of Taylor expansions of the first two conditional moments of
the discrete-time process, Kessler (1997) archives similar results.
Depending of the specific SDE to be estimated, the first method
could be computationally time demanding, while the second one could
be affected by numerical instabilities resulting from a high order
Taylor expansion. More recently, Huang (2011) proposed new
estimators based on high-order numerical integrators, which improve
the accuracy of the approximation for the first two conditional
moments and can be straightforward applied to SDEs with
multiplicative noise.

A common feature of the approximate QML methods mentioned above is
that, once the observations are given, the error between the
approximate and the exact moments of the diffusion is fixed and
completely determined by the distance between observations. Clearly,
this fixes the bias of the QML estimation for finite samples and
obstructs its asymptotic correction when the number of observations
increases.

In this paper, an alternative modification of the conventional
approximations to the QML estimator for diffusion processes is
introduced, which is oriented to reduce and control the estimation
bias. This is based on a recursive computation of the first two
conditional moments of discrete-time approximations converging to
the diffusion process between two consecutive observations. It is
shown that, for finite samples, the resulting approximate estimators
converge to the exact QML estimator when the error between the
discrete-time approximation and the diffusion process decreases. For
an increasing number of observations, the approximate estimators are
asymptotically normal distributed and their bias decreases when the
above mentioned error does it. As a particular instance, the
approximate QML estimators designed with the well-known Local Linear
approximations for SDEs are presented. Their convergence, practical
algorithms and performance in simulations are also considered in
detail. The simulations show that, with respect to the conventional
QML estimators, the new approximate estimators significantly enhance
the parameter estimation of the test equations given a reduced
number of discrete observations distant in time, which is a typical
situation in many practical inference problems.

The paper is organized as follows. In section 2, basic notations and
definitions are presented. In section 3, the new approximate
estimators are defined and some of their properties are studied. As
example, the order-$\beta $ QML estimator based on the Local
Linearization schemes is presented in Section 4, as well as
algorithms for its practical implementation. In the last section,
the performance of the new estimators is illustrated with various
examples.

\section{Notation and preliminary\label{PLK Method Section}}

Let $(\Omega ,\mathcal{F},P)$ be the underlying complete probability
space and $\{\mathcal{F}_{t},$ $t\geq t_{0}\}$ be an increasing
right continuous
family of complete sub $\sigma $-algebras of $\mathcal{F}$. Consider a $d$%
-dimensional diffusion process $\mathbf{x}$ defined by the following
stochastic differential equation

\begin{equation}
d\mathbf{x}(t)=\mathbf{f}(t,\mathbf{x}(t);\mathbf{\theta }%
)dt+\sum\limits_{i=1}^{m}\mathbf{g}_{i}(t,\mathbf{x}(t);\mathbf{\theta })d%
\mathbf{w}^{i}(t)  \label{SDE PLK}
\end{equation}%
for $t\geq t_{0}\in
%TCIMACRO{\U{211d} }%
%BeginExpansion
\mathbb{R}
%EndExpansion
$, where $\mathbf{f}$ and $\mathbf{g}_{i}$ are differentiable functions, $%
\mathbf{w=(\mathbf{w}}^{1},..,\mathbf{w}^{m}\mathbf{)}$ is an $m$%
-dimensional $\mathcal{F}_{t}$-adapted standard Wiener process, $\mathbf{%
\theta }\in \mathcal{D}_{\theta }$ is a vector of parameters, and $\mathcal{D%
}_{\theta }\subset
%TCIMACRO{\U{211d} }%
%BeginExpansion
\mathbb{R}
%EndExpansion
^{p}$ is a compact set. Linear growth, uniform Lipschitz and
smoothness conditions on the functions $\mathbf{f}$ and
$\mathbf{g}_{i}$ that ensure the existence and uniqueness of a
strong solution of (\ref{SDE PLK}) with bounded moments are assumed
for all $\mathbf{\theta }\in \mathcal{D}_{\theta }$.

Denote by $\mathbf{z}$ the diffusion process defined by (\ref{SDE
PLK}) with $\mathbf{\theta =\theta }_{0}\in \mathcal{D}_{\theta }$,
and suppose that $M$ observations of the process $\mathbf{z}$ on an
increasing sequence of time instants
$\{t\}_{M}=\{t_{k}:t_{k}<t_{k+1}$, $k=0,1,..,M-1\}$ are given. More
precisely, denote by $\mathbf{z}_{k}$ the observation of the process $%
\mathbf{z}$ at $t_{k}$ for all $t_{k}\in \{t\}_{M}$ and by $Z=\{\mathbf{z}%
_{0},..,\mathbf{z}_{M-1}\}$ the sequence of these observations.

The inference problem to be consider here is the estimation of the
parameter $\mathbf{\theta }_{0}$ of the SDE (\ref{SDE PLK}) given
the time series $Z$. In particular, let us consider the
quasi-maximum likelihood estimator defined by
\begin{equation}
\widehat{\mathbf{\theta }}_{M}=\arg \{\underset{\mathbf{\theta }}{\mathbf{%
\min }}U_{M}(\mathbf{\theta },Z)\}  \label{QML estimator}
\end{equation}%
where
\begin{equation*}
U_{M}(\mathbf{\theta },Z)=(M-1)\ln (2\pi )+\sum\limits_{k=1}^{M-1}\ln (\det (%
\mathbf{\Sigma }_{k}))+(\mathbf{z}_{k}-\mathbf{\mu }_{k})^{\intercal }(%
\mathbf{\Sigma }_{k})^{-1}(\mathbf{z}_{k}-\mathbf{\mu }_{k}),
\end{equation*}%
and $\mathbf{\mu }_{k}=E(\mathbf{x}(t_{k})|\mathbf{z}_{k-1})$ and $\mathbf{%
\Sigma }_{k}=E(\mathbf{x}(t_{k})\mathbf{x}^{\intercal }(t_{k})|\mathbf{z}%
_{k-1})-\mathbf{\mu }_{k}\mathbf{\mu }_{k}^{\intercal }$ denote the
conditional mean and variance of the diffusion process $\mathbf{x}$
at $t_{k}
$ given $\mathbf{z}_{k-1}$, for all $t_{k-1},t_{k}\in \{t\}_{M}$ and $%
\mathbf{\theta }\in \mathcal{D}_{\theta }$. Because the first two
conditional moments of $\mathbf{x}$ are correctly specified, the
score of the normal log-likelihood satisfies the martingale
difference property, and so the QML estimator (\ref{QML estimator})
is consistent and has an
asymptotically normal distribution. See Bollerslev \& Wooldridge\textbf{\ }%
(1992) and Wooldridge\textbf{\ }(1994) for ergodic and no ergodic
processes, respectively.

In general, since the conditional mean and variance of equation
(\ref{SDE
PLK}) have not explicit formulas, approximations to them are needed. If $%
\widetilde{\mathbf{\mu }}_{k}$ and $\widetilde{\mathbf{\Sigma
}}_{k}$ are approximations to $\mathbf{\mu }_{k}$ and
$\mathbf{\Sigma }_{k}$, then the estimator
\begin{equation*}
\widehat{\mathbf{\vartheta }}_{M}=\arg \{\underset{\mathbf{\theta }}{\mathbf{%
\min }}\widetilde{U}_{M}\mathbf{(\theta },Z)\},
\end{equation*}%
with%
\begin{equation*}
\widetilde{U}_{M}(\mathbf{\theta },Z)=(M-1)\ln (2\pi )+\sum\limits_{k=1}^{M-1}\ln (\det (\widetilde{\mathbf{\Sigma }}_{k}))+(%
\mathbf{z}_{k}-\widetilde{\mathbf{\mu }}_{k})^{\intercal }(\widetilde{%
\mathbf{\Sigma }}_{k})^{-1}(\mathbf{z}_{k}-\widetilde{\mathbf{\mu
}}_{k})
\end{equation*}%
provides an approximation to the quasi-maximum likelihood estimator $%
\widehat{\mathbf{\theta }}_{M}$.

Approximate estimators of this type have early been considered in a
number of papers (Prakasa Rao 1983; Florens-Zmirou 1989; Yoshida
1992; Ozaki 1985, 1992; Shoji \& Ozaki 1997, 1998) and recently in
Huang (2011). In all of
them, the approximate mean $\widetilde{\mathbf{\mu }}_{k}=E(\mathbf{y}_{k}|%
\mathbf{z}_{k-1})$ and variance $\widetilde{\mathbf{\Sigma }}_{k}=E((\mathbf{%
y}_{k}-\widetilde{\mathbf{\mu }}_{k})(\mathbf{y}_{k}-\widetilde{\mathbf{\mu }%
}_{k})^{\intercal }|\mathbf{z}_{k-1})$ are derived from a
discrete-time
scheme $\mathbf{y}_{k}=\mathbf{z}_{k-1}+\phi (t_{k-1},\mathbf{z}%
_{k-1},t_{k}-t_{k-1})$ that approximate to $\mathbf{x}(t_{k})$ in
just one step of size $t_{k}-t_{k-1}$ from the observation
$\mathbf{z}_{k-1}$. Indistinctly, these estimators are called
pseudo-likelihood estimators, or minimum contrast estimators or
prediction error estimators depending of the inferential
considerations that want to be emphasized. It has been proved
(Florens-Zmirou, 1989) that, for the time partition
$\{t\}_{M}=\left\{ t_{k}=k\mathbf{\delta }:k=0,1,\cdots ,M-1\right\}
$ with $\mathbf{\delta }>0$
fixed, these estimators are biased as $M\mathbf{\delta }\rightarrow \infty $%
. Contrary, they are asymptotically unbiased on the time partition $%
\{t\}_{M}=\left\{ t_{k}=k\mathbf{\delta }_{M}:k=0,1,\cdots
,M-1\right\} $ in
the case that $M\mathbf{\delta }_{M}\rightarrow \infty $, but with $M\mathbf{%
\delta }_{M}^{3}\rightarrow 0$ (or more accurately with $M\mathbf{\delta }%
_{M}^{2}\rightarrow 0$ as in Yoshida, 1992). However, last
restriction on $M$ and $\mathbf{\delta }_{M}$ imposes too strong
relation among the number of observations and the time distance
between them, which is very inconvenient from a practical viewpoint.
Further note that, once the data $Z$ are given (and
so the time partition $\{t\}_{M}$ is specified), the error between $\mathbf{y%
}_{k}$ and $\mathbf{x}(t_{k})$ is completely settled by
$t_{k}-t_{k-1}$ and can not be reduced. In this way, the difference
between the approximate quasi-maximum likelihood estimator
$\widehat{\mathbf{\vartheta }}_{M}$ and the exact one
$\widehat{\mathbf{\theta }}_{M}$ can not be reduced neither.

Denote by $\mathcal{C}_{P}^{l}(\mathbb{R}^{d},\mathbb{R})$ the space
of $l$ time continuously differentiable functions
$g:\mathbb{R}^{d}\rightarrow \mathbb{R}$ for which $g$ and all its
partial derivatives up to order $l$ have polynomial growth.

\section{Order-$\protect\beta $ quasi-maximum likelihood estimator}

Let $\left( \tau \right) _{h>0}=\{\tau _{n}:\tau _{n+1}-\tau _{n}\leq h,$ $%
n=0,1,\ldots ,N\}$ be a time discretization of $[t_{0},t_{M-1}]$ such that $%
\left( \tau \right) _{h}\supset \{t\}_{M}$, and $\mathbf{y}_{n}$ be
the approximate value of $\mathbf{x}(\tau _{n})$ obtained from a
discretization of the equation (\ref{SDE PLK}) for all $\tau _{n}\in
\left( \tau \right)
_{h}$. Let us consider the continuous time approximation $\mathbf{y}=\{%
\mathbf{y}(t),$ $t\in \lbrack t_{0},t_{M-1}]:\mathbf{y}(\tau _{n})=\mathbf{y}%
_{n}$ for all $\tau _{n}\in \left( \tau \right) _{h}\}$ of
$\mathbf{x}$ with
initial conditions%
\begin{equation*}
E\left( \mathbf{y}(t_{0})\text{{\LARGE \TEXTsymbol{\vert}}}\mathcal{F}%
_{t_{0}}\right) =E\left( \mathbf{x}(t_{0})\text{{\LARGE \TEXTsymbol{\vert}}}%
\mathcal{F}_{t_{0}}\right) \text{ \ \ and \ }E\left( \mathbf{y}(t_{0})%
\mathbf{y}^{\intercal }(t_{0})\text{{\LARGE \TEXTsymbol{\vert}}}\mathcal{F}%
_{t_{0}}\right) =E\left( \mathbf{x}(t_{0})\mathbf{x}^{\intercal }(t_{0})%
\text{{\LARGE \TEXTsymbol{\vert}}}\mathcal{F}_{t_{0}}\right) ;\text{
}
\end{equation*}%
satisfying the bound condition
\begin{equation}
E\left( \left\vert \mathbf{y}(t)\right\vert ^{2q}\text{{\LARGE \TEXTsymbol{%
\vert}}}\mathcal{F}_{t_{k}}\right) \leq L  \label{LMVF6}
\end{equation}%
for all $t\in \lbrack t_{k},t_{k+1}]$; and the weak convergence
criteria
\begin{equation}
\underset{t_{k}\leq t\leq t_{k+1}}{\sup }\left\vert E\left( g(\mathbf{x}(t))%
\text{{\LARGE \TEXTsymbol{\vert}}}\mathcal{F}_{t_{k}}\right) -E\left( g(%
\mathbf{y}(t))\text{{\LARGE
\TEXTsymbol{\vert}}}\mathcal{F}_{t_{k}}\right) \right\vert \leq
L_{k}h^{\beta }  \label{LMVF7}
\end{equation}%
for all $t_{k},t_{k+1}\in \{t\}_{M}$ and $\mathbf{\theta }\in \mathcal{D}%
_{\theta }$, where $g\in \mathcal{C}_{P}^{2(\beta +1)}(\mathbb{R}^{d},%
\mathbb{R})$, $L$ and $L_{k}$ are positive constants, $\beta \in
%TCIMACRO{\U{2115} }%
%BeginExpansion
\mathbb{N}
%EndExpansion
_{+}$, and $q=1,2...$. The process $\mathbf{y}$ defined in this way
is typically called order-$\beta $ approximation to $\mathbf{x}$ in
weak sense
(Kloeden \& Platen, 1999). In addition, the second conditional moment of $%
\mathbf{y}$ is assumed to be positive definite and continuous for all $%
\mathbf{\theta }\in \mathcal{D}_{\theta }$.

When an order-$\beta $ approximation to the solution of equation
(\ref{SDE PLK}) is chosen, the following approximate quasi-maximum
likelihood estimator can be naturally defined.

\begin{definition}
\label{Definition order-B QML estimator}Given a time series $Z$ of
$M$ observations of the SDE (\ref{SDE PLK}) with $\mathbf{\theta
=\theta }_{0}$ on $\{t\}_{M}$, the order-$\beta $ quasi-maximum
likelihood estimator for
the parameters of (\ref{SDE PLK}) is defined by%
\begin{equation}
\widehat{\mathbf{\theta }}_{M}(h)=\arg \{\underset{\mathbf{\theta }}{\mathbf{%
\min }}U_{M,h}\mathbf{(\theta },Z)\},  \label{order-B QML estimator}
\end{equation}%
where%
\begin{equation*}
U_{M,h}(\mathbf{\theta },Z)=(M-1)\ln (2\pi )+\sum\limits_{k=1}^{M-1}\ln (\det (%
\mathbf{\Sigma }_{h,k}))+(\mathbf{z}_{k}-\mathbf{\mu }_{h,k})^{\intercal }(%
\mathbf{\Sigma }_{h,k})^{-1}(\mathbf{z}_{k}-\mathbf{\mu }_{h,k}),
\end{equation*}%
\textbf{$\mu $}$_{h,k}=E(\mathbf{y}(t_{k})|\mathbf{z}_{k-1})$, $\mathbf{%
\Sigma }_{h,k}=E(\mathbf{y}(t_{k})\mathbf{y}^{\intercal }(t_{k})|\mathbf{z}%
_{k-1})-\mathbf{\mu }_{h,k}\mathbf{\mu }_{h,k}^{\intercal }$,
$\mathbf{y}$ is an order-$\beta $ approximation to the solution of
(\ref{SDE PLK}) in
weak sense such that $E(\mathbf{y}(t_{k})|\mathbf{z}_{k})=\mathbf{z}_{k}$ and $E(%
\mathbf{y}(t_{k})\mathbf{y}^{\intercal }(t_{k})|\mathbf{z}_{k})=\mathbf{z}%
_{k}\mathbf{z}_{k}^{\intercal }$ for all $t_{k}\in \{t\}_{M}$, and
$h$ is the maximum stepsize of the time discretization $\left( \tau
\right) _{h}\supset \{t\}_{M}$ associated to $\mathbf{y}$.
\end{definition}

In principle, according to the above definition, any kind of approximation $%
\mathbf{y}$ converging to $\mathbf{x}$ in a weak sense can be used
to construct an approximate order-$\beta $ quasi-maximum likelihood
estimator, e.g., those considered in Kloeden \& Platen (1999). In
this way, the Euler-Maruyama, the Local Linearization and any high
order numerical scheme
for SDEs might be used as well, but the approximations \textbf{$\mu $}$%
_{h,k} $ and $\mathbf{\Sigma }_{h,k}$ will be now derived from the
conditional moments of the numerical scheme after various iterations
with stepsizes lower than $t_{k}-t_{k-1}$. Note that, when $\left(
\tau \right) _{h}\equiv \{t\}_{M}$, the so defined order-$\beta $
quasi-maximum likelihood estimator reduces to the corresponding
approximate quasi-maximum likelihood estimator mentioned in Section
\ref{PLK Method Section}. That is, to one of those considered in
Prakasa Rao (1983), Yoshida (1992), Florens-Zmirou (1989), Ozaki
(1985,1992), Shoji \& Ozaki (1997,1998), or Huang (2011).

Note that the goodness of the approximation $\mathbf{y}$ to
$\mathbf{x}$ is measured (in weak sense) by the left hand side of
(\ref{LMVF7}). Thus, the
inequality (\ref{LMVF7}) gives a bound for the errors of the approximation $\mathbf{y%
}$ to $\mathbf{x}$, for all $t\in \lbrack t_{k},t_{k+1}]$ and all
pair of consecutive observations $t_{k},t_{k+1}\in \{t\}_{M}$.
Moreover, this inequality states the convergence (in weak sense and
with rate $\beta $) of the approximation $\mathbf{y}$ to
$\mathbf{x}$ as the maximum stepsize $h$ of the time discretization
$(\tau )_{h}\supset \{t\}_{M}$ goes to zero. Clearly this includes,
as particular case, the convergence of the first two conditional
moments of $\mathbf{y}$ to those of $\mathbf{x}$. Since the
approximate estimator in Definition \ref{Definition order-B QML
estimator} is designed in terms of the first two conditional moments
of the approximation $\mathbf{y}$, the weak convergence of
$\mathbf{y}$ to $\mathbf{x}$ should imply the convergence of the
approximate QML estimator to the exact one and the similarity of
their asymptotic properties, as $h$ goes to zero. Next results deal
with these matters.

\subsection{Convergence}

For a finite sample $Z$ of $M$ observation of (\ref{SDE PLK}), the
following convergence results are useful.

\begin{theorem}
\label{PLK convergence theorem}Let $Z$ be a time series of $M$
observations of the SDE (\ref{SDE PLK}) with $\mathbf{\theta =\theta
}_{0}$ on the time
partition $\{t\}_{M}$. Let $\widehat{\mathbf{\theta }}_{M}$ and $\widehat{%
\mathbf{\theta }}_{M}(h)$ be, respectively, the quasi-maximum
likelihood and
an order-$\beta $ quasi-maximum likelihood estimator for the parameters of (%
\ref{SDE PLK}) given $Z$. Then%
\[
\left\vert \widehat{\mathbf{\theta }}_{M}(h)-\widehat{\mathbf{\theta }}%
_{M}\right\vert \rightarrow 0
\]%
as $h\rightarrow 0$. Moreover,%
\[
E(\left\vert \widehat{\mathbf{\theta
}}_{M}(h)-\widehat{\mathbf{\theta }}_{M}\right\vert )\rightarrow 0
\]%
as $h\rightarrow 0$, where the expectation is with respect to the
measure on
the underlying probability space generating the realizations of the SDE (\ref%
{SDE PLK}) with $\mathbf{\theta =\theta }_{0}$.
\end{theorem}

\begin{proof}
Defining $\Delta \mathbf{\Sigma }_{h,k}=\mathbf{\Sigma }_{k}-\mathbf{\Sigma }%
_{h,k}$, it follows that
\begin{eqnarray}
\det (\mathbf{\Sigma }_{h,k}) &=&\det (\mathbf{\Sigma }_{k}-\Delta \mathbf{%
\Sigma }_{h,k})  \nonumber \\
&=&\det (\mathbf{\Sigma }_{k})\det (\mathbf{I}-\mathbf{\Sigma }%
_{k}^{-1}\Delta \mathbf{\Sigma }_{h,k})  \label{PLK Ind 1}
\end{eqnarray}%
and%
\begin{eqnarray}
\mathbf{\Sigma }_{h,k}^{-1} &=&(\mathbf{\Sigma }_{k}-\Delta
\mathbf{\Sigma }
_{h,k})^{-1}  \nonumber \\
&=&\mathbf{\Sigma }_{k}^{-1}+\mathbf{\Sigma }_{k}^{-1}\Delta
\mathbf{\Sigma }
_{h,k}(\mathbf{I}-\mathbf{\Sigma }_{k}^{-1}\Delta \mathbf{\Sigma }%
_{h,k})^{-1}\mathbf{\Sigma }_{k}^{-1}.  \label{PLK Ind 2}
\end{eqnarray}%
By using these two identities and the identity
\begin{eqnarray}
(\mathbf{z}_{k}-\mathbf{\mu }_{h,k})^{\intercal }(\mathbf{\Sigma }%
_{h,k})^{-1}(\mathbf{z}_{k}-\mathbf{\mu }_{h,k}) &=&(\mathbf{z}_{k}-\mathbf{%
\mu }_{k})^{\intercal }(\mathbf{\Sigma }_{h,k})^{-1}(\mathbf{z}_{k}-\mathbf{%
\mu }_{k})  \nonumber \\
&&+(\mathbf{z}_{k}-\mathbf{\mu }_{k})^{\intercal }(\mathbf{\Sigma }%
_{h,k})^{-1}(\mathbf{\mu }_{k}-\mathbf{\mu }_{h,k})  \nonumber \\
&&+(\mathbf{\mu }_{k}-\mathbf{\mu }_{h,k})^{\intercal }(\mathbf{\Sigma }%
_{h,k})^{-1}(\mathbf{z}_{k}-\mathbf{\mu }_{k})  \nonumber \\
&&+(\mathbf{\mu }_{k}-\mathbf{\mu }_{h,k})^{\intercal }(\mathbf{\Sigma }%
_{h,k})^{-1}(\mathbf{\mu }_{k}-\mathbf{\mu }_{h,k})  \label{PLK Ind
3}
\end{eqnarray}%
it is obtained that
\begin{equation}
U_{M,h}(\mathbf{\theta },Z)=U_{M}(\mathbf{\theta
},Z)+R_{M,h}(\mathbf{\theta }),  \label{PLK Ind 4}
\end{equation}%
where $U_{M}$ and $U_{M,h}$ are defined in (\ref{QML estimator}) and (\ref%
{order-B QML estimator}), respectively, and
\begin{eqnarray*}
R_{M,h}(\mathbf{\theta }) &=&\sum\limits_{k=1}^{M-1}\ln (\det (\mathbf{I}-%
\mathbf{\Sigma }_{k}^{-1}\Delta \mathbf{\Sigma }_{h,k}))+(\mathbf{z}_{k}-%
\mathbf{\mu }_{k})^{\intercal }\mathbf{M}_{h,k}\mathbf{(z}_{k}-\mathbf{\mu }%
_{k}) \\
&&+(\mathbf{z}_{k}-\mathbf{\mu }_{k})^{\intercal }(\mathbf{\Sigma }%
_{h,k})^{-1}(\mathbf{\mu }_{k}-\mathbf{\mu }_{h,k})+(\mathbf{\mu }_{k}-%
\mathbf{\mu }_{h,k})^{\intercal }(\mathbf{\Sigma }_{h,k})^{-1}(\mathbf{z}%
_{k}-\mathbf{\mu }_{k}) \\
&&+(\mathbf{\mu }_{k}-\mathbf{\mu }_{h,k})^{\intercal }(\mathbf{\Sigma }%
_{h,k})^{-1}(\mathbf{\mu }_{k}-\mathbf{\mu }_{h,k})
\end{eqnarray*}%
with $\mathbf{M}_{h,k}=\mathbf{\Sigma }_{k}^{-1}\Delta \mathbf{\Sigma }%
_{h,k}(\mathbf{I-\Sigma }_{k}^{-1}\Delta \mathbf{\Sigma }_{h,k})^{-1}\mathbf{%
\Sigma }_{k}^{-1}$.

For the functions $g(\mathbf{x}(t))=\mathbf{x}^{i}(t)$ and $g(\mathbf{x}(t))=%
\mathbf{x}^{i}(t)\mathbf{x}^{j}(t)$ belonging to the function space $%
\mathcal{C}_{P}^{2(\beta +1)}(\mathbb{R}^{d},\mathbb{R})$, for all $i,j=1..d$%
, condition (\ref{LMVF7}) directly implies that
\begin{equation}
\left\vert E(\mathbf{x}(t_{k})|\mathbf{z}_{k-1})-E(\mathbf{y}(t_{k})|\mathbf{%
z}_{k-1})\right\vert \leq \sqrt{d}L_{k-1}h^{\beta }  \label{PLK Ind
5}
\end{equation}%
and
\begin{equation}
\left\vert E(\mathbf{x}(t_{k})\mathbf{x}^{\intercal }(t_{k})|\mathbf{z}%
_{k-1})-E(\mathbf{y}(t_{k})\mathbf{y}^{\intercal }(t_{k})|\mathbf{z}%
_{k-1})\right\vert \leq dL_{k-1}h^{\beta }.  \label{PLK Ind 6}
\end{equation}%
From this and the finite bound for the conditional mean of $\mathbf{x}$ and $%
\mathbf{y}$, it is obtained that%
\[
\left\vert \mathbf{\mu }_{k}-\mathbf{\mu }_{h,k}\right\vert
\rightarrow
\mathbf{0}\text{ \ \ \ \ \ \ and \ \ \ \ \ \ \ }\left\vert \mathbf{\Sigma }%
_{k}-\mathbf{\Sigma }_{h,k}\right\vert \rightarrow \mathbf{0}
\]%
as $h\rightarrow 0$ for all $\mathbf{\theta }\in \mathcal{D}_{\theta }$ and $%
k=1,..,M-1$. This and the finite bound for the first two conditional
moments
of $\mathbf{x}$ and $\mathbf{y}$ imply that $R_{M,h}(\mathbf{\theta }%
)\rightarrow \mathbf{0}$ as well with $h$. From this and (\ref{PLK
Ind 4}),
\begin{equation}
\left\vert \widehat{\mathbf{\theta }}_{M}(h)-\widehat{\mathbf{\theta }}%
_{M}\right\vert =\left\vert \arg \{\underset{\mathbf{\theta }}{\mathbf{\min }%
}\{U_{M}(\mathbf{\theta },Z)+R_{M,h}(\mathbf{\theta })\}\}-\arg \{\underset{%
\mathbf{\theta }}{\mathbf{\min }}U_{M}(\mathbf{\theta
},Z)\}\right\vert \rightarrow 0  \label{PLK Ind 7}
\end{equation}%
as $h\rightarrow 0$, which implies the first assertion of the
theorem.

On the other hand, since the value of the constant $L_{k-1}$ in
(\ref{PLK Ind 5}) and (\ref{PLK Ind 6}) does not depend of a
specific realization of the SDE (\ref {SDE PLK}), from these
inequalities follows that
\[
E(\left\vert
E(\mathbf{x}(t_{k})|\mathbf{z}_{k-1})-E(\mathbf{y}(t_{k})|
\mathbf{z}_{k-1})\right\vert )\leq \sqrt{d}L_{k-1}h^{\beta }
\]%
and
\[
E(\left\vert E(\mathbf{x}(t_{k})\mathbf{x}^{\intercal
}(t_{k})|\mathbf{z} _{k-1})-E(\mathbf{y}(t_{k})\mathbf{y}^{\intercal
}(t_{k})|\mathbf{z} _{k-1})\right\vert )\leq dL_{k-1}h^{\beta },
\]
where the new expectation here is with respect to the measure on the
underlying probability space generating the realizations of the SDE
(\ref {SDE PLK}) with $\mathbf{\theta =\theta }_{0}$. From this and
(\ref{PLK Ind 7}) follows that $E(\left\vert \widehat{\mathbf{\theta
}}_{M}(h)-\widehat{ \mathbf{\theta }}_{M}\right\vert )\rightarrow 0$
as $h\rightarrow 0$, which concludes the proof.
\end{proof}

The the first assertion of this theorem states that, for each given
data $Z$, the order-$\beta $ QML estimator $\widehat{\mathbf{\theta
}}_{M}(h)$ converges to the exact one $\widehat{\mathbf{\theta
}}_{M}$ as $h$ goes to zero. Because $h$ controls the weak
convergence criteria (\ref{LMVF7}) is then clear that the
order-$\beta $ QML estimator (\ref{order-B QML estimator}) converges
to the exact one (\ref{QML estimator}) when the error (in weak
sense) of the order-$\beta $ approximation $\mathbf{y}$ to
$\mathbf{x}$ decreases. On the other hand, the second assertion
implies that the average of the errors $\left\vert
\widehat{\mathbf{\theta }}_{M}(h)-\widehat{\mathbf{ \theta
}}_{M}\right\vert $ corresponding to different realizations of (\ref
{SDE PLK}) decreases when $h$ does.

Next theorem deals with error between the averages of the estimators
$ \widehat{\mathbf{\theta }}_{M}(h)$ and $\widehat{\mathbf{\theta
}}_{M}$ computed for different realizations of the SDE.

\begin{theorem}
\label{PLK week convergence}Let $Z$ be a time series of $M$
observations of the SDE (\ref{SDE PLK}) with $\mathbf{\theta =\theta
}_{0}$
on the time partition $\{t\}_{M}$. Let $\widehat{\mathbf{\theta }}_{M}$ and $%
\widehat{\mathbf{\theta }}_{M}(h)$ be, respectively, the
quasi-maximum likelihood and an order-$\beta $ quasi-maximum
likelihood estimator for the
parameters of (\ref{SDE PLK}) given $Z$. Then,%
\[
\left\vert E(\widehat{\mathbf{\theta }}_{M}(h))-E(\widehat{\mathbf{\theta }}%
_{M})\right\vert \rightarrow 0
\]%
as $h\rightarrow 0$, where the expectation is with respect to the
measure on
the underlying probability space generating the realizations of the SDE (\ref%
{SDE PLK}) with $\mathbf{\theta =\theta }_{0}$.
\end{theorem}

\begin{proof}
Trivially,
\begin{eqnarray*}
\left\vert E(\widehat{\mathbf{\theta }}_{M}(h))-E(\widehat{\mathbf{\theta }}%
_{M})\right\vert  &=&\left\vert E(\widehat{\mathbf{\theta }}_{M}(h)-\widehat{%
\mathbf{\theta }}_{M})\right\vert  \\
&\leq &E(\left\vert \widehat{\mathbf{\theta }}_{M}(h)-\widehat{\mathbf{%
\theta }}_{M}\right\vert ),
\end{eqnarray*}%
where the expectation here is taken with respect to the measure on
the
underlying probability space generating the realizations of the SDE (\ref%
{SDE PLK}) with $\mathbf{\theta =\theta }_{0}$. From this and the
second assertion of Theorem \ref{PLK convergence theorem}, the proof
is completed.
\end{proof}

Here, it is worth to remak that the conventional approximate QML
estimators mentioned in Section \ref{PLK Method Section} do not have
the desired convergence properties stated in the theorems above for
the order-$\beta $ QML estimator. Further note that, either in
Definition \ref{Definition order-B QML estimator} nor in Theorems
\ref{PLK convergence theorem} and \ref{PLK week convergence} some
restriction on the time partition $\{t\}_{M}$ for the data has been
assumed. Thus, there are not specific constraints about the time
distance between two consecutive observations, which allows the
application of the order-$\beta $ QML estimator in a variety of
practical problems with a reduced number of not close observations
in time, with sequential random measurements, or with multiple
missing data. Neither there are restrictions on the time
discretization $(\tau )_{h}$ $ \supset \{t\}_{M}$ on which the
order-$\beta $ QML estimator is defined. Thus, $(\tau )_{h}$ can be
set by the user by taking into account some specifications or
previous knowledge on the inference problem under consideration, or
automatically designed by an adaptive strategy as it will be shown
in the section concerning the numerical simulations.

\subsection{Asymptotic properties\label{PLK Theoretical Section}}

In this section, asymptotic properties of the approximate
quasi-maximum likelihood estimator $\widehat{\mathbf{\theta
}}_{M}(h)$ will be studied by using a general result obtained in
Ljung and Caines (1979) for prediction
error estimators. According to that, the relation between the estimator $%
\widehat{\mathbf{\theta }}_{M}(h)$ and the global minimum $\mathbf{\theta }%
_{M}^{\ast }$ of the function
\begin{equation}
W_{M}(\mathbf{\theta })=E(U_{M}(\mathbf{\theta },Z))\text{ with }\mathbf{%
\theta }\in \mathcal{D}_{\theta }  \label{PLKW}
\end{equation}%
should be considered, where $U_{M}$ is defined in (\ref{QML
estimator}) and the expectation is taken with respect to the measure
on the underlying probability space generating the realizations of
the SDE (\ref {SDE PLK}). Here, it is worth to remark that
$\mathbf{\theta }_{M}^{\ast }$ is not an estimator of
$\mathbf{\theta }$ since the function $W_{M}$ does not depend of a
given data $Z$. In fact, $\mathbf{\theta }_{M}^{\ast }$ indexes the
best predictor, in the sense that the average prediction error loss
function $W_{M}$ is minimized at this parameter (Ljung \& Caines,
1979).

In what follows, regularity conditions for the unique
identifiability of the SDE (\ref{SDE PLK}) are assumed, which are
typically satisfied by stationary
and ergodic diffusion processes (see, e.g., Bollerslev \& Wooldridge\textbf{%
\ }(1992) or Ljung \& Caines (1979)).

\begin{lemma}
\label{PLK Lemma}If $\mathbf{\Sigma }_{k}$ is positive definite for all $%
k=1,..,M-1$, then the function $W_{M}(\mathbf{\theta })$ defined in (\ref%
{PLKW}) has an unique minimum and
\begin{equation}
\arg \{\underset{\mathbf{\theta \in }\mathcal{D}_{\theta }}{\mathbf{\min }}%
W_{M}(\mathbf{\theta })\}=\mathbf{\theta }_{0}.  \label{M-estimator}
\end{equation}
\end{lemma}

\begin{proof}
Since $\mathbf{\Sigma }_{k}$ is positive definite for all
$k=1,..,M-1$,
Lemma A.2 in Bollerslev \& Wooldridge\textbf{\ }(1992) ensures that $\mathbf{%
\theta }_{0}$ is the unique minimum of the function%
\begin{equation*}
l_{k}(\mathbf{\theta })=E(\ln (\det (\mathbf{\Sigma }_{k}))+(\mathbf{z}_{k}-%
\mathbf{\mu }_{k})^{\intercal }(\mathbf{\Sigma }_{k})^{-1}(\mathbf{z}_{k}-%
\mathbf{\mu }_{k})|\mathbf{z}_{k-1})
\end{equation*}%
on $\mathcal{D}_{\theta }$ for all $k$. Consequently and under the
assumed unique identifiability of the SDE (\ref{SDE PLK}),
$\mathbf{\theta }_{0}$ is then the unique minimum of
\begin{equation*}
W_{M}(\mathbf{\theta })=(M-1)\ln (2\pi )+\sum\limits_{k=1}^{M-1}E(l_{k}(\mathbf{%
\theta }))
\end{equation*}%
on $\mathcal{D}_{\theta }.$
\end{proof}

Here, it is worth to remark that the result of this Lemma is
restricted to the QML estimator (\ref{QML estimator}) for SDEs.
However, for other types of stochastic processes, a similar result
can be found in the proof of Theorem 2.1 of Bollerslev \&
Wooldridge\textbf{\ }(1992) concerning the asymptotic properties of
the QML estimator under more general framework.

Denote by $U_{M,h}^{\prime }$ the derivative of $U_{M,h}$ with respect to $%
\mathbf{\theta }$, and by $W_{M}^{^{\prime \prime }}$ the second
derivative of $W_{M}$ with respect to $\mathbf{\theta }$.

\begin{theorem}
\label{PLK main theorem}Let $Z$ be a time series of $M$ observations
of the SDE (\ref{SDE PLK}) with $\mathbf{\theta =\theta }_{0}$ on
the time
partition $\{t\}_{M}$. \ Let $\widehat{\mathbf{\theta }}_{M}(h)$ be an order-%
$\beta $ quasi-maximum likelihood estimator for the parameters of
(\ref{SDE PLK}) given $Z$. Then
\begin{equation}
\widehat{\mathbf{\theta }}_{M}(h)-\mathbf{\theta }_{0}\rightarrow
\Delta \mathbf{\theta }_{M}(h)  \label{PLK1}
\end{equation}%
w.p.1 as $M\rightarrow \infty $, where $\Delta \mathbf{\theta }%
_{M}(h)\rightarrow 0$ as $h\rightarrow 0$. Moreover, if for some
$M_{0}\in
%TCIMACRO{\U{2115} }%
%BeginExpansion
\mathbb{N}
%EndExpansion
$ there exists $\epsilon >0$ such that%
\begin{equation}
W_{M}^{^{\prime \prime }}(\mathbf{\theta })>\epsilon
\mathbf{I}\text{ \ \ \
and \ \ }\mathbf{H}_{M,h}(\mathbf{\theta )}=ME(U_{M,h}^{\prime }(\mathbf{%
\theta },Z\mathbf{)(}U_{M,h}^{\prime }(\mathbf{\theta },Z\mathbf{))}%
^{\intercal })>\epsilon \mathbf{I}  \label{PLK4}
\end{equation}%
for all $M>M_{0}$ and $\mathbf{\theta }\in \mathcal{D}_{\theta }$, then%
\begin{equation}
\sqrt{M}\mathbf{P}_{M,h}^{-1/2}(\widehat{\mathbf{\theta }}_{M}(h)-\mathbf{%
\theta }_{0})\sim \mathcal{N}(\Delta \mathbf{\theta
}_{M}(h),\mathbf{I}) \label{PLK2}
\end{equation}%
as $M\rightarrow \infty $, where $\mathbf{P}_{M,h}=(W_{M}^{^{\prime
\prime
}}(\mathbf{\theta }_{0}+\Delta \mathbf{\theta }_{M}(h)))^{-1}\mathbf{H}%
_{M,h}(\mathbf{\theta }_{0}+\Delta \mathbf{\theta
}_{M}(h))(W_{M}^{^{\prime \prime }}(\mathbf{\theta }_{0}+\Delta
\mathbf{\theta }_{M}(h)))^{-1}+\Delta
\mathbf{P}_{M,h}$ with $\Delta \mathbf{P}_{M,h}\rightarrow \mathbf{0}$ as $%
h\rightarrow 0$.
\end{theorem}

\begin{proof}
Let $W_{M,h}(\mathbf{\theta })=E(U_{M,h}(\mathbf{\theta },Z))$ and $\mathbf{%
\alpha }_{M}(h)=\arg \{\underset{\mathbf{\theta \in }\mathcal{D}_{\theta }}{%
\mathbf{\min }}W_{M,h}(\mathbf{\theta })\}$, where $U_{M,h}$ is defined in (%
\ref{order-B QML estimator}).

For a $h$ fixed, Theorem 1 in Ljung \& Caines (1979) implies that%
\begin{equation}
\widehat{\mathbf{\theta }}_{M}(h)-\mathbf{\alpha }_{M}(h)\rightarrow
0 \label{PLK7}
\end{equation}%
w.p.1 as $M\rightarrow \infty $; and
\begin{equation}
\sqrt{M}\mathbf{P}_{M,h}^{-1/2}(\mathbf{\alpha }_{M}(h))(\widehat{\mathbf{%
\theta }}_{M}(h)-\mathbf{\alpha }_{M}(h))\sim
\mathcal{N}(0,\mathbf{I}) \label{PLK8}
\end{equation}%
as $M\rightarrow \infty $, where
\begin{equation*}
\mathbf{P}_{M,h}(\mathbf{\theta })=(W_{M,h}^{\prime \prime }(\mathbf{\theta }%
))^{-1}\text{ }\mathbf{H}_{M,h}(\mathbf{\theta })\text{
}(W_{M,h}^{\prime \prime }(\mathbf{\theta }))^{-1}
\end{equation*}%
with $\mathbf{H}_{M,h}(\mathbf{\theta })=ME(U_{M,h}^{\prime
}(\mathbf{\theta },Z\mathbf{)(}U_{M,h}^{\prime }(\mathbf{\theta
},Z\mathbf{))}^{\intercal })$.

By using the identities (\ref{PLK Ind 1})-(\ref{PLK Ind 3}), the
function
\begin{equation*}
W_{M,h}(\mathbf{\theta })=(M-1)\ln (2\pi )+\sum\limits_{k=1}^{M-1}E(\ln (\det (%
\mathbf{\Sigma }_{h,k}))+(\mathbf{z}_{k}-\mathbf{\mu }_{h,k})^{\intercal }(%
\mathbf{\Sigma }_{h,k})^{-1}(\mathbf{z}_{k}-\mathbf{\mu }_{h,k}))
\end{equation*}%
can be written as
\begin{equation}
W_{M,h}(\mathbf{\theta })=W_{M}(\mathbf{\theta })+E(R_{M,h}(\mathbf{\theta }%
)),  \label{PLK3}
\end{equation}%
where $W_{M}$ is defined in (\ref{PLKW}) and%
\begin{eqnarray*}
R_{M,h}(\mathbf{\theta }) &=&\sum\limits_{k=1}^{M-1}E(\ln (\det (\mathbf{I}-%
\mathbf{\Sigma }_{k}^{-1}\Delta \mathbf{\Sigma }_{h,k}))|\mathcal{F}%
_{t_{k-1}})+E((\mathbf{z}_{k}-\mathbf{\mu }_{k})^{\intercal }\mathbf{M}_{h,k}%
\mathbf{(z}_{k}-\mathbf{\mu }_{k})|\mathcal{F}_{t_{k-1}}) \\
&&+E((\mathbf{z}_{k}-\mathbf{\mu }_{k})^{\intercal }(\mathbf{\Sigma }%
_{h,k})^{-1}(\mathbf{\mu }_{k}-\mathbf{\mu }_{h,k})|\mathcal{F}%
_{t_{k-1}})+E((\mathbf{\mu }_{k}-\mathbf{\mu }_{h,k})^{\intercal }(\mathbf{%
\Sigma }_{h,k})^{-1}(\mathbf{z}_{k}-\mathbf{\mu
}_{k})|\mathcal{F}_{t_{k-1}})
\\
&&+E((\mathbf{\mu }_{k}-\mathbf{\mu }_{h,k})^{\intercal }(\mathbf{\Sigma }%
_{h,k})^{-1}(\mathbf{\mu }_{k}-\mathbf{\mu
}_{h,k})|\mathcal{F}_{t_{k-1}})
\end{eqnarray*}%
with $\mathbf{M}_{h,k}=\mathbf{\Sigma }_{k}^{-1}\Delta \mathbf{\Sigma }%
_{h,k}(\mathbf{I-\Sigma }_{k}^{-1}\Delta \mathbf{\Sigma }_{h,k})^{-1}\mathbf{%
\Sigma }_{k}^{-1}$ and $\Delta $\textbf{$\Sigma $}$_{h,k}=\mathbf{\Sigma }%
_{k}-\mathbf{\Sigma }_{h,k}$.

Denote by $W_{M,h}^{\prime \prime }$ and $R_{M,h}^{\prime \prime }$
the second derivative of $W_{M,h}$ and $R_{M,h}$ with respect to
$\mathbf{\theta }$.

Taking into account that
\begin{eqnarray*}
(W_{M,h}^{\prime \prime }(\mathbf{\theta }))^{-1} &=&(W_{M}^{\prime \prime }(%
\mathbf{\theta })+E(R_{M,h}^{\prime \prime }(\mathbf{\theta })))^{-1} \\
&=&(W_{M}^{\prime \prime }(\mathbf{\theta }))^{-1}+\mathbf{K}_{M,h}(\mathbf{%
\theta })
\end{eqnarray*}%
with
\begin{equation*}
\mathbf{K}_{M,h}(\mathbf{\theta })=-(W_{M}^{\prime \prime }(\mathbf{\theta }%
))^{-1}E(R_{M,h}^{\prime \prime }(\mathbf{\theta }))(\mathbf{I}%
+(W_{M}^{\prime \prime }(\mathbf{\theta }))^{-1}E(R_{M,h}^{\prime \prime }(%
\mathbf{\theta })))^{-1}(W_{M}^{\prime \prime }(\mathbf{\theta
}))^{-1},
\end{equation*}%
it is obtained that%
\begin{equation}
\mathbf{P}_{M,h}(\mathbf{\theta })=(W_{M}^{^{\prime \prime
}}(\mathbf{\theta
}))^{-1}\mathbf{H}_{M,h}(\mathbf{\theta })(W_{M}^{^{\prime \prime }}(\mathbf{%
\theta }))^{-1}+\Delta \mathbf{P}_{M,h}(\mathbf{\theta }),
\label{PLK12}
\end{equation}%
where%
\begin{equation*}
\Delta \mathbf{P}_{M,h}(\mathbf{\theta })=\mathbf{K}_{M,h}(\mathbf{\theta })%
\mathbf{H}_{M,h}(\mathbf{\theta })(W_{M}^{^{\prime \prime }}(\mathbf{\theta }%
))^{-1}+(W_{M}^{^{\prime \prime }}(\mathbf{\theta }))^{-1}\mathbf{H}_{M,h}(%
\mathbf{\theta })\mathbf{K}_{M,h}(\mathbf{\theta })+\mathbf{K}_{M,h}(\mathbf{%
\theta })\mathbf{H}_{M,h}(\mathbf{\theta })\mathbf{K}_{M,h}(\mathbf{\theta }%
).
\end{equation*}

For the functions $g(\mathbf{x}(t))=\mathbf{x}^{i}(t)$ and $g(%
\mathbf{x}(t))=\mathbf{x}^{i}(t)\mathbf{x}^{j}(t)$ belonging to the
function space
$\mathcal{C}_{P}^{2(\beta+1)}(\mathbb{R}^{d},\mathbb{R})$, for all
$i,j=1..d$, condition (\ref{LMVF7}) directly implies that
\begin{equation*}
\left\vert E(\mathbf{x}(t_{k})|\mathbf{z}_{k-1})-E(\mathbf{y}(t_{k})|\mathbf{%
z}_{k-1})\right\vert \leq \sqrt{d} L_{k-1}h^{\beta }
\end{equation*}%
and
\begin{equation*}
\left\vert E(\mathbf{x}(t_{k})\mathbf{x}^{\intercal }(t_{k})|\mathbf{z}%
_{k-1})-E(\mathbf{y}(t_{k})\mathbf{y}^{\intercal }(t_{k})|\mathbf{z}%
_{k-1})\right\vert \leq d L_{k-1}h^{\beta }.
\end{equation*}%
From this and the finite bound for the conditional mean of $\mathbf{x}$ and $%
\mathbf{y}$, it is obtained that%
\begin{equation*}
\left\vert \mathbf{\mu }_{k}-\mathbf{\mu }_{h,k}\right\vert
\rightarrow
\mathbf{0}\text{ \ \ \ \ \ \ and \ \ \ \ \ }\left\vert \text{\ }\mathbf{%
\Sigma }_{k}-\mathbf{\Sigma }_{h,k}\right\vert \rightarrow
\mathbf{0}
\end{equation*}%
as $h\rightarrow 0$ for all $\mathbf{\theta }\in \mathcal{D}_{\theta }$ and $%
k=1,..,M-1$. This and the finite bound for the first two conditional
moments
of $\mathbf{x}$ and $\mathbf{y}$ imply that $\left\vert R_{M,h}(\mathbf{%
\theta },Z)\right\vert \rightarrow 0$ and $\left\vert
R_{M,h}^{\prime \prime }(\mathbf{\theta },Z)\right\vert \rightarrow
0$ as well with $h$. From this and (\ref{PLK3}), it is obtained that
\begin{equation}
W_{M,h}(\mathbf{\theta })\rightarrow W_{M}(\mathbf{\theta })\text{ \
\ \ \ and \ \ \ \ }W_{M,h}^{\prime \prime }(\mathbf{\theta
})\rightarrow
W_{M}^{\prime \prime }(\mathbf{\theta })\text{ \ \ \ as \ \ \ \ \ }%
h\rightarrow 0.  \label{PLK9}
\end{equation}%
In addition, left (\ref{PLK9}) and Lemma \ref{PLK Lemma} imply that
\begin{equation}
\Delta \mathbf{\theta }_{M}(h)=\mathbf{\alpha }_{M}(h)-\mathbf{\theta }%
_{0}=\arg \{\underset{\mathbf{\theta \in }\mathcal{D}_{\theta }}{\mathbf{%
\min }}W_{M,h}(\mathbf{\theta })\}-\arg \{\underset{\mathbf{\theta \in }%
\mathcal{D}_{\theta }}{\mathbf{\min }}W_{M}(\mathbf{\theta })\}\rightarrow 0%
\text{ \ \ \ as \ \ \ \ \ }h\rightarrow 0,  \label{PLK11}
\end{equation}%
whereas from right (\ref{PLK9}) follows that%
\begin{equation}
\Delta \mathbf{P}_{M,h}(\mathbf{\theta })\rightarrow 0\text{ \ \ \
as \ \ \ \ \ }h\rightarrow 0.  \label{PLK10}
\end{equation}

Finally, (\ref{PLK11})-(\ref{PLK10}) together (\ref{PLK7}),
(\ref{PLK8}) and (\ref{PLK12}) imply that (\ref{PLK1}) and
(\ref{PLK2}) hold, which completes the proof.
\end{proof}

Theorem \ref{PLK main theorem} states that, for an increasing number
of observations, the order-$\beta $ QML estimator $\widehat{\mathbf{\theta }}%
_{M}(h)$ is asymptotically normal distributed and its bias decreases
when $h$ goes to zeros. This is a predictable result due to the
known asymptotic properties of the exact QML estimator
$\widehat{\mathbf{\theta }}_{M}$ stated in Bollerslev and
Wooldridge\textbf{\ }(1992) and the convergence of
the approximate estimator $\widehat{\mathbf{\theta }}_{M}(h)$ to $\widehat{%
\mathbf{\theta }}_{M}$ given by Theorem \ref{PLK convergence theorem} when $%
h\rightarrow 0$. Further note that, when $h=0$, the Theorem \ref{PLK
main theorem} reduces to Theorem 1 in Ljung \& Caines (1979) for the
exact QML estimator $\widehat{\mathbf{\theta }}_{M}$. This is other
expected result since the order-$\beta $ QML estimator
$\widehat{\mathbf{\theta }}_{M}(h)$ reduces to the exact one
$\widehat{\mathbf{\theta }}_{M}$ when $h=0$. Further note that,
neither in Theorem \ref{PLK main theorem} there are restrictions on
the time partition $\{t\}_{M}$ for
the data or on the time discretization $(\tau )_{h}$ $%
\supset \{t\}_{M}$ on which the approximate estimator is defined.
Therefore, the comments about them at the end of the previous
subsection are valid here as well.

\section{Order-$\protect\beta $ QML estimator based on Local Linear approximations}

Since, in principle, any type of approximation converging to the
solution of (\ref{SDE PLK}) in a weak sense can be used to construct
an order-$\beta $ QML estimator, some additional criterions could be
considered for the selection of one of them. For instance, high
order of convergence, efficient algorithm for the computation of the
moments, and so on. In this paper, we elected the order-$\beta $
Local Linear approximations (see, e.g., Jimenez \& Biscay, 2002, and
Jimenez \& Ozaki, 2003) for the following reasons: 1) their first
two conditional moments have simple explicit formulas that can be
computed by means of efficient algorithm (including high dimensional
equations) as in Jimenez \& Ozaki (2002,2003) and Jimenez (2012a);
2) their first two conditional moments are exact for linear
equations in all the possible variants (with additive and/or
multiplicative noise, autonomous or not), see Jimenez \& Ozaki
(2002); 3) they have an adequate order $\beta =1,2 $ of weak
convergence (Carbonell et al., 2006 and Jimenez, 2012b); and 4) the
better performance of the conventional QML estimators based on Local
Linearization schemes due to their simplicity, computational
efficiency and negligible bias (see, e.g., Shoji \& Ozaki 1997,
Durham \& Gallant 2002, Singer 2002, Hurn et al. 2007).

It is known that the first two conditional moments of the Local
Linear approximations satisfy a set ordinary differential equations.
Explicit formulas for the solution of these equations can be found
in various papers as it was mentioned before. In what follows, the
simplified expressions derived in Jimenez (2012a) are presented.

Denote by $\mathbf{y}_{\tau _{n}/t_{k}}=E(\mathbf{y}(\tau _{n})|\mathbf{z}%
_{k})$ and $\mathbf{P}_{\tau _{n}/t_{k}}=E(\mathbf{y(}\tau _{n})\mathbf{y}%
^{\intercal }(\tau _{n})|\mathbf{z}_{k})$ the first two conditional
moment of the order-$\beta $ Local Linear approximation $\mathbf{y}$
at $\tau _{n}$ given the observation $\mathbf{z}_{k}$, for all $\tau
_{n}\in \{\left( \tau
\right) _{h}$ $\cap $ $[t_{k},t_{k+1}]\}$ and $k=0,..,M-2$. Clearly, $%
\mathbf{y}_{t_{k+1}/t_{k}}$ and $\mathbf{V}_{t_{k+1}/t_{k}}=\mathbf{P}%
_{t_{k+1}/t_{k}}-\mathbf{y}_{t_{k+1}/t_{k}}\mathbf{y}_{t_{k+1}/t_{k}}^{%
\intercal }$ provide approximations to the exact conditional mean $\mathbf{%
\mu }_{k+1}$ and variance $\mathbf{\Sigma }_{k+1}$, respectively, for all $%
t_{k},t_{k+1}\in \{t\}_{M}$. Moreover, $\mathbf{y}_{t_{k}/t_{k}}=\mathbf{z}%
_{k}$ and
$\mathbf{P}_{t_{k}/t_{k}}=\mathbf{z}_{k}\mathbf{z}_{k}^{\intercal }
$ for all $t_{k}\in \{t\}_{M}$. Let $n_{t}=\max \{n=0,1,\ldots :\tau
_{n}\leq t$ and $\tau _{n}\in \left( \tau \right) _{h}\}$ for all $t\in %
\left[ t_{0},t_{M-1}\right] $.

According \ to Jimenez (2012b), the approximate moments $\mathbf{y}%
_{t_{k+1}/t_{k}}$ and $\mathbf{P}_{t_{k+1}/t_{k}}$ are obtained by
evaluating the recursive formulas%
\begin{equation}
\mathbf{y}_{t/t_{k}}=\mathbf{y}_{\tau _{n_{t}}/t_{k}}+\mathbf{L}_{2}e^{%
\mathbf{M}(\tau _{n_{t}})(t-\tau _{n_{t}})}\mathbf{u}_{\tau
_{n_{t}},t_{k}} \label{ALLF8}
\end{equation}%
and
\begin{equation}
vec(\mathbf{P}_{t/t_{k}})=\mathbf{L}_{1}e^{\mathbf{M}(\tau
_{n_{t}})(t-\tau _{n_{t}})}\mathbf{u}_{\tau _{n_{t}},t_{k}}
\label{ALLF9}
\end{equation}%
at $t=t_{k+1}$, where the vector $\mathbf{u}_{\tau ,t_{k}}$ and the
matrices $\mathbf{M}(\tau )$, $\mathbf{L}_{1}$, $\mathbf{L}_{2}$\
are defined as
\begin{equation*}
\mathbf{M}(\tau )=\left[
\begin{array}{cccccc}
\mathcal{A}(\tau ) & \mathcal{B}_{5}(\tau ) & \mathcal{B}_{4}(\tau )
&
\mathcal{B}_{3}(\tau ) & \mathcal{B}_{2}(\tau ) & \mathcal{B}_{1}(\tau ) \\
\mathbf{0} & \mathbf{C}(\tau ) & \mathbf{I}_{d+2} & \mathbf{0} &
\mathbf{0}
& \mathbf{0} \\
\mathbf{0} & \mathbf{0} & \mathbf{C}(\tau ) & \mathbf{0} &
\mathbf{0} &
\mathbf{0} \\
\mathbf{0} & \mathbf{0} & \mathbf{0} & 0 & 2 & 0 \\
\mathbf{0} & \mathbf{0} & \mathbf{0} & 0 & 0 & 1 \\
\mathbf{0} & \mathbf{0} & \mathbf{0} & 0 & 0 & 0%
\end{array}%
\right] \text{, \ \ }\mathbf{u}_{\tau ,t_{k}}=\left[
\begin{array}{c}
vec(\mathbf{P}_{\tau /t_{k}}) \\
\mathbf{0} \\
\mathbf{r} \\
0 \\
0 \\
1%
\end{array}%
\right] \in
%TCIMACRO{\U{211d} }%
%BeginExpansion
\mathbb{R}
%EndExpansion
^{(d^{2}+2d+7)}
\end{equation*}%
and%
\begin{equation*}
\mathbf{L}_{1}=\left[
\begin{array}{cc}
\mathbf{I}_{d^{2}} & \mathbf{0}_{d^{2}\times (2d+7)}%
\end{array}%
\right] \text{, \ \ \ \ \ \ \ \ \ }\mathbf{L}_{2}=\left[
\begin{array}{ccc}
\mathbf{0}_{d\times (d^{2}+d+2)} & \mathbf{I}_{d} & \mathbf{0}_{d\times 5}%
\end{array}%
\right]
\end{equation*}%
in terms of the matrices and vectors
\begin{equation*}
\mathcal{A}(\tau )=\mathbf{A}(\tau )\mathbf{\oplus A}(\tau
)+\sum\limits_{i=1}^{m}\mathbf{B}_{i}(\tau )\mathbf{\otimes B}%
_{i}^{\intercal }(\tau ),
\end{equation*}%
\begin{equation*}
\mathbf{C(}\tau )=\left[
\begin{array}{ccc}
\mathbf{A}(\tau ) & \mathbf{a}_{1}(\tau ) & \mathbf{A}(\tau )\mathbf{y}%
_{\tau /t_{k}}+\mathbf{a}_{0}(\tau ) \\
0 & 0 & 1 \\
0 & 0 & 0%
\end{array}%
\right] \in \mathbb{R}^{(d+2)\times (d+2)},
\end{equation*}%
\begin{equation*}
\mathbf{r}^{\intercal }=\left[
\begin{array}{ll}
\mathbf{0}_{1\times (d+1)} & 1%
\end{array}%
\right]
\end{equation*}%
$\mathcal{B}_{1}(\tau )=vec(\mathbf{\beta }_{1}(\tau ))+\beta _{4}(\tau )%
\mathbf{y}_{\tau /t_{k}}$, $\mathcal{B}_{2}(\tau )=vec(\mathbf{\beta }%
_{2}(\tau ))+\mathbf{\beta }_{5}(\tau )\mathbf{y}_{\tau /t_{k}}$, $\mathcal{B%
}_{3}(\tau )=vec(\mathbf{\beta }_{3}(\tau ))$, $\mathcal{B}_{4}(\tau )=%
\mathbf{\beta }_{4}(\tau )\mathbf{L}$ and $\mathcal{B}_{5}(\tau )=\mathbf{%
\beta }_{5}(\tau )\mathbf{L}$ with
\begin{align*}
\mathbf{\beta }_{1}(\tau )& =\sum\limits_{i=1}^{m}\mathbf{b}_{i,0}(\tau )%
\mathbf{b}_{i,0}^{\intercal }(\tau ) \\
\mathbf{\beta }_{2}(\tau )& =\sum\limits_{i=1}^{m}\mathbf{b}_{i,0}(\tau )%
\mathbf{b}_{i,1}^{\intercal }(\tau )+\mathbf{b}_{i,1}(\tau )\mathbf{b}%
_{i,0}^{\intercal }(\tau ) \\
\mathbf{\beta }_{3}(\tau )& =\sum\limits_{i=1}^{m}\mathbf{b}_{i,1}(\tau )%
\mathbf{b}_{i,1}^{\intercal }(\tau ) \\
\mathbf{\beta }_{4}(\tau )& =\mathbf{a}_{0}(\tau )\oplus
\mathbf{a}_{0}(\tau
)+\sum\limits_{i=1}^{m}\mathbf{b}_{i,0}(\tau )\otimes \mathbf{B}_{i}(\tau )+%
\mathbf{B}_{i}(\tau )\otimes \mathbf{b}_{i,0}(\tau ) \\
\mathbf{\beta }_{5}(\tau )& =\mathbf{a}_{1}(\tau )\oplus
\mathbf{a}_{1}(\tau
)+\sum\limits_{i=1}^{m}\mathbf{b}_{i,1}(\tau )\otimes \mathbf{B}_{i}(\tau )+%
\mathbf{B}_{i}(\tau )\otimes \mathbf{b}_{i,1}(\tau ),
\end{align*}%
$\mathbf{L}=\left[
\begin{array}{ll}
\mathbf{I}_{d} & \mathbf{0}_{d\times 2}%
\end{array}%
\right] $, and the $d$-dimensional identity matrix $\mathbf{I}_{d}$.
Here,
\begin{equation*}
\mathbf{A}(\tau )=\frac{\partial \mathbf{f}(\tau ,\mathbf{y}_{\tau /t_{k}})}{%
\partial \mathbf{y}}\text{ \ \ \ \ \ \ and \ \ \ \ \ \ }\mathbf{B}_{i}(\tau
)=\frac{\partial \mathbf{g}_{i}(\tau ,\mathbf{y}_{\tau
/t_{k}})}{\partial \mathbf{y}}
\end{equation*}%
are matrices, and the vectors $\mathbf{a}_{0}(\tau _{n_{t}})$, $\mathbf{a}%
_{1}(\tau _{n_{t}})$, $\mathbf{b}_{i,0}(\tau _{n_{t}})$ and $\mathbf{b}%
_{i,1}(\tau _{n_{t}})$ satisfy the expressions%
\begin{equation*}
\mathbf{a}^{\beta }(t;\tau _{n_{t}})=\mathbf{a}_{0}(\tau _{n_{t}})+\mathbf{a}%
_{1}(\tau _{n_{t}})(t-\tau _{n_{t}})\text{ \ \ \ \ and \ \ \ \ }\mathbf{b}%
_{i}^{\beta }(t;\tau _{n_{t}})=\mathbf{b}_{i,0}(\tau _{n_{t}})+\mathbf{b}%
_{i,1}(\tau _{n_{t}})(t-\tau _{n_{t}})
\end{equation*}%
for all $t\in \lbrack t_{k},t_{k+1}]$ and $\tau _{n_{t}}\in \left(
\tau \right) _{h}$, where
\begin{equation*}
\mathbf{a}^{\beta }(t;\tau )=\left\{
\begin{array}{ll}
\mathbf{f}(\tau ,\mathbf{y}_{\tau /t_{k}})-\frac{\partial \mathbf{f}(\tau ,%
\mathbf{y}_{\tau /t_{k}})}{\partial \mathbf{y}}\mathbf{y}_{\tau /t_{k}}+%
\frac{\partial \mathbf{f}(\tau ,\mathbf{y}_{\tau /t_{k}})}{\partial \tau }%
(t-\tau ) & \text{for }\mathbb{\beta }=1 \\
\mathbf{a}^{1}(t;\tau )+\frac{1}{2}\sum\limits_{j,l=1}^{d}[\mathbf{G}(\tau ,%
\mathbf{y}_{\tau /t_{k}})\mathbf{G}^{\intercal }(\tau
,\mathbf{y}_{\tau /t_{k}})]^{j,l}\text{ }\frac{\partial
^{2}\mathbf{f}(\tau ,\mathbf{y}_{\tau
/t_{k}})}{\partial \mathbf{y}^{j}\partial \mathbf{y}^{l}}(t-\tau ) & \text{%
for }\mathbb{\beta }=2%
\end{array}%
\right.
\end{equation*}%
and%
\begin{equation*}
\mathbf{b}_{i}^{\beta }(t;\tau )=\left\{
\begin{array}{ll}
\mathbf{g}_{i}(\tau ,\mathbf{y}(\tau ))-\frac{\partial \mathbf{g}_{i}(\tau ,%
\mathbf{y}_{\tau /t_{k}})}{\partial \mathbf{y}}\mathbf{y}_{\tau /t_{k}}+%
\frac{\partial \mathbf{g}_{i}(\tau ,\mathbf{y}_{\tau
/t_{k}})}{\partial \tau
}(t-\tau ) & \text{for }\mathbb{\beta }=1 \\
\mathbf{b}_{i}^{1}(t;\tau )+\frac{1}{2}\sum\limits_{j,l=1}^{d}[\mathbf{G}%
(\tau ,\mathbf{y}_{\tau /t_{k}})\mathbf{G}^{\intercal }(\tau \mathbf{,y}%
_{\tau /t_{k}})]^{j,l}\text{ }\frac{\partial ^{2}\mathbf{g}_{i}(\tau ,%
\mathbf{y}(\tau ))}{\partial \mathbf{y}^{j}\partial
\mathbf{y}^{l}}(t-\tau )
& \text{for }\mathbb{\beta }=2%
\end{array}%
\right.
\end{equation*}%
are functions associated to the order-$\beta $ Ito-Taylor expansions
for the drift and diffusion coefficients of (\ref{SDE PLK}) in the
neighborhood of $(\tau \mathbf{,y} _{\tau /t_{k}})$, respectively,
and $\mathbf{G=[g}_{1},\ldots ,\mathbf{g} _{m}]$ is an $d\times m$
matrix function. The symbols $vec$, $\oplus $ and $ \otimes $ denote
the vectorization operator, the Kronecker sum and product,
respectively.

Under general conditions, Lemma 7 and Theorem 9 in Jimenez (2012b)
state that the order-$\beta $ Local Linear approximation
$\mathbf{y}$ satisfies
the bound condition (\ref{LMVF6}) and the weak convergence criteria (\ref%
{LMVF7}). Hence, Theorem \ref{PLK convergence theorem} implies that
the order-$\beta $ QML estimator
\begin{equation}
\widehat{\mathbf{\theta }}_{M}(h)=\arg \{\underset{\mathbf{\theta }}{\mathbf{%
\min }}U_{M,h}\mathbf{(\theta },Z)\},  \label{LL-based QML
estimator}
\end{equation}%
with%
\begin{equation*}
U_{M,h}(\mathbf{\theta },Z)=(M-1)\ln
(2\pi)+\sum\limits_{k=0}^{M-2}\ln (\det (
\mathbf{V}_{t_{k+1}/t_{k}}))+(\mathbf{z}_{k+1}-\mathbf{y}_{t_{k+1}/t_{k}})^{
\intercal
}(\mathbf{V}_{t_{k+1}/t_{k}})^{-1}(\mathbf{z}_{k+1}-\mathbf{y}
_{t_{k+1}/t_{k}}),
\end{equation*}%
converges to the exact one (\ref{QML estimator}) as $h$ goes to zero
for all given $Z$. For the same reason, the order-$\beta $ QML
estimator (\ref {LL-based QML estimator}) has the asymptotic
properties stated in Theorem \ref{PLK main theorem}, and the average
of their values for different realizations of the SDE satisfies the
convergence property of Theorem \ref{PLK week convergence}.

For one-dimensional SDEs with additive noise, the order-$\beta $ QML
estimator (\ref{LL-based QML estimator}) reduces to the conventional
estimators of Ozaki (1985, 1992) and Shoji \& Ozaki (1997, 1998)
when $\left( \tau \right) _{h}\equiv \{t\}_{M}$. It is worth to
emphasize that, for each data $\mathbf{z}_{k}$, the formulas
(\ref{ALLF8})-(\ref{ALLF9}) are recursively evaluated at all the
time instants $\tau _{n}\in \{\left( \tau \right) _{h}$ $\cap $
$(t_{k},t_{k+1}]\}$ for the first estimator, whereas they are
evaluated only at $t_{k+1}=\{\{t\}_{M}$ $\cap $ $ (t_{k},t_{k+1}]\}$
for the conventional ones.

From computational viewpoint, each evaluation of\ the formulas
(\ref{ALLF8})-(\ref{ALLF9}) at $\tau _{n}$ requires the computation
of just one exponential matrix whose matrix depends of the drift and
diffusion coefficients of (\ref{SDE PLK}) at $(\tau
_{n-1},\mathbf{y}_{\tau _{n-1}/t_{k}})$. This exponential matrix can
the efficiently computed through the well known Pad\'{e} method
(Moler \& Van Loan, 2003) or, alternatively, by means of the Krylov
subspace method (Moler \& Van Loan, 2003) in the case of high
dimensional SDEs. Even more, low order Pad\'{e} and Krylov methods
as suggested in Jimenez \& de la Cruz (2012) can be used as well for
reducing the computation cost, but preserving the order-$\beta $ of
the approximate moments. Alternatively, simplified formulas for the
moments can be used when the equation to be estimate is autonomous
or has additive noise (see Jimenez, 2012a). All this makes simple
and efficient the evaluation of the approximate predictors
$\mathbf{y}_{t_{k+1}/t_{k}}$ and $\mathbf{V}_{t_{k+1}/t_{k}}$.

In practical situations, it is convenient to write a code that
automatically determines the time discretization $\left( \tau
\right) _{h}$ for achieving a prescribed absolute
($atol_{\mathbf{y}},atol_{\mathbf{P}})$ and relative ($
rtol_{\mathbf{y}},rtol_{\mathbf{P}})$ error tolerance in the
computation of $\mathbf{y}_{t_{k+1}/t_{k}}$ and
$\mathbf{P}_{t_{k+1}/t_{k}}$. With this purpose the adaptive
strategy proposed in Jimenez (2012b) is useful.

\section{Simulation study}

In this section, the performance of the new approximate estimators
is illustrated, by means of simulations, with four test SDEs. To do
so, four
types of QML estimators are computed and compared: 1) the exact one (\ref%
{QML estimator}), when it is possible; 2) the conventional one based
on the LL scheme. That is, the estimator defined by (\ref{LL-based
QML estimator}) with $\left( \tau \right)
_{h}\equiv \{t\}_{M}$ and $\beta =1$; 3) the order-$1$ QML estimator (\ref%
{LL-based QML estimator}) with various uniform time discretizations
$\left(
\tau \right) _{h,T}^{u}$; and 4) the adaptive order-$1$ QML estimator (\ref%
{LL-based QML estimator}) with the adaptive selection of time
discretizations $\left( \tau \right) _{\cdot ,T}$ proposed in
Jimenez (2012b). For each example, histograms and confidence limits
for the estimators are computed from various sets of discrete
observations taken with different time distances (sampling periods)
on time intervals with distinct lengths.

\subsection{Test equations}

\textbf{Example 1.} Equation with multiplicative noise%
\begin{equation}
dx=\alpha txdt+\sigma \sqrt{t}xdw_{1}  \label{SE EJ1}
\end{equation}%
where $\alpha =-0.1$ and $\sigma =0.1$ are parameters to be estimated, and $%
x(t_{0})=1$ is the initial value of $x$ at $t_{0}=0.5$. For this
equation,
the conditional mean and variance of $x$ at $t_{k+1}$ given the observation $%
z_{k}$ of $x$ at $t_{k}$ are
\[
\mathbf{\mu }_{k+1}=z_{k}e^{\alpha (t_{k+1}^{2}-t_{k}^{2})/2}\text{
\ and\ \ \ \ }\mathbf{\Sigma }_{k+1}=z_{k}^{2}e^{(\alpha +\sigma
^{2}/2)(t_{k+1}^{2}-t_{k}^{2})}-\mathbf{\mu }_{k+1}^{2},
\]%
respectively, for all $t_{k+1}>t_{k}\geq t_{0}$.

\textbf{Example 2.} Equation with two additive noise%
\begin{equation}
dx=\alpha txdt+\sigma t^{2}e^{\alpha t^{2}/2}dw_{1}+\rho
\sqrt{t}dw_{2} \label{SE EJ2}
\end{equation}%
where $\alpha =-0.25$, $\sigma =5$, and $\rho =0.1$ are parameters
to be estimated, and $x(t_{0})=10$ is the initial value of $x$ at
$t_{0}=0.01$. For this equation, the conditional mean and variance
of $x$ at $t_{k+1}$
given the observation $z_{k}$ of $x$ at $t_{k}$ are%
\[
\mathbf{\mu }_{k+1}=z_{k}e^{\alpha (t_{k+1}^{2}-t_{k}^{2})/2}\text{
}
\]%
and
\[
\mathbf{\Sigma }_{k+1}=\frac{\rho ^{2}}{2\alpha }e^{\alpha
(t_{k+1}^{2}-t_{k}^{2})}+\frac{\sigma ^{2}}{5}(t_{k+1}^{5}-t_{k}^{5})e^{%
\alpha t_{k+1}^{2}}-\frac{\rho ^{2}}{2\alpha },
\]%
respectively, for all $t_{k+1}>t_{k}\geq t_{0}$.

\textbf{Example 3.} Van der Pool oscillator with random input
(Gitterman,
2005)%
\begin{align}
dx_{1}& =x_{2}dt  \label{SEa EJ3} \\
dx_{2}& =(-(x_{1}^{2}-1)x_{2}-x_{1}+\alpha )dt+\sigma dw  \label{SEb
EJ3}
\end{align}%
where $\alpha =0.5$ and $\sigma ^{2}=(0.75)^{2}$ are, respectively,
the intensity and the variance of the random input that should be
estimated. In addition, $t_{0}=0$, and $\mathbf{x}^{\intercal
}(t_{0})=[1$ $1]$.

\textbf{Example 4.} Van der Pool oscillator with random frequency
(Gitterman, 2005)%
\begin{align}
dx_{1}& =x_{2}dt  \label{SEa EJ4} \\
dx_{2}& =(-(x_{1}^{2}-1)x_{2}-\alpha x_{1})dt+\sigma x_{1}dw
\label{SEb EJ4}
\end{align}%
where $\alpha =1$ and $\sigma ^{2}=1$ are, respectively, the
frequency mean
value and variance that should be estimated. In addition, $t_{0}=0$, and $%
\mathbf{x}^{\intercal }(t_{0})=[1$ $1]$.

In these examples, autonomous or non autonomous, linear or
nonlinear, one or two dimensional equations with additive or
multiplicative noise are considered for the estimation of two or
three parameters. Note that, since the first two conditional moments
of the Examples 1 and 2 have explicit expressions, the exact QML
estimator (\ref{QML estimator}) can be computed.

These four equations have previously been used in Jimenez (2012b) to
illustrate the convergence of the approximate moments (\ref{ALLF8})-(\ref%
{ALLF9}) by means of simulations. Tables with the errors between the
approximate moments and the exact ones as a function of $h$ were
given for the Examples 1 and 2. Tables with the estimated rate of
convergence were provided for the fours examples.

\subsection{Simulations with one-dimensional equations}

For the first two examples, $100$ realizations of the solution were
computed by means of the Euler (Kloeden \& Platen, 1999) or the
Local Linearization scheme (Jimenez et al., 1999) for the equation
with multiplicative or additive noise, respectively. For each
example, the realizations where computed over the thin time
partition $\{t_{0}+10^{-4}n:n=0,..,30\times 10^{4}\}$ to guarantee a
precise simulation of the stochastic solutions on the time interval
$[t_{0},t_{0}+30]$. Twelve subsamples of each realization at the
time instants $\{t\}_{M,T}=\{t_{k}=t_{0}+kT/M:$ $k=0,..,M-1\}$ were
taken as observation $Z$ of $\mathbf{x}$ for making inference with
various values of $M$ and $T$. In particular, the values $T=10,20,30$ and $%
M=T/\delta $ with $\delta =1,0.1,0.01,0.001$ were used. In this way,
twelve
sets of $100$ time series $Z_{\delta ,T}^{i}=\{z_{k}^{i}:k=0,..,M-1,$ $%
M=T/\delta \}$, with $i=1,..,100$, of $M$ observations $z_{k}^{i}$
each one were finally available for each example with the twelve
values of $(\delta ,T)$ mentioned above. This will allow us to
explore and compare the performance of each estimator from
observations taken with different sampling periods $\delta $\ on
time intervals with distinct lengths $T$.

Figure 1 shows the histograms and the confidence limits for both, the exact (%
$\widehat{\alpha }_{\delta ,T}^{E}$) and the conventional ($\widehat{\alpha }%
_{\delta ,T}$) QML estimators of $\alpha $ computed from the twelve sets of $%
100$ time series $Z_{\delta ,T}^{i}$ available for the example 1.
Figure 2 shows the same but, for the exact ($\widehat{\sigma
}_{\delta ,T}^{E}$) and the conventional ($\widehat{\sigma }_{\delta
,T}$) QML estimators of $\sigma $. As it was expected, for the
samples $Z_{\delta ,T}^{i}$ with largest sampling periods, the
parameter estimation is distorted by the well-known lowpass filter
effect of signals sampling (see, e.g., Oppenheim \& Schafer, 2010).
This is the reason of the under estimation of the variance
$\widehat{\sigma }_{\delta ,T}^{E}$ from the samples $Z_{\delta
,T}^{i}$, with $\delta =1$ and $T=10,20,30$, when the parameter
$\alpha $ in
the drift coefficient of (\ref{SE EJ1}) is better estimated by $\widehat{%
\alpha }_{\delta ,T}^{E}$. Contrarily, from these samples, the
conventional QML estimators $\widehat{\alpha }_{\delta ,T}$ can not
provided a good approximation to $\alpha $, and so the whole
unexplained component of the drift term of (\ref{SE EJ1}) included
in the samples is interpreted
as noise by the conventional QML estimators. For this reason, $\widehat{%
\sigma }_{\delta ,T}$ over estimates the value of the parameter
$\sigma $. Further, note that when the sampling period $\delta $
decreases and the length
$T$ of the observation time increases, the difference between the exact ($%
\widehat{\alpha }_{\delta ,T}^{E},\widehat{\sigma }_{\delta
,T}^{E}$) and the conventional ($\widehat{\alpha }_{\delta
,T},\widehat{\sigma }_{\delta ,T}$) QML estimators decreases, as
well as the bias of both estimators. This is also other expected
result. Here, the bias is estimated by the difference between the
parameter value and the estimator average, whereas the difference
between estimators refers to the histogram shape and confidence
limits.{\small \newline }

\begin{tabular}{|l|c|c|c|c|}
\hline $\delta =1$ & $h=\delta $ & $h=\delta /2$ & $h=\delta /8$ &
$h=\delta /32$
\\ \hline
\begin{tabular}{ll}
& $T=10$ \\
$\alpha$ & $T=20$ \\
& $T=30$%
\end{tabular}
& \multicolumn{1}{|l|}{%
\begin{tabular}{l}
$5.7\pm 3.7\times 10^{-3}$ \\ \hline $3.7\pm 1.6\times 10^{-3}$ \\
\hline
$5.8\pm 1.6\times 10^{-3}$%
\end{tabular}%
} & \multicolumn{1}{|l|}{%
\begin{tabular}{l}
$1.3\pm 0.9\times 10^{-3}$ \\ \hline $6.6\pm 3.8\times 10^{-4}$ \\
\hline
$8.1\pm 3.4\times 10^{-4}$%
\end{tabular}%
} & \multicolumn{1}{|l|}{%
\begin{tabular}{l}
$2.2\pm 1.6\times 10^{-4}$ \\ \hline $8.2\pm 7.6\times 10^{-5}$ \\
\hline
$5.4\pm 4.6\times 10^{-5}$%
\end{tabular}%
} & \multicolumn{1}{|l|}{%
\begin{tabular}{l}
$5.4\pm 4.0\times 10^{-5}$ \\ \hline $2.3\pm 2.1\times 10^{-5}$ \\
\hline
$1.7\pm 1.4\times 10^{-5}$%
\end{tabular}%
} \\ \hline
\begin{tabular}{ll}
& $T=10$ \\
$\sigma $ & $T=20$ \\
& $T=30$%
\end{tabular}
& \multicolumn{1}{|l|}{%
\begin{tabular}{l}
$2.8\pm 2.0\times 10^{-2}$ \\ \hline $1.4\pm 1.3\times 10^{-2}$ \\
\hline
$8.4\pm 8.7\times 10^{-3}$%
\end{tabular}%
} & \multicolumn{1}{|l|}{%
\begin{tabular}{l}
$8.6\pm 5.2\times 10^{-3}$ \\ \hline $4.5\pm 2.9\times 10^{-3}$ \\
\hline
$2.6\pm 2.1\times 10^{-3}$%
\end{tabular}%
} & \multicolumn{1}{|l|}{%
\begin{tabular}{l}
$1.7\pm 0.9\times 10^{-3}$ \\ \hline $9.4\pm 5.1\times 10^{-4}$ \\
\hline
$6.2\pm 3.6\times 10^{-4}$%
\end{tabular}%
} & \multicolumn{1}{|l|}{%
\begin{tabular}{l}
$4.0\pm 2.1\times 10^{-4}$ \\ \hline $2.2\pm 1.1\times 10^{-4}$ \\
\hline
$1.5\pm 8.4\times 10^{-5}$%
\end{tabular}%
} \\ \hline
\end{tabular}

{\small Table I. Confidence limits for the error between the exact
and the
approximate QML estimators of the equation (\ref{SE EJ1}). }$h=\delta $%
{\small , for the conventional; and }$h=\delta /2,\delta /8,\delta /32,$%
{\small \ for the order-$1$ on }$\left( \tau \right) _{h,T}^{u}${\small .%
\newline
}

\begin{tabular}{|c|c|c|}
\hline $\delta =1$ & $\alpha$ & $\sigma $ \\ \hline
\begin{tabular}{c}
$h$ \\ \hline $\delta $ \\ \hline $\delta /2$ \\ \hline $\delta /8$
\\ \hline $\delta /32$ \\ \hline
$\cdot $%
\end{tabular}
&
\begin{tabular}{ccc}
$T=10$ & $T=20$ & $T=30$ \\ \hline \multicolumn{1}{r}{$-0.0030$} &
\multicolumn{1}{r}{$-0.0035$} & \multicolumn{1}{r}{$-0.0058$} \\
\hline \multicolumn{1}{r}{$-0.0006$} & \multicolumn{1}{r}{$-0.0005$}
& \multicolumn{1}{r}{$-0.0008$} \\ \hline
\multicolumn{1}{r}{$0$} & \multicolumn{1}{r}{$0.0001$} & \multicolumn{1}{r}{$%
0$} \\ \hline \multicolumn{1}{r}{$0$} & \multicolumn{1}{r}{$0$} &
\multicolumn{1}{r}{$0$}
\\ \hline
\multicolumn{1}{r}{$0$} & \multicolumn{1}{r}{$0$} & \multicolumn{1}{r}{$%
-0.0003$}%
\end{tabular}
&
\begin{tabular}{ccc}
$T=10$ & $T=20$ & $T=30$ \\ \hline \multicolumn{1}{r}{$-0.0286$} &
\multicolumn{1}{r}{$-0.0150$} & \multicolumn{1}{r}{$-0.0071$} \\
\hline \multicolumn{1}{r}{$-0.0086$} & \multicolumn{1}{r}{$-0.0045$}
& \multicolumn{1}{r}{$-0.0027$} \\ \hline
\multicolumn{1}{r}{$-0.0017$} & \multicolumn{1}{r}{$-0.0009$} &
\multicolumn{1}{r}{$-0.0006$} \\ \hline
\multicolumn{1}{r}{$-0.0004$} & \multicolumn{1}{r}{$-0.0002$} &
\multicolumn{1}{r}{$-0.0002$} \\ \hline
\multicolumn{1}{r}{$-0.0002$} & \multicolumn{1}{r}{$-0.0002$} &
\multicolumn{1}{r}{$-0.0013$}%
\end{tabular}
\\ \hline
\end{tabular}

{\small Table II: Difference between the averages of the exact and
the
approximate QML estimators for the equation (\ref{SE EJ1}). }$h=\delta $%
{\small , for the conventional; }$h=\delta /2,\delta /8,\delta /32,${\small %
\ for the order-$1$ on }$\left( \tau \right) _{h,T}^{u}${\small ; and }$%
h=\cdot ,${\small \ for the adaptive order-$1$ on }$\left( \tau
\right) _{\cdot ,T}${\small .\newline }

For the data of (\ref{SE EJ1}) with largest sampling period $\delta
=1$, the
order-$1$ QML estimators ($\widehat{\alpha }_{h,\delta ,T}^{u},\widehat{%
\sigma }_{h,\delta ,T}^{u}$) and ($\widehat{\alpha }_{\cdot ,\delta ,T},%
\widehat{\sigma }_{\cdot ,\delta ,T}$) on uniform $\left( \tau
\right) _{h,T}^{u}=\{\tau _{n}=t_{0}+nh:$ $n=0,..,T/h\}\supset
\{t\}_{T/\delta ,T}$ and adaptive $\left( \tau \right) _{\cdot
,T}\supset \{t\}_{T/\delta ,T}$ time discretizations, respectively,
were computed with $h=\delta
/2,\delta /8,\delta /32$ and tolerances $rtol_{\mathbf{y}}=rtol_{\mathbf{P}%
}=5\times 10^{-6}$ and $atol_{\mathbf{y}}=5\times 10^{-9}$, $atol_{\mathbf{P}%
}=5\times 10^{-12}$. For each data $Z_{\delta ,T}^{i}$, with
$i=1,..,100$, the errors
\begin{equation*}
\varepsilon _{i}(\alpha ,h,\delta ,T)=\left\vert \widehat{\alpha
}_{\delta
,T}^{E}-\widehat{\alpha }_{h,\delta ,T}^{u}\right\vert \text{ and }%
\varepsilon _{i}(\sigma ,h,\delta ,T)=\left\vert \widehat{\sigma
}_{\delta ,T}^{E}-\widehat{\sigma }_{h,\delta ,T}^{u}\right\vert
\end{equation*}%
between the exact ($\widehat{\alpha }_{\delta ,T}^{E},\widehat{\sigma }%
_{\delta ,T}^{E}$) and the approximate ($\widehat{\alpha }_{h,\delta ,T}^{u},%
\widehat{\sigma }_{h,\delta ,T}^{u}$) QML estimators were computed.
Average and standard deviation of these $100$ errors were calculated
for each set of values $h,\delta ,T$ specified above, which are
summarized in Table I. Note as, for fixed $T$, the average of the
errors decreases as $h$ does it. This clearly illustrates the
convergence of the order-$1$ QML estimators to the exact one stated
in Theorem \ref{PLK convergence theorem} when $h$ goes to zero. In
addition, Figure 3 shows the histograms and the confidence limits
for the order-$1$ QML estimators ($\widehat{\alpha }_{h,\delta ,T}^{u},%
\widehat{\sigma }_{h,\delta ,T}^{u}$) and ($\widehat{\alpha }_{\cdot
,\delta ,T},\widehat{\sigma }_{\cdot ,\delta ,T}$) for each set of
values $h,\delta ,T$. By comparing the results of this figure with
the corresponding in the previous ones, the decreasing difference
between the order-$1$ QML estimators ($\widehat{\alpha }_{h,\delta
,T}^{u},\widehat{\sigma }_{h,\delta
,T}^{u}$) and the exact one ($\widehat{\alpha }_{\delta ,T}^{E},\widehat{%
\sigma }_{\delta ,T}^{E}$) is observed as $h$ decreases, which is
consistent with the convergence results of Table I. Similarly, for
$T=10,20$, the
difference between the order-$1$ QML estimators ($\widehat{\alpha }%
_{h,\delta ,T}^{u},\widehat{\sigma }_{h,\delta ,T}^{u}$) and the
adaptive
QML estimators ($\widehat{\alpha }_{\cdot ,\delta ,T},\widehat{\sigma }%
_{\cdot ,\delta ,T}$) decreases when $h$ does it, due to the
negligible difference between the adaptive QML estimators
($\widehat{\alpha }_{\cdot
,\delta ,T},\widehat{\sigma }_{\cdot ,\delta ,T}$) with the exact ones ($%
\widehat{\alpha }_{\delta ,T}^{E},\widehat{\sigma }_{\delta
,T}^{E}$). This illustrates the usefulness of the adaptive strategy
for improving the QML\ parameter estimation for finite samples with
large sampling periods. These findings are more precisely summarized
in Table II, which shows the difference between the averages of the
exact and the approximate QML estimators. Observe the lightly higher
difference between the averages of the exact and the adaptive
estimators for both parameters when $T=30$. The reason is that, for
$t_{k}>21 $, the variance $\mathbf{\Sigma }_{k}$ of the diffusion
(\ref{SE EJ1}) becomes almost indistinguishable of zero. This is so
small that the roundoff errors becomes significant in such a way
that the adaptive strategy under estimates the integration errors
and so over estimates the length of the step sizes. This can be seem
in Figure 4, which shows the average of accepted and fail steps of
the adaptive QML estimators at each $t_{k}\in \{t\}_{T/\delta ,T}$.
Note that, for $t_{k}>21$, the number of accepted steps is lower
that $8.$ Therefore, the adaptive estimator can not be so good as
those on uniform time discretization with $h=1/8,1/32$ in Table II.
Contrary, for $t_{k}<21$, the number of accepted steps is larger
that $32$, and the adaptive estimator performs better than all the
others. Further, note that the results of Table II illustrate the
convergence findings of Theorem \ref{PLK week convergence}.
\newline

\resizebox{\textwidth}{!}{
\begin{tabular}{|c|c|c|c|}
\hline $\delta =0.1$ & $\alpha$ & $\sigma $ & $\rho $ \\ \hline
\begin{tabular}{c}
$h$ \\ \hline $\delta $ \\ \hline $\delta /2$ \\ \hline $\delta /4$
\\ \hline $\delta /8$ \\ \hline $\cdot $\end{tabular} &
\begin{tabular}{ccc}
$T=10$ & $T=20$ & $T=30$ \\ \hline \multicolumn{1}{r}{$0.00048$} &
\multicolumn{1}{r}{$0.00045$} & \multicolumn{1}{r}{$0.00043$} \\
\hline \multicolumn{1}{r}{$0.00014$} & \multicolumn{1}{r}{$0.00013$}
& \multicolumn{1}{r}{$0.00011$} \\ \hline
\multicolumn{1}{r}{$0.00003$} & \multicolumn{1}{r}{$0.00004$} &
\multicolumn{1}{r}{$0.00003$} \\ \hline \multicolumn{1}{r}{$0$} &
\multicolumn{1}{r}{$0.00001$} & \multicolumn{1}{r}{$0$} \\ \hline
\multicolumn{1}{r}{$-0.00008$} & \multicolumn{1}{r}{$0.00004$} &
\multicolumn{1}{r}{$-0.00006$}\end{tabular} &
\begin{tabular}{ccc}
$T=10$ & $T=20$ & $T=30$ \\ \hline \multicolumn{1}{r}{$-0.0405$} &
\multicolumn{1}{r}{$-0.0397$} & \multicolumn{1}{r}{$-0.0396$} \\
\hline \multicolumn{1}{r}{$-0.0110$} & \multicolumn{1}{r}{$-0.0100$}
& \multicolumn{1}{r}{$-0.0100$} \\ \hline
\multicolumn{1}{r}{$-0.0025$} & \multicolumn{1}{r}{$-0.0023$} &
\multicolumn{1}{r}{$-0.0022$} \\ \hline
\multicolumn{1}{r}{$-0.0004$} & \multicolumn{1}{r}{$-0.0004$} &
\multicolumn{1}{r}{$-0.0004$} \\ \hline \multicolumn{1}{r}{$0.0049$}
& \multicolumn{1}{r}{$0.0010$} &
\multicolumn{1}{r}{$0.0092$}\end{tabular} &
\begin{tabular}{ccc}
$T=10$ & $T=20$ & $T=30$ \\ \hline \multicolumn{1}{r}{$-0.00015$} &
$4.4\times 10^{-5}$ & $3.9\times 10^{-5}$
\\ \hline
\multicolumn{1}{r}{$-0.00009$} & $0.9\times 10^{-5}$ & $1.4\times
10^{-5}$
\\ \hline
\multicolumn{1}{r}{$-0.00002$} & $0.8\times 10^{-5}$ & $1.0\times
10^{-5}$
\\ \hline
\multicolumn{1}{r}{$0$} & $0.5\times 10^{-5}$ & $0.6\times 10^{-5}$
\\ \hline \multicolumn{1}{r}{$0.00011$} & $0$ & $0.4\times
10^{-5}$\end{tabular}
\\ \hline
\end{tabular}
}

{\small Table III: Difference between the averages of the exact and
the
approximate QML estimators for the equation (\ref{SE EJ2}). }$h=\delta $%
{\small , for the conventional; }$h=\delta /2,\delta /4,\delta
/8,${\small \
for the order-$1$ on }$\left( \tau \right) _{h,T}^{u}${\small ; and }$%
h=\cdot ,${\small \ for the adaptive order-$1$ on }$\left( \tau
\right) _{\cdot ,T}${\small .\newline }

\begin{tabular}{|l|c|c|c|c|}
\hline $\delta =0.1$ & $h=\delta $ & $h=\delta /2$ & $h=\delta /4$ &
$h=\delta /8$
\\ \hline
\begin{tabular}{ll}
& $T=10$ \\
$\alpha$ & $T=20$ \\
& $T=30$%
\end{tabular}
& \multicolumn{1}{|l|}{%
\begin{tabular}{l}
$6.6\pm 4.9\times 10^{-4}$ \\ \hline $6.8\pm 4.7\times 10^{-4}$ \\
\hline
$6.7\pm 4.6\times 10^{-4}$%
\end{tabular}%
} & \multicolumn{1}{|l|}{%
\begin{tabular}{l}
$1.6\pm 1.2\times 10^{-4}$ \\ \hline $1.4\pm 0.9\times 10^{-4}$ \\
\hline
$1.4\pm 0.9\times 10^{-4}$%
\end{tabular}%
} & \multicolumn{1}{|l|}{%
\begin{tabular}{l}
$3.9\pm 2.9\times 10^{-5}$ \\ \hline $3.4\pm 2.1\times 10^{-5}$ \\
\hline
$3.3\pm 2.1\times 10^{-5}$%
\end{tabular}%
} & \multicolumn{1}{|l|}{%
\begin{tabular}{l}
$9.9\pm 8.4\times 10^{-6}$ \\ \hline $9.3\pm 5.8\times 10^{-6}$ \\
\hline
$8.2\pm 5.9\times 10^{-6}$%
\end{tabular}%
} \\ \hline
\begin{tabular}{ll}
& $T=10$ \\
$\sigma $ & $T=20$ \\
& $T=30$%
\end{tabular}
& \multicolumn{1}{|l|}{%
\begin{tabular}{l}
$5.9\pm 4.4\times 10^{-2}$ \\ \hline $6.0\pm 4.2\times 10^{-2}$ \\
\hline
$5.9\pm 4.1\times 10^{-2}$%
\end{tabular}%
} & \multicolumn{1}{|l|}{%
\begin{tabular}{l}
$1.3\pm 0.8\times 10^{-2}$ \\ \hline $1.2\pm 0.8\times 10^{-2}$ \\
\hline
$1.1\pm 0.7\times 10^{-2}$%
\end{tabular}%
} & \multicolumn{1}{|l|}{%
\begin{tabular}{l}
$2.6\pm 1.8\times 10^{-3}$ \\ \hline $2.4\pm 1.6\times 10^{-3}$ \\
\hline
$2.4\pm 1.6\times 10^{-3}$%
\end{tabular}%
} & \multicolumn{1}{|l|}{%
\begin{tabular}{l}
$5.1\pm 5.1\times 10^{-4}$ \\ \hline $4.9\pm 4.0\times 10^{-4}$ \\
\hline
$4.4\pm 3.8\times 10^{-4}$%
\end{tabular}%
} \\ \hline
\begin{tabular}{ll}
& $T=10$ \\
$\rho $ & $T=20$ \\
& $T=30$%
\end{tabular}
& \multicolumn{1}{|l|}{%
\begin{tabular}{l}
$1.1\pm 1.3\times 10^{-3}$ \\ \hline $1.6\pm 1.5\times 10^{-4}$ \\
\hline
$9.0\pm 7.4\times 10^{-5}$%
\end{tabular}%
} & \multicolumn{1}{|l|}{%
\begin{tabular}{l}
$2.9\pm 4.0\times 10^{-4}$ \\ \hline $5.1\pm 4.6\times 10^{-5}$ \\
\hline
$3.1\pm 2.2\times 10^{-5}$%
\end{tabular}%
} & \multicolumn{1}{|l|}{%
\begin{tabular}{l}
$6.2\pm 8.3\times 10^{-5}$ \\ \hline $1.5\pm 0.8\times 10^{-5}$ \\
\hline
$1.1\pm 0.4\times 10^{-5}$%
\end{tabular}%
} & \multicolumn{1}{|l|}{%
\begin{tabular}{l}
$1.5\pm 1.9\times 10^{-5}$ \\ \hline $6.2\pm 2.4\times 10^{-6}$ \\
\hline
$5.8\pm 1.9\times 10^{-6}$%
\end{tabular}%
} \\ \hline
\end{tabular}

{\small Table IV: Confidence limits for the error between the exact
and the
approximate QML estimators of the equation (\ref{SE EJ2}). }$h=\delta $%
{\small , for the conventional; and }$h=\delta /2,\delta /4,\delta /8,$%
{\small \ for the order-$1$ on }$\left( \tau \right) _{h,T}^{u}${\small .%
\newline
}

Figure 5 shows the histograms and the confidence limits for both, the exact (%
$\widehat{\alpha }_{\delta ,T}^{E}$) and the conventional ($\widehat{\alpha }%
_{\delta ,T}$) QML estimators of $\alpha $ computed from the twelve sets of $%
100$ time series $Z_{\delta ,T}^{i}$ available for the example 2.
Figure 6 shows the same but, for the exact ($\widehat{\sigma
}_{\delta ,T}^{E}$) and the conventional ($\widehat{\sigma }_{\delta
,T}$) QML estimators of $\sigma $, whereas Figure 7 does it for the
estimators $\widehat{\rho }_{\delta ,T}^{E}$ and $\widehat{\rho
}_{\delta ,T}$ of $\rho $. Note that, for this example, the
diffusion parameters $\sigma $ and $\rho $ can not be estimated from
the samples $Z_{\delta ,T}^{i}$ with the largest sampling period
$\delta =1$. From the other data with sampling period $\delta <1$,
the tree parameters can be estimated and, the bias of the exact and
the conventional QML estimators is not so large as in the previous
example. Nevertheless, in this extreme situation of low information
in the data, the order-$1$ QML estimators is able to improve the
accuracy of the parameter estimation when $h$ decreases. This is
shown in Figure 8 for the samples $Z_{\delta ,T}^{i}$ with $\delta
=0.1$ and $T=10,20,30$, and summarized in Table III. The order-$1$
QML estimators ($\widehat{\alpha }_{h,\delta ,T}^{u},\widehat{\sigma
}_{h,\delta ,T}^{u},\widehat{\rho }_{h,\delta ,T}^{u}$) and
($\widehat{\alpha }_{\cdot ,\delta ,T},\widehat{\sigma }_{\cdot
,\delta ,T},\widehat{\rho }_{\cdot ,\delta ,T}$) are again computed
on uniform $\left( \tau \right) _{h,T}^{u}\supset \{t\}_{T/\delta
,T}$ and adaptive $\left( \tau \right) _{\cdot ,T}\supset
\{t\}_{T/\delta ,T}$ time discretizations, respectively, with
$T=10,20,30$, $h=\delta /2,\delta /4,\delta /8$ and
tolerances $rtol_{\mathbf{y}}=rtol_{\mathbf{P}}=10^{-6}$ and $atol_{\mathbf{y%
}}=10^{-9}$, $atol_{\mathbf{P}}=10^{-12}$. The average of accepted
and fail steps of the adaptive QML estimators at each $t_{k}\in
\{t\}_{T/\delta ,T}$ are shown in Figure 4. Note that, the average
of the accepted steps of the adaptive algorithm is not bigger than
$8$ for each $t_{k}\in \{t\}_{T/\delta ,T}$. Therefore, as it is
shown in Table III, the average of the adaptive estimator is not so
good than that of some other estimators on uniform time
discretizations $\left( \tau \right) _{h,T}^{u}$ with $h\geq\delta
/8$. For this example, Table IV gives the confidence limits for the
error between the exact and the order-$1$ QML estimators for
different values of $h$. Note that, Table III and IV illustrate the
convergence results of Theorems \ref{PLK week convergence} and
\ref{PLK convergence theorem}, respectively.

\subsection{Simulations with two-dimensional equations}

For the examples 3 and 4, $100$ realizations of the equation were
similarly computed by means of the Local Linearization and the Euler
scheme, respectively. For each example, the realizations where
computed over the thin time partition
$\{t_{0}+10^{-4}n:n=0,..,30\times 10^{4}\}$ for guarantee a precise
simulation of the stochastic solutions on the time interval
$[t_{0},t_{0}+30]$. Two subsamples of each realization at the time
instants $\{t\}_{M,T}=\{t_{k}=t_{0}+kT/M:$ $k=0,..,M-1\}$ were taken
as observation $Z$ of $\mathbf{x}$ for making inference with $T=30$
and two values of $M$. In particular, $M=30,300$ were used, which
correspond to the
sampling periods $\delta =1,0.1$. In this way, two sets of $100$ time series $%
Z_{\delta ,T}^{i}=\{\mathbf{z}_{k}^{i}:k=0,..,M-1,M=T/\delta \}$, with $%
i=1,..,100$, of $M$ observations $\mathbf{z}_{k}^{i}$ each one were
available for each example with the two values of $(\delta ,T)$
mentioned above.{\small \newline }

\begin{tabular}{|c|c|c|}
\hline $T=30$ & $\alpha$ & $\sigma $ \\ \hline
\begin{tabular}{c}
$h$ \\ \hline $\delta $ \\ \hline $\delta /16$ \\ \hline $\delta
/64$ \\ \hline
$\cdot $%
\end{tabular}
&
\begin{tabular}{cc}
$\delta =1$ & $\delta =0.1$ \\ \hline \multicolumn{1}{r}{$-0.2500$}
& \multicolumn{1}{r}{$-0.1428$} \\ \hline
\multicolumn{1}{r}{$-0.0965$} & \multicolumn{1}{r}{$-0.0044$} \\
\hline \multicolumn{1}{r}{$-0.0333$} & \multicolumn{1}{r}{$0.0029$}
\\ \hline
\multicolumn{1}{r}{$-0.0096$} & \multicolumn{1}{r}{$0.0068$}%
\end{tabular}
&
\begin{tabular}{cc}
$\delta =1$ & $\delta =0.1$ \\ \hline \multicolumn{1}{r}{$-0.7328$}
& \multicolumn{1}{r}{$-0.0052$} \\ \hline
\multicolumn{1}{r}{$-0.0893$} & \multicolumn{1}{r}{$0.0012$} \\
\hline \multicolumn{1}{r}{$-0.0757$} & \multicolumn{1}{r}{$0.0013$}
\\ \hline
\multicolumn{1}{r}{$-0.0739$} & \multicolumn{1}{r}{$0.0013$}%
\end{tabular}
\\ \hline
\end{tabular}

{\small Table V: Bias of the approximate QML estimators for the equation (%
\ref{SEa EJ3})-(\ref{SEb EJ3}). }$h=\delta ${\small , for the conventional; }%
$h=\delta /16,\delta /64,${\small \ for the order-$1$ on }$\left(
\tau \right) _{h,T}^{u}${\small ; and }$h=\cdot ,${\small \ for the
adaptive order-$1$ on }$\left( \tau \right) _{\cdot ,T}${\small
.\newline }

\begin{tabular}{|c|c|c|}
\hline $T=30$ & $\alpha$ & $\sigma $ \\ \hline
\begin{tabular}{c}
$h$ \\ \hline $\delta $ \\ \hline $\delta /8$ \\ \hline $\delta /32$
\\ \hline
$\cdot $%
\end{tabular}
&
\begin{tabular}{cc}
$\delta =1$ & $\delta =0.1$ \\ \hline \multicolumn{1}{r}{$-0.8000$}
& \multicolumn{1}{r}{$-0.2507$} \\ \hline
\multicolumn{1}{r}{$-0.4234$} & \multicolumn{1}{r}{$-0.0481$} \\
\hline \multicolumn{1}{r}{$-0.2210$} & \multicolumn{1}{r}{$-0.0219$}
\\ \hline
\multicolumn{1}{r}{$-0.1611$} & \multicolumn{1}{r}{$-0.0046$}%
\end{tabular}
&
\begin{tabular}{cc}
$\delta =1$ & $\delta =0.1$ \\ \hline \multicolumn{1}{r}{$-0.8000$}
& \multicolumn{1}{r}{$-0.3748$} \\ \hline
\multicolumn{1}{r}{$-0.2451$} & \multicolumn{1}{r}{$0.0005$} \\
\hline \multicolumn{1}{r}{$-0.1910$} & \multicolumn{1}{r}{$0.0015$}
\\ \hline
\multicolumn{1}{r}{$-0.1898$} & \multicolumn{1}{r}{$0.0001$}%
\end{tabular}
\\ \hline
\end{tabular}

{\small Table VI: Bias of the approximate QML estimators for the equation (%
\ref{SEa EJ4})-(\ref{SEb EJ4}). }$h=\delta ${\small , for the conventional; }%
$h=\delta /8,\delta /32,${\small \ for the order-$1$ on }$\left(
\tau \right) _{h,T}^{u}${\small ; and }$h=\cdot ,${\small \ for the
adaptive order-$1$ on }$\left( \tau \right) _{\cdot ,T}${\small
.}\newline

For both examples, the order-$1$ QML estimators ($\widehat{\alpha }%
_{h,\delta ,T}^{u},\widehat{\sigma }_{h,\delta ,T}^{u}$) and ($\widehat{%
\alpha }_{\cdot ,\delta ,T},\widehat{\sigma }_{\cdot ,\delta ,T}$)
on uniform $\left( \tau \right) _{h,T}^{u}\supset \{t\}_{T/\delta
,T}$ and adaptive $\left( \tau \right) _{\cdot ,T}\supset
\{t\}_{T/\delta ,T}$ time discretizations, respectively, were
computed from the two sets of $100$ data $Z_{\delta ,T}^{i}$ with
$T=30$ and $\delta =1,0.1$. The values of $h$ were set as $h=\delta
,\delta /16,\delta /64$ for the example 3, and as $h=\delta ,\delta
/8,\delta /32$ for the example 4. The tolerances for the adaptive
estimators were set as in the first example. Figures 9 and 11
show the histograms and the confidence limits for the estimators ($\widehat{%
\alpha }_{h,\delta ,T}^{u},\widehat{\sigma }_{h,\delta ,T}^{u}$) and ($%
\widehat{\alpha }_{\cdot ,\delta ,T},\widehat{\sigma }_{\cdot
,\delta ,T}$) corresponding to each example. For the two examples,
the difference between
the order-$1$ QML estimator ($\widehat{\alpha }_{h,\delta ,T}^{u},\widehat{%
\sigma }_{h,\delta ,T}^{u}$) and the adaptive one ($\widehat{\alpha
}_{\cdot ,\delta ,T},\widehat{\sigma }_{\cdot ,\delta ,T}$)
decreases when $h$ does it. This is, according Theorem \ref{PLK
convergence theorem}, an expected result by assuming that the
difference between the adaptive and the exact QLM estimators is
negligible for $\left( \tau \right) _{\cdot ,T}$ thin enough. In
addition, Table V and VI show the bias of the approximate QML
estimators for these examples. Observe as the adaptive
($\widehat{\alpha }_{\cdot ,\delta ,T},\widehat{\sigma }_{\cdot
,\delta ,T}$) and the order-$1$ QML estimator ($\widehat{\alpha
}_{h,\delta ,T}^{u},\widehat{\sigma }_{h,\delta ,T}^{u}$) with
$h<\delta $ provide much less biased estimation of the
parameters $(\alpha ,\sigma )$ than the conventional QML estimator ($%
\widehat{\alpha }_{\delta ,\delta ,T}^{u},\widehat{\sigma }_{\delta
,\delta ,T}^{u}$), which is in fact unable to identify the
parameters of the examples. Clearly, this illustrates the usefulness
of the order-$1$ QML estimator and its adaptive implementation.
However, as it is shown in Table
V for $\delta =0.1$, no always the adaptive estimator ($\widehat{\alpha }%
_{\cdot ,\delta ,T},\widehat{\sigma }_{\cdot ,\delta ,T}$) is less
unbiased
than the order-$1$ QML estimator ($\widehat{\alpha }_{h,\delta ,T}^{u},%
\widehat{\sigma }_{h,\delta ,T}^{u}$) for some $h<\delta $. This can
happen for one of following reasons: 1) the bias of the exact QML\
estimator when the adaptive estimator is close enough to it, or 2)
an insufficient number of accepted steps of the adaptive estimator
for a given tolerance. In our
case, since ($\widehat{\alpha }_{h,\delta ,T}^{u},\widehat{\sigma }%
_{h,\delta ,T}^{u}$) converges to ($\widehat{\alpha }_{\cdot ,\delta ,T},%
\widehat{\sigma }_{\cdot ,\delta ,T}$) as $h$ decreases (Figure 9 with $%
\delta =0.1$) and the average of accepted steps of the adaptive
estimators is acceptable (Figure 10 with $\delta =0.1$), the first
explanation is more suitable. Figures 10 and 12 show the average of
accepted and fail steps of the adaptive QML estimators at each
$t_{k}\in \{t\}_{T/\delta ,T}$ for each example. Note how the
average of accepted steps corresponding to the estimators from
samples with $\delta =0.1$ is ten time lower than that of the
estimators from samples with $\delta =1$, which is an expected
result as well.

\section{Conclusions}

A modification of the conventional approximations to the
quasi-maximum likelihood (QML) method was introduced for the
parameter estimation of diffusion processes given a time series of
complete observations. This is based on a recursive approximation to
the first two conditional moments of the diffusion process through
discrete-time schemes. For finite samples, the convergence of the
modified QML estimators to the exact one was proved when the error
between the discrete-time approximation and the diffusion process
decreases. It was also demonstrated that, for an increasing number
of observations, they are asymptotically normal distributed and
their bias decreases when the above mentioned error does it. As
particular instance, the order-$\beta $ QML estimators based on
Local Linearization schemes were proposed. For them, practical
algorithms were also provided and their performance in simulation
illustrated with various examples. Simulations shown that: 1) with
thin time discretizations between observations, the order-$1$ QML
estimator provides satisfactory approximations to the exact QML
estimator; 2) the convergence of the order-$1$ QML estimator to the
exact one when the maximum stepsize of the time discretization
between observations decreases; 3) with respect to the conventional
QML estimator, the order-$1$ QML estimator gives much better
approximation to the exact QML estimator, and has less bias and
higher efficiency; 4) with an adequate tolerance, the adaptive
order-$1$ QML estimator gives an automatic, suitable and
computational efficient approximation to the exact QML estimator;
and 5) the effectiveness of the order-$1$ QML estimator for the
identification of SDEs from a reduced number of complete
observations distant in time. Further note that new estimators can
also be easily applied to a variety of practical problems with
sequential random measurements or with multiple missing data.
\newline

\textbf{Acknowledgement. } The numerical simulations of this paper
were concluded on July 2012 within the framework of the
Associateship Scheme of the Abdus Salam International Centre for
Theoretical Physics (ICTP), Trieste, Italy. The author thanks to the
ICTP for the partial support to this work.

\newpage

\textbf{References}
\newline

\textbf{Bibby B.M. and Sorensen M.} (1996) Estimation for discrete
observed diffusion processes: A review, Theory Stoch. Processes, 2,
49-56.

\textbf{Bollerslev T. and Wooldridge J.M.} (1992) Quasi-maximun
likelihood estimation and inference in dynamic models with
time-varying covariances. Econom. Rev., 11, 143-172.

\textbf{Carbonell F., Jimenez J.C. and Biscay R.J.} (2006) Weak
local linear discretizations for stochastic differential equations:
convergence and numerical schemes, J. Comput. Appl. Math., 197,
578-596.

\textbf{Clement E.} (1995) Bias correction for the estimation of
discretized diffusion processes from an approximate likelihood
function, Theory Prob. Appl., 42, 283-288.

\textbf{Durham G.B. and Gallant A.R.} (2002) Numerical techniques
for maximum likelihood estimation of continuous-time diffusion
processes, J. Business Econom. Statist., 20, 297-316.

\textbf{Florens-Zmirou D.}(1989) Approximate discrete-time schemes
for statistics of diffusion processes, Statistics, 20, 547-557.

\textbf{Gitterman M.}{\normalsize , The noisy oscillator, World
Scientific, 2005. }

\textbf{Huang X.} (2011) Quasi-maximun likelihood estimation of
discretely observed diffusions. Econom. J., 14, 241-256.

\textbf{Hurn A.S., Jeisman I.J. and Lindsay, K.A.} (2007) Seeing the
wood for the trees: a critical evaluation of methods to estimate the
parameters of stochastic differential equations. J. Financial
Econometrics, 5, 390-455.

\textbf{Jimenez J.C.} (2012a) Simplified formulas for the mean and
variance of linear stochastic differential equations.
http://arxiv.org/abs/1207.5067. Submitted.

\textbf{Jimenez J.C.} (2012b) Approximate linear minimum variance
filters for continuous-discrete state space models: convergence and
practical algorithms. http://arxiv.org/abs/1207.6023. Submitted.

\textbf{Jimenez J.C. and Biscay R.} (2002) Approximation of
continuous time stochastic processes by the Local Linearization
method revisited. Stochast. Anal. \& Appl., 20, 105-121.

\textbf{Jimenez J.C., Biscay R. and Ozaki T.} (2006) Inference
methods for discretely observed continuous-time stochastic
volatility models: A commented overview, Asia-Pacific Financial
Markets, 12, 109-141.

\textbf{Jimenez J.C. and de la Cruz H.} (2012) Convergence rate of
strong Local Linearization schemes for stochastic differential
equations with additive noise, BIT, 52, 357-382.

\textbf{Jimenez J.C. and Ozaki T.} (2002) Linear estimation of
continuous-discrete linear state space models with multiplicative
noise, Systems \& Control Letters, 47, 91-101.

\textbf{Jimenez J.C. and Ozaki T.} (2003) Local Linearization
filters for nonlinear continuous-discrete state space models with
multiplicative noise. Int. J. Control, 76, 1159-1170.

\textbf{Jimenez J.C., Shoji I. and Ozaki T.} (1999) Simulation of
stochastic differential equations through the Local Linearization
method. A comparative study, J. Statist. Physics, 94, 587-602.

\textbf{Kessler M.} \textbf{\ }(1997). Estimation of an ergodic
diffusion from discrete observations, Scand. J. Statist., 24,
211-229.

\textbf{Kloeden P.E. and Platen E}. (1999) Numerical Solution of
Stochastic Differential Equations, Springer-Verlag, Berlin, Third
Edition.

\textbf{Ljung L. and Caines P.E.} (1979) Asymptotic normality of
prediction error estimators for approximate system models,
Stochastics, 3, 29-46.

\textbf{Moler C. and Van Loan C.} (2003) Nineteen dubious ways to
compute the exponential of a matrix, SIAM Review, 45, 3-49.

\textbf{Nielsen J.N., Madsen H. and Young, P. C.} (2000) Parameter
estimation in stochastic differential equations: an overview. Annual
Review of Control, 24, 83-94.

\textbf{Oppenheim A.V. and Schafer R.W.} (2010) Discrete-Time Signal
Processing, Prentice Hall, Third Edition.

\textbf{Ozaki T.} (1985) Statistical identification of storage model
with application to stochastic hydrology, Water Resourse Bulletin,
21, 663-675.

\textbf{Ozaki T.} (1992) A bridge between nonlinear time series
models and nonlinear stochastic dynamical systems: a local
linearization approach, Statistica Sinica, 2, 113-135.

\textbf{Prakasa-Rao B.L.S.} (1983) Asymptotic theory for nonlinear
least squares estimator for diffusion processes, Mathematische
Operationsforschang Statistik Serie Statistik, 14, 195-209.

\textbf{Prakasa-Rao B.L.S.} (1999) Statistical inference for
diffusion type processes. Oxford University Press.

\textbf{Shoji I. and Ozaki T.} (1997) Comparative study of
estimation methods for continuous time stochastic processes, J. Time
Ser. Anal., 18, 485--506.

\textbf{Shoji I. and Ozaki T}. (1998) A statistical method of
estimation and simulation for systems of stochastic differential
equations, Biometrika, 85, 240--243.

\textbf{Singer H}. (2002) Parameter estimation of nonlinear
stochastic differential equations: Simulated maximum likelihood
versus extended Kalman filter and Ito-Taylor expansion, J. Comput.
Graphical Statist., 11, 972-995.

\textbf{Wooldridge J.M.\ }(1994) Estimation and inference for
dependent processes. In Engle R.F. and McFadden D. (Eds.) Handbook
of Econometrics, 4, 2639-2738.

\textbf{Yoshida N.} (1992) Estimation for diffusion processes from
discrete observation. J. Multivariate Anal., 41, 220--242.

\newpage

%% Figure 1
\begin{figure}
\centering
 $\begin{array}{c}
  \includegraphics[width=5in]{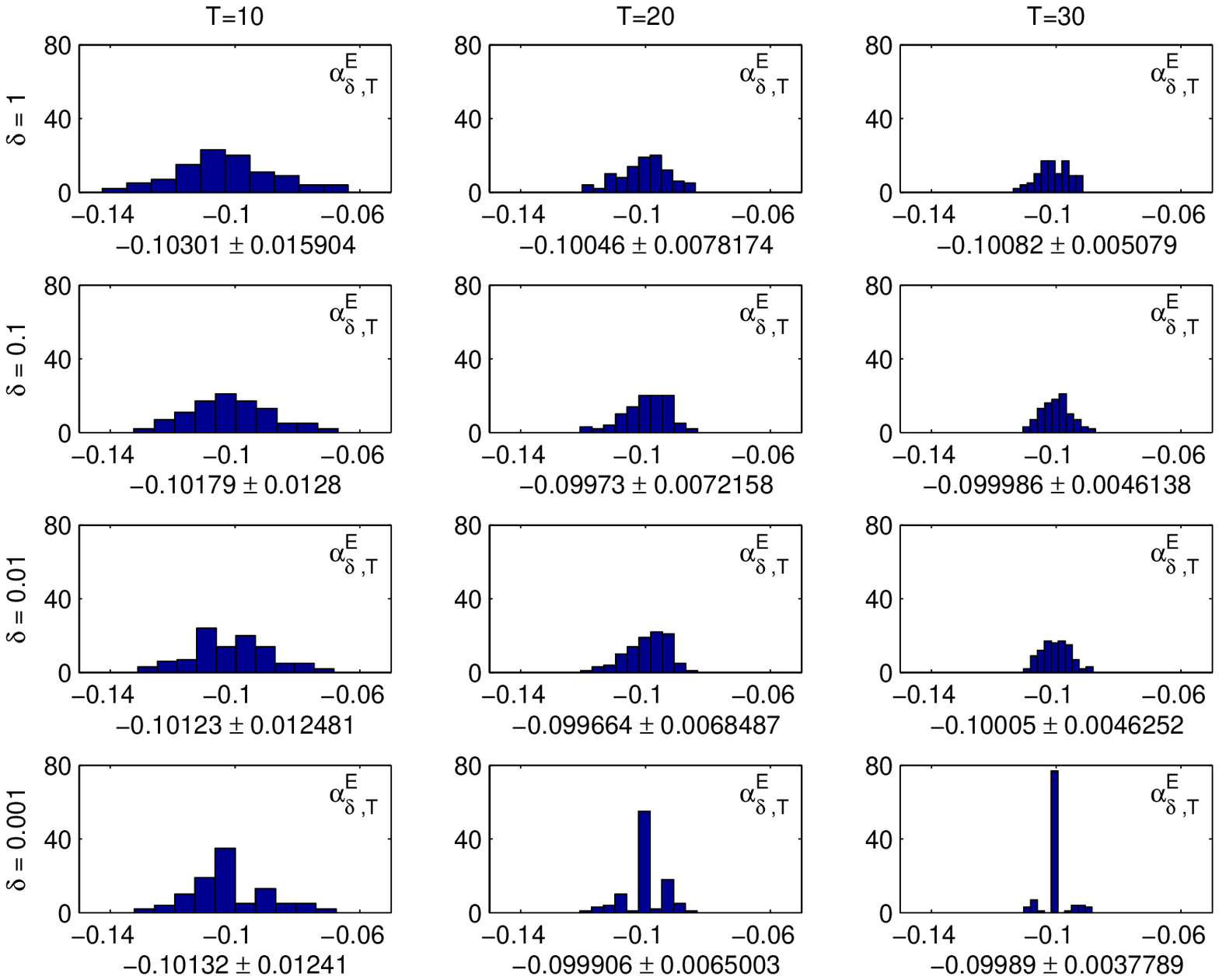} \\
  \includegraphics[width=5in]{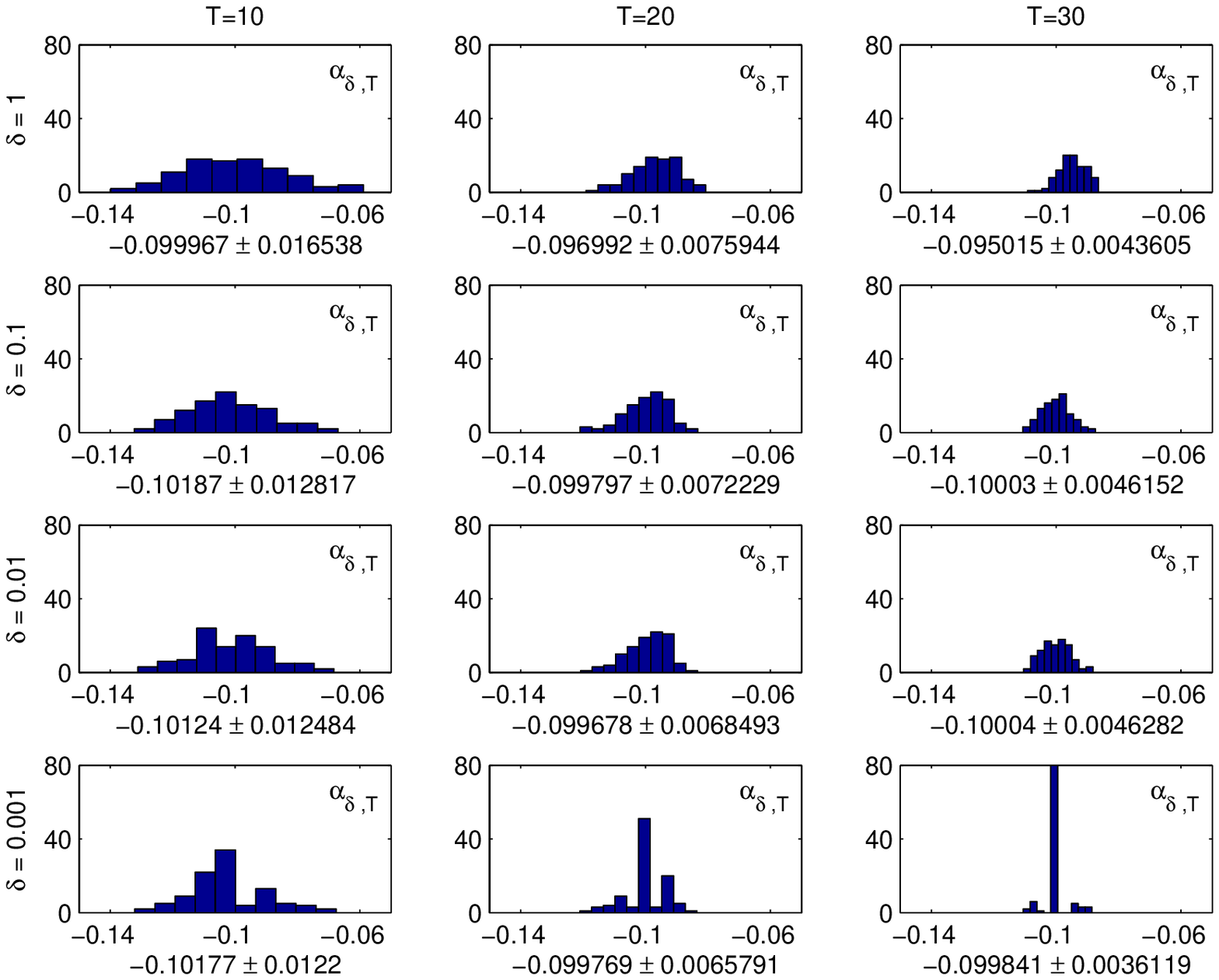}
\end{array}$
\caption{Histograms and confidence limits for the exact ($\widehat{\protect\alpha }_{\protect\delta %
,T}^{E}$) and the conventional ($\widehat{\protect\alpha }_{\protect\delta %
,T}$) QML estimators of $\protect\alpha $ computed from the Example
1 data with sampling period $\protect\delta $ and time interval of
length $T$.}
\end{figure}

%% Figure 2
\begin{figure}
\centering
 $\begin{array}{c}
  \includegraphics[width=5in]{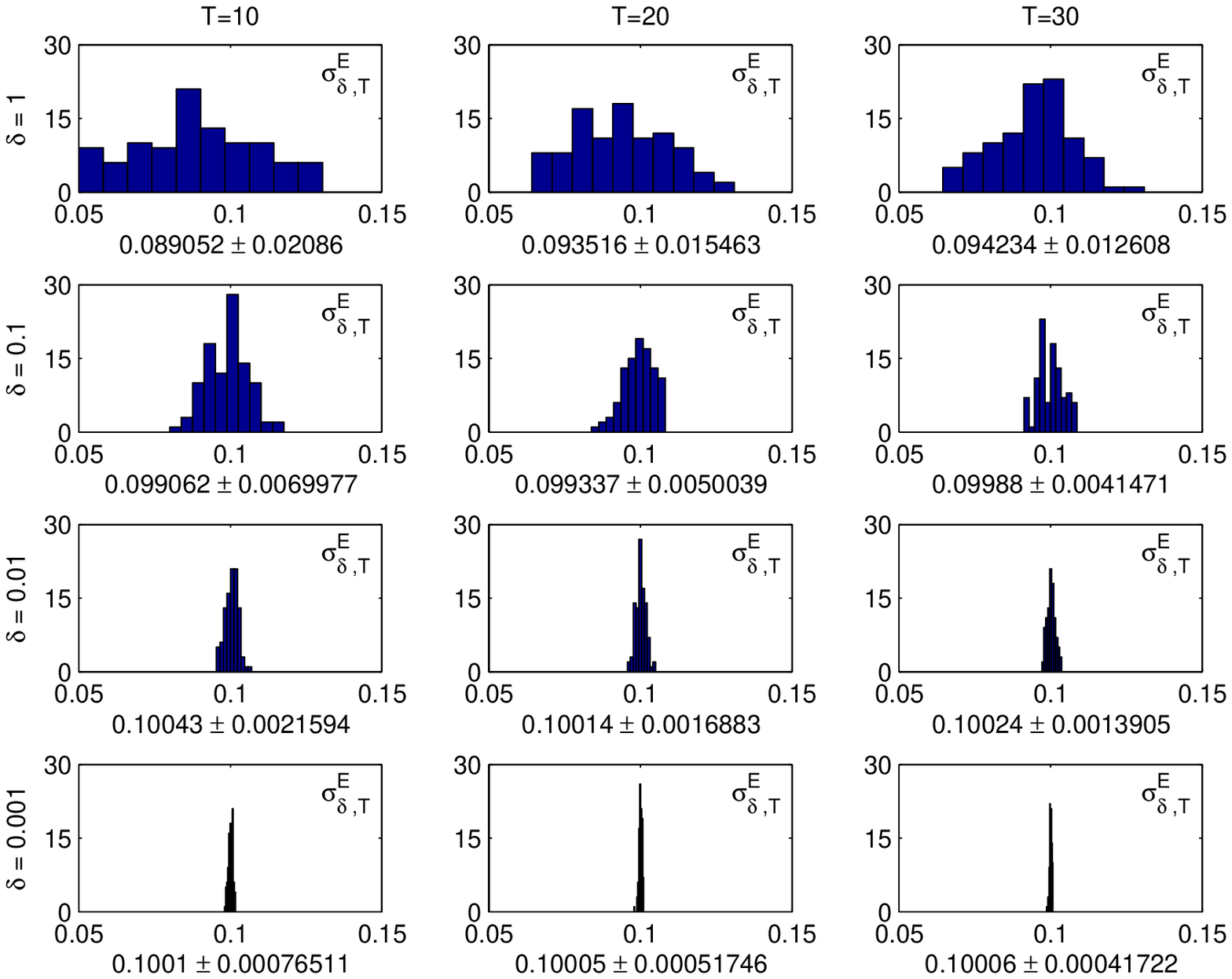} \\
  \includegraphics[width=5in]{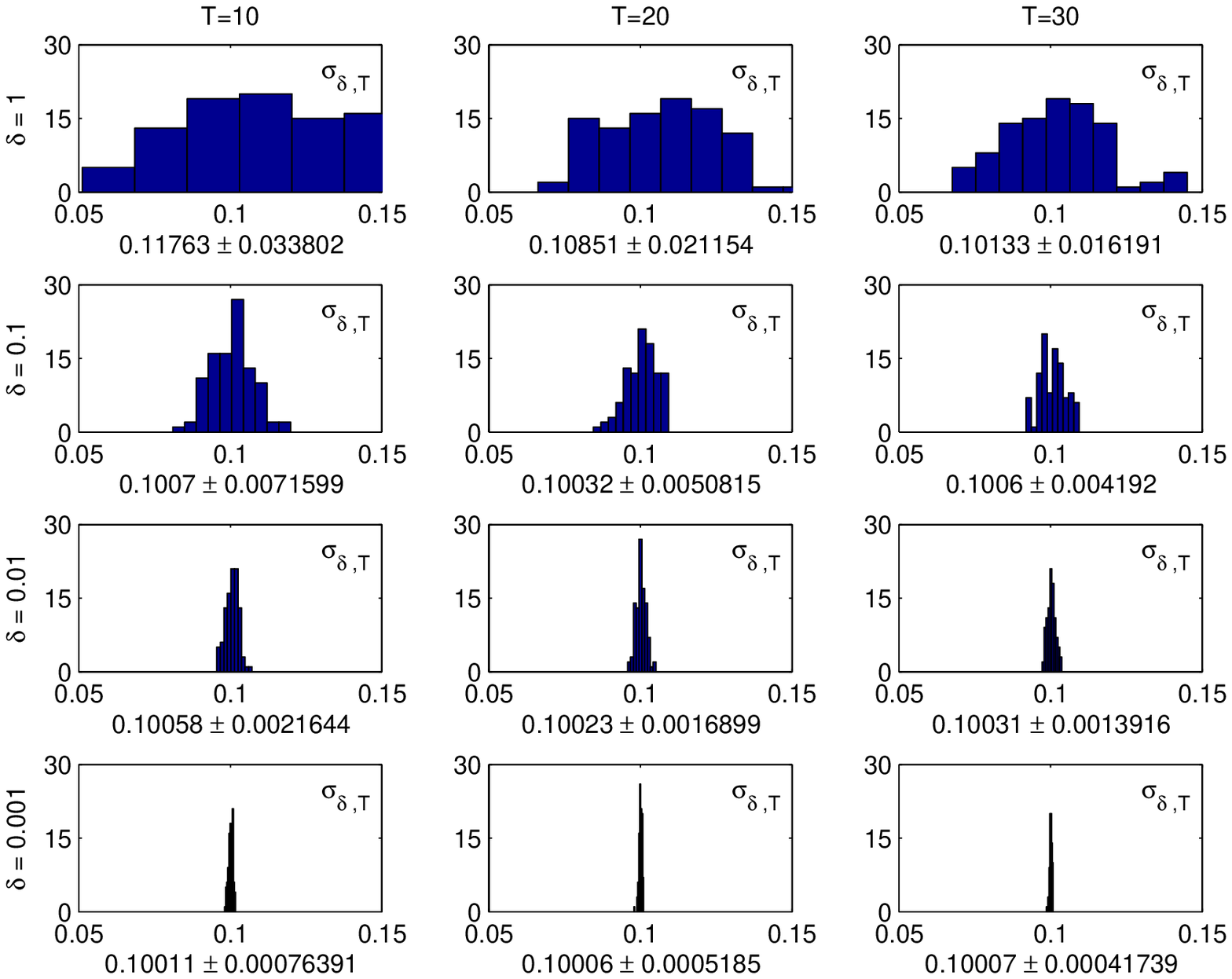}
\end{array}$
\caption{Histograms and confidence limits for the exact ($\widehat{\protect\sigma }_{\protect\delta %
,T}^{E}$) and the conventional ($\widehat{\protect\sigma }_{\protect\delta %
,T}$) QML estimators of $\protect\sigma $ computed from the Example
1 data with sampling period $\protect\delta $ and time interval of
length $T$.}
\end{figure}

%% Figure 3
\begin{figure}
\centering
 $\begin{array}{c}
  \includegraphics[width=5in]{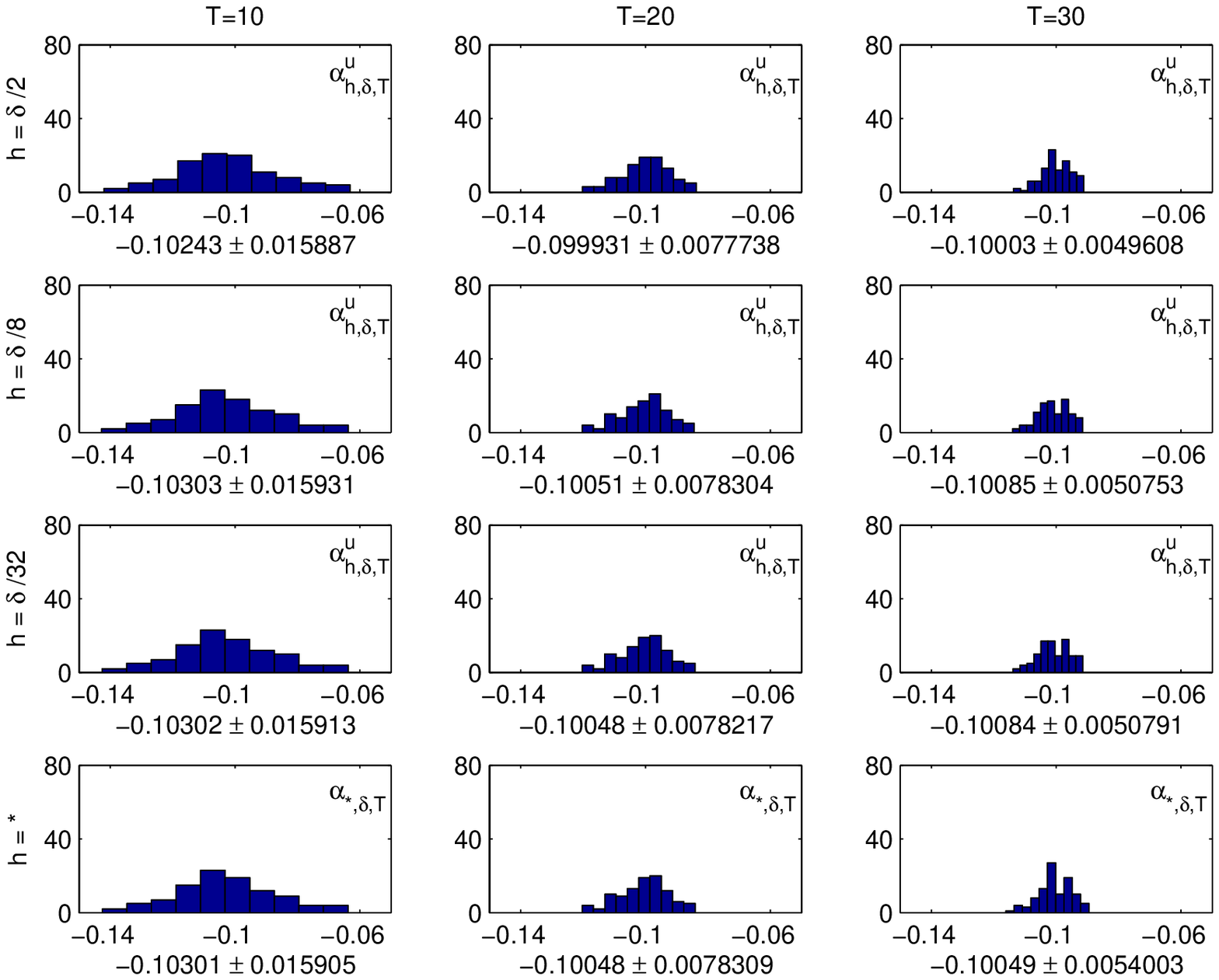} \\
  \includegraphics[width=5in]{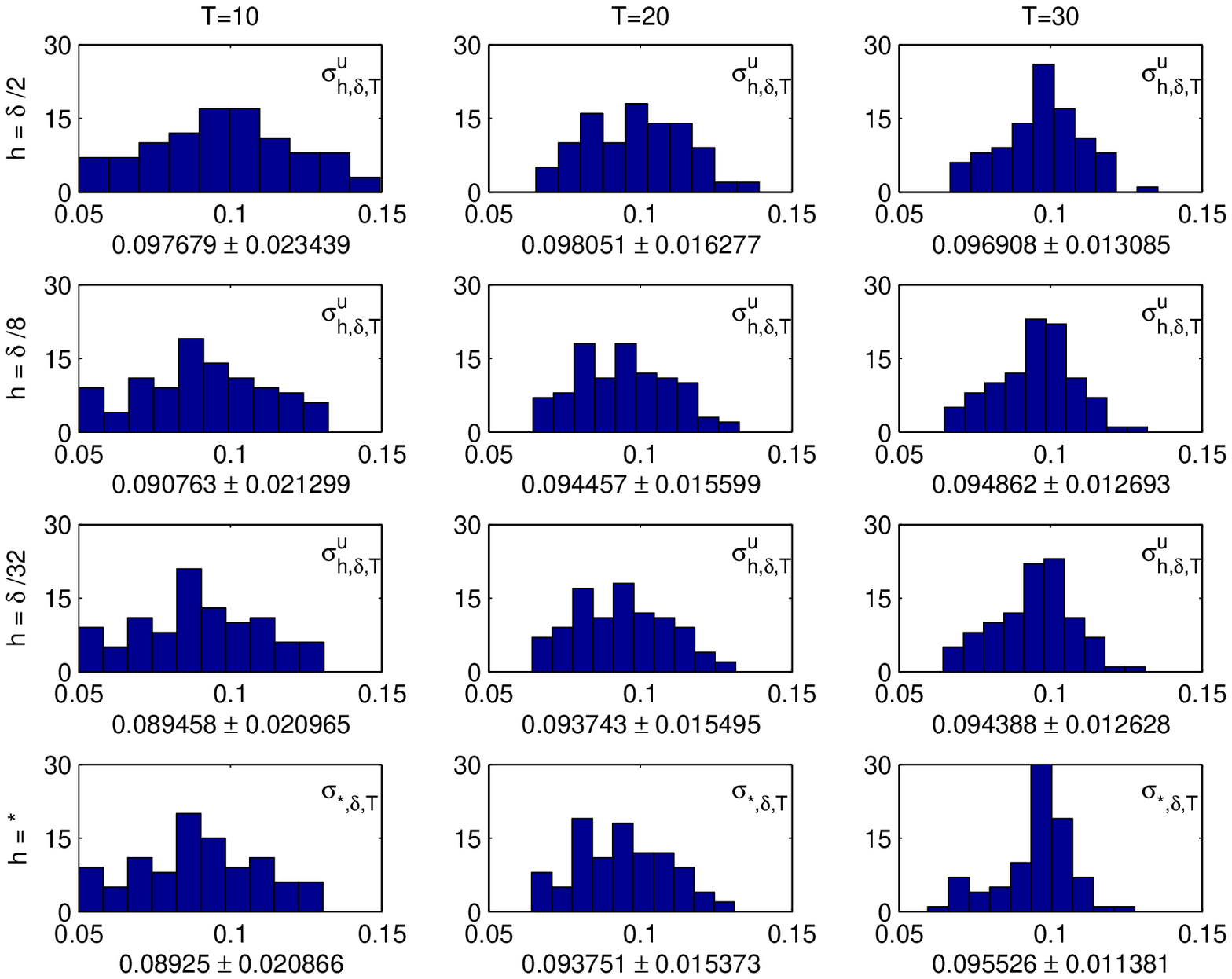}
\end{array}$
\caption{Histograms and confidence limits for the oder-1 QML estimators of $\protect\alpha $ and $%
\protect\sigma $\ computed on uniform $\left( \protect\tau \right)
_{h,T}^{u}$ and adaptive $\left( \protect\tau \right) _{\cdot ,T}$
time discretizations from the Example 1 data with sampling period
$\protect\delta =1$ and time interval of length $T$.}
\end{figure}

%% Figure 4
\begin{figure}
\centering
 $\begin{array}{c}
  \includegraphics[width=5in]{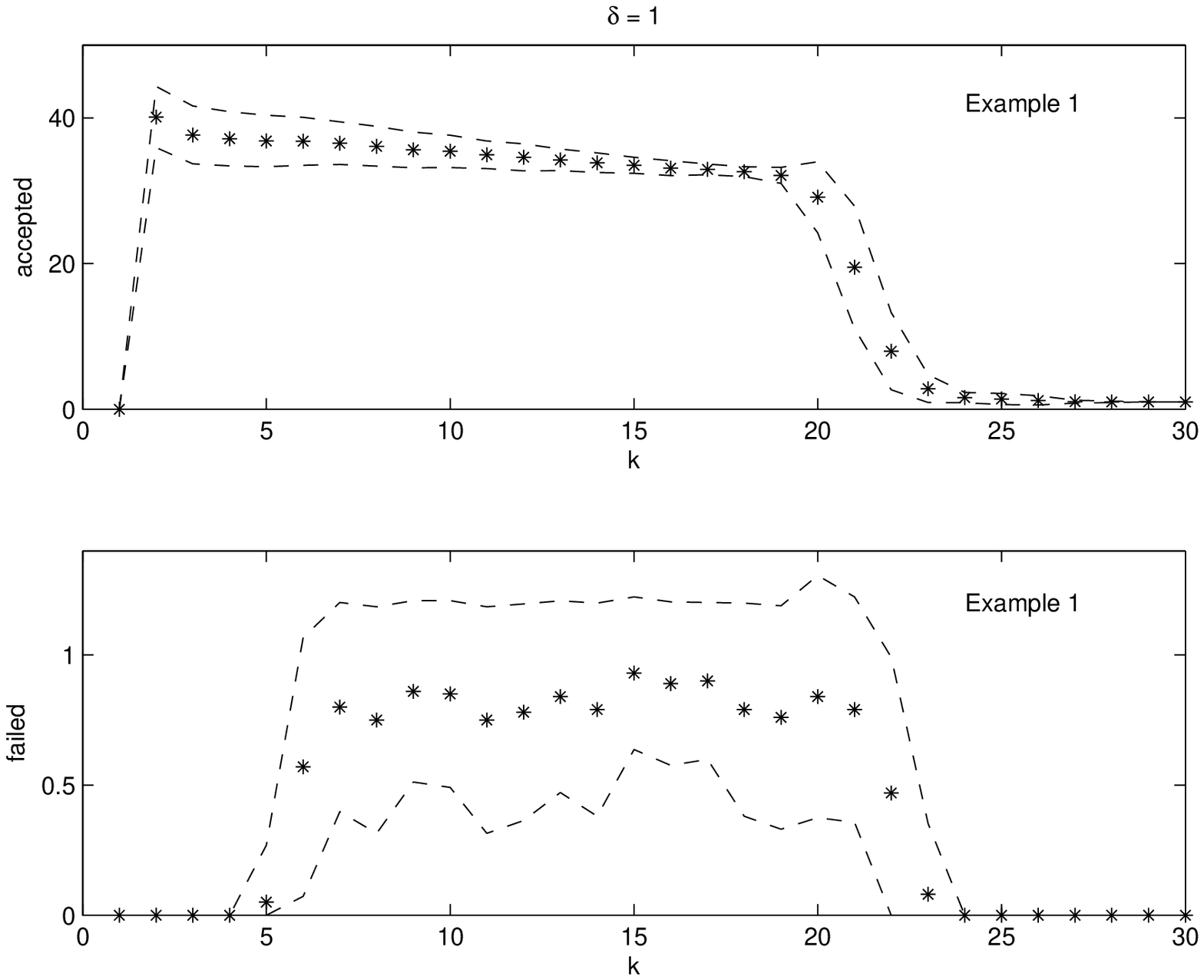} \\
  \includegraphics[width=5in]{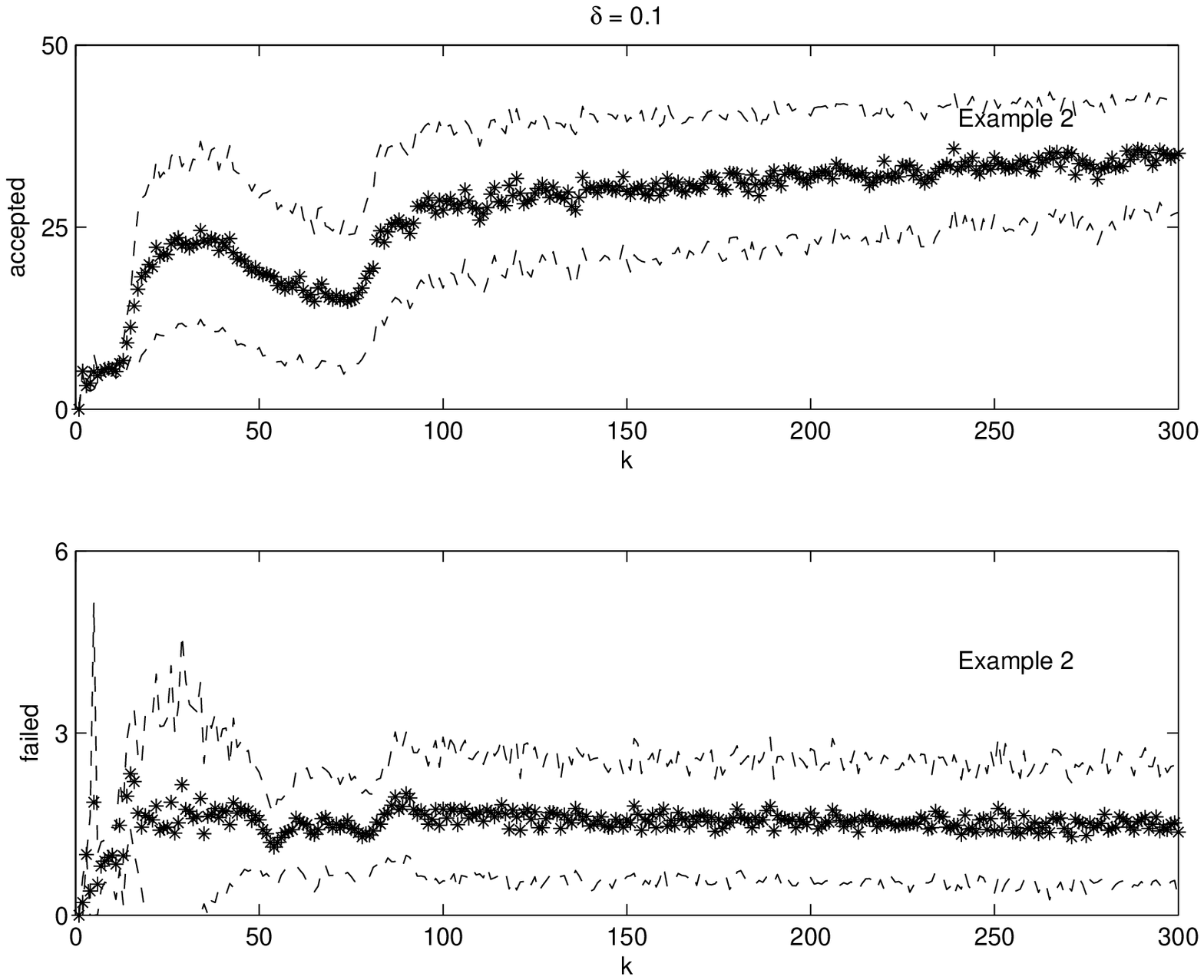}
\end{array}$
\caption{Average (*) and 90\% confidence limits (-) of accepted and
failed steps of the adaptive QML estimator at each $t_{k}\in
\{t\}_{N}$ in the Examples 1 and 2.}
\end{figure}

%% Figure 5
\begin{figure}
\centering
 $\begin{array}{c}
  \includegraphics[width=5in]{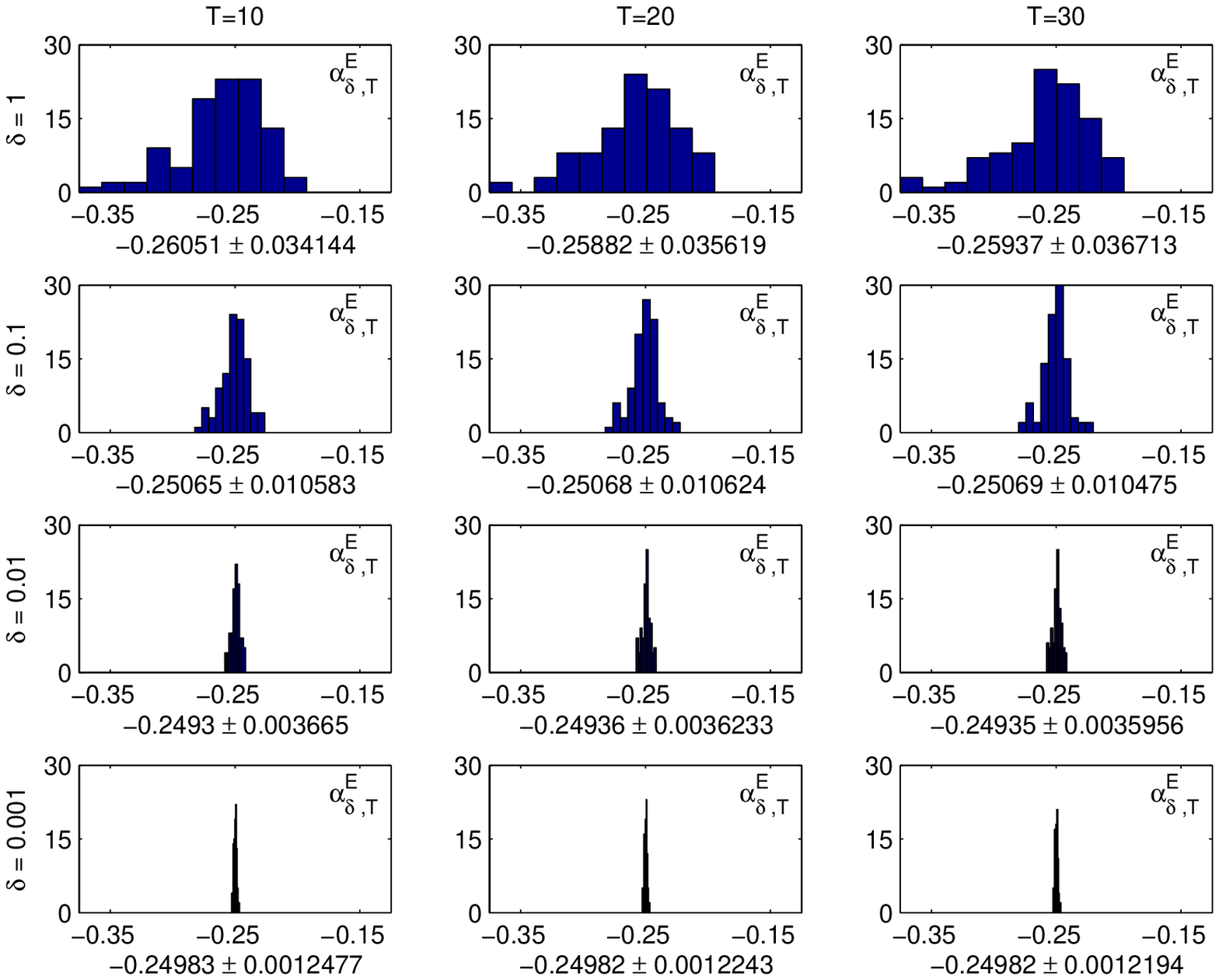} \\
  \includegraphics[width=5in]{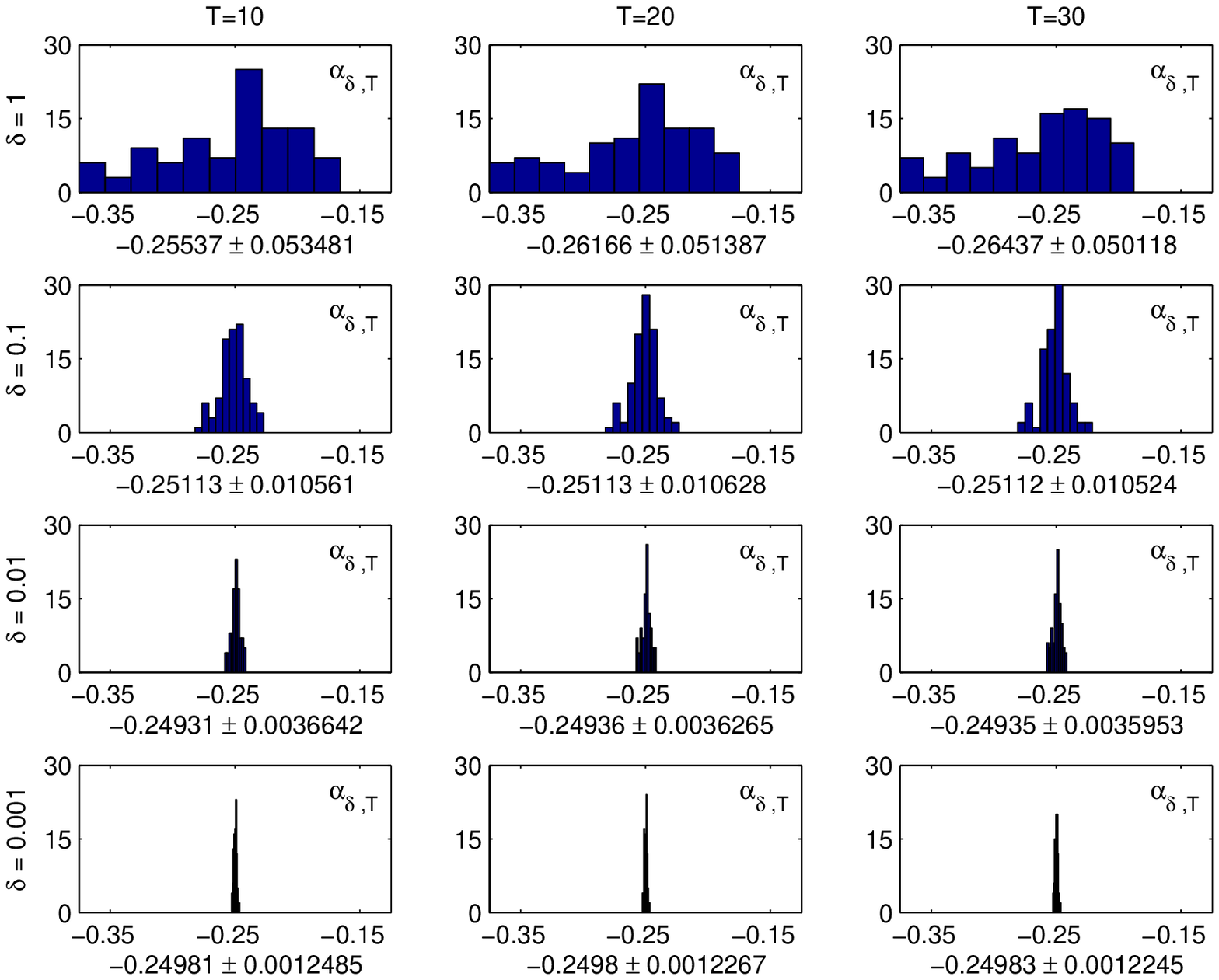}
\end{array}$
\caption{Histograms and confidence limits for the exact ($\widehat{\protect\alpha }_{\protect\delta %
,T}^{E}$) and the conventional ($\widehat{\protect\alpha }_{\protect\delta %
,T}$) QML estimators of $\protect\alpha $ computed from the Example
2 data with sampling period $\protect\delta $ and time interval of
length $T$.}
\end{figure}

%% Figure 6
\begin{figure}
\centering
 $\begin{array}{c}
  \includegraphics[width=5in]{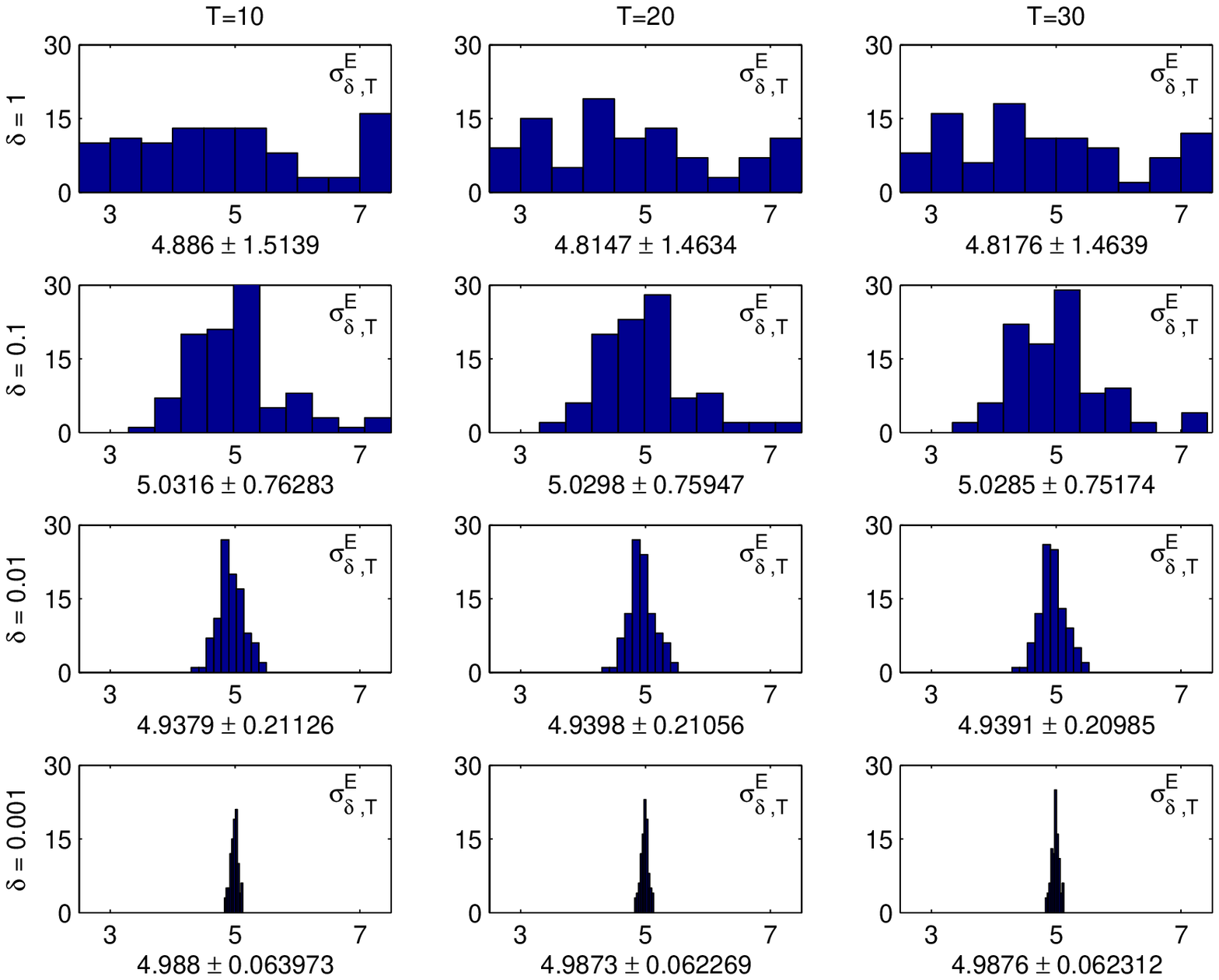} \\
  \includegraphics[width=5in]{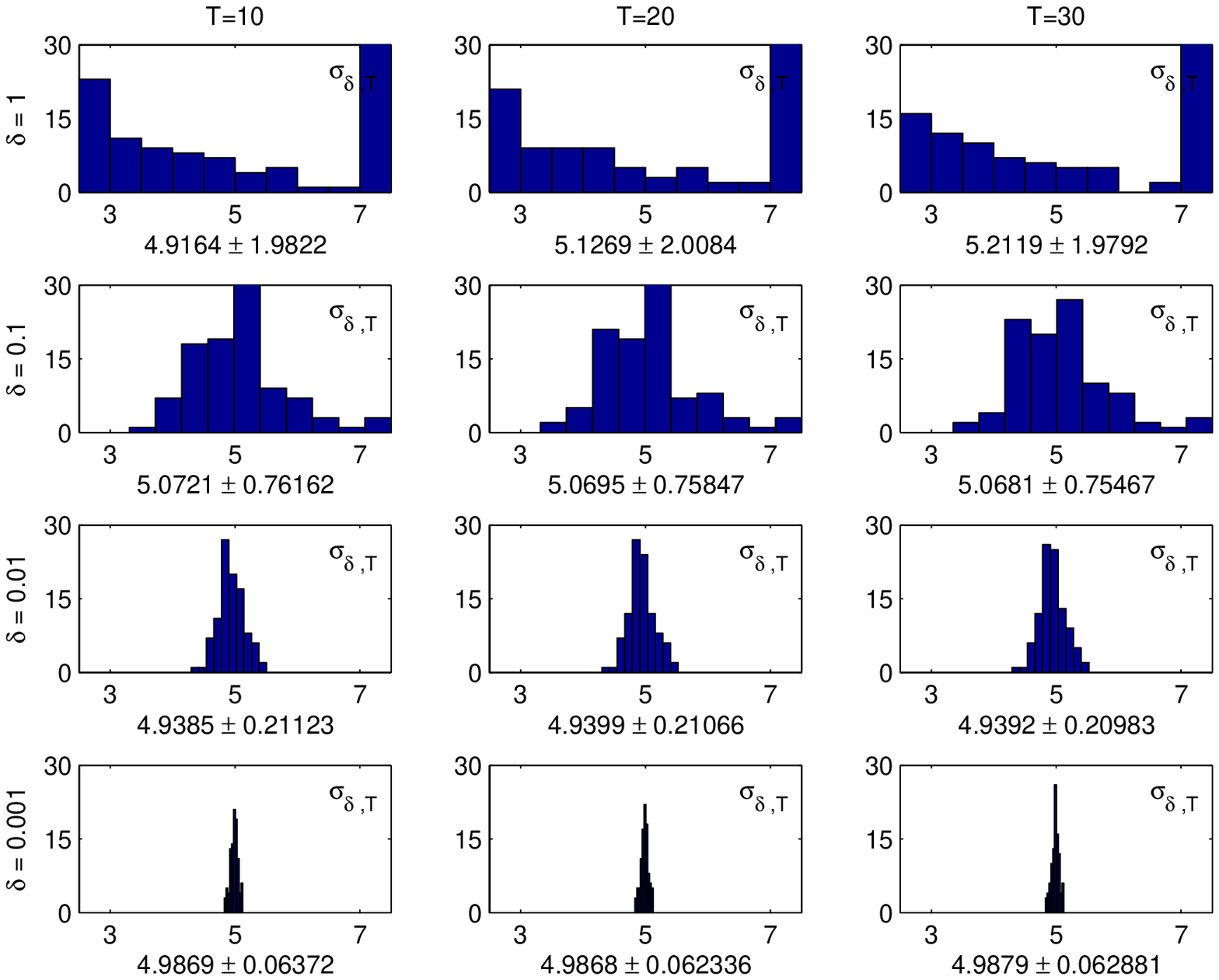}
\end{array}$
\caption{Histograms and confidence limits for the exact ($\widehat{\protect\sigma }_{\protect\delta %
,T}^{E}$) and the conventional ($\widehat{\protect\sigma }_{\protect\delta %
,T}$) QML estimators of $\protect\sigma $ computed from the Example
2 data with sampling period $\protect\delta $ and time interval of
length $T$.}
\end{figure}

%% Figure 7
\begin{figure}
\centering
 $\begin{array}{c}
  \includegraphics[width=5in]{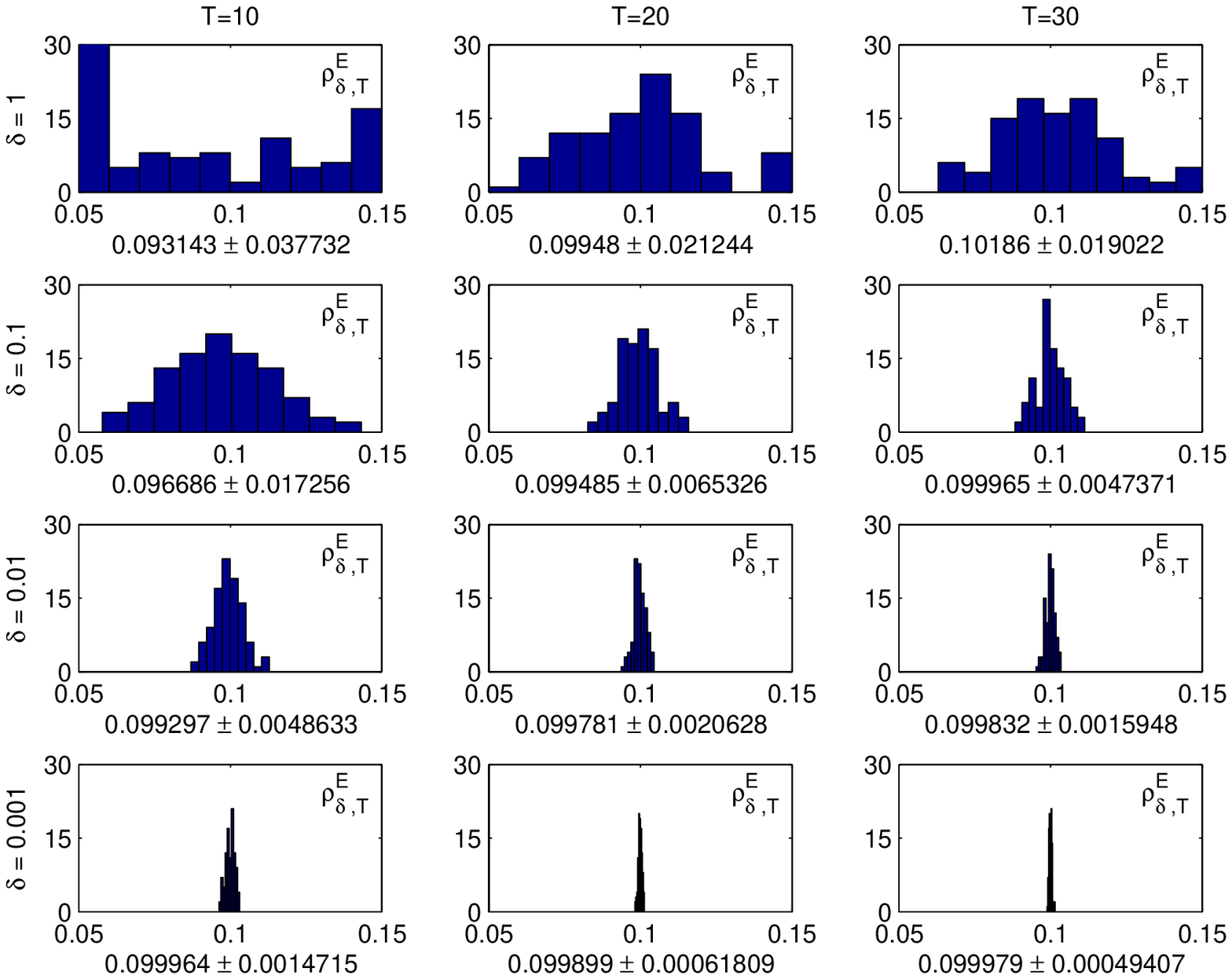} \\
  \includegraphics[width=5in]{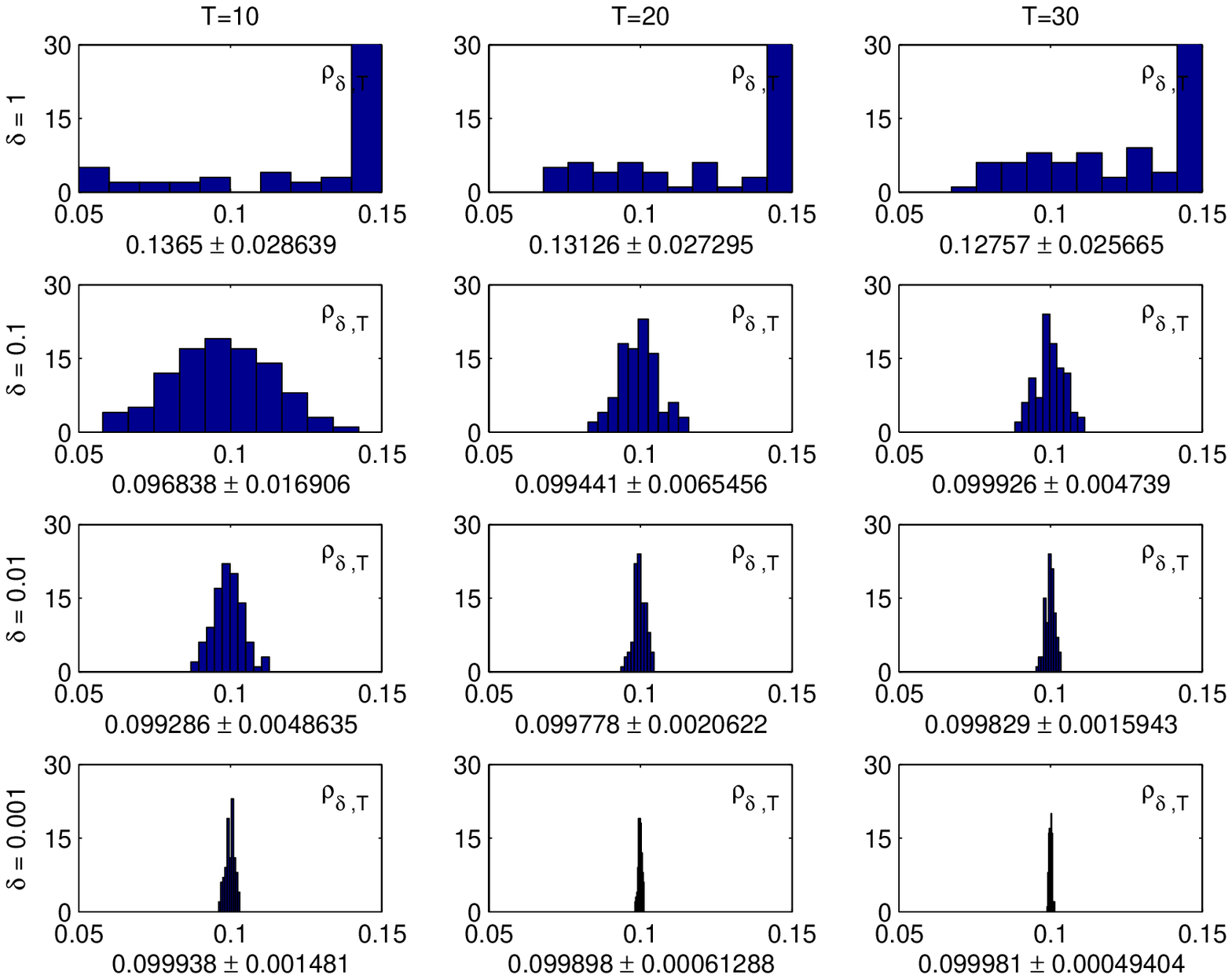}
\end{array}$
\caption{Histograms and confidence limits for the exact ($\widehat{\protect\rho }_{\protect\delta %
,T}^{E}$) and the conventional ($\widehat{\protect\rho }_{\protect\delta ,T}$%
) QML estimators of $\protect\rho $ computed from the Example 2 data
with sampling period $\protect\delta $ and time interval of length
$T$.}
\end{figure}

%% Figure 8a
\subfiguresbegin
\begin{figure}
\centering
 $\begin{array}{c}
  \includegraphics[width=5in]{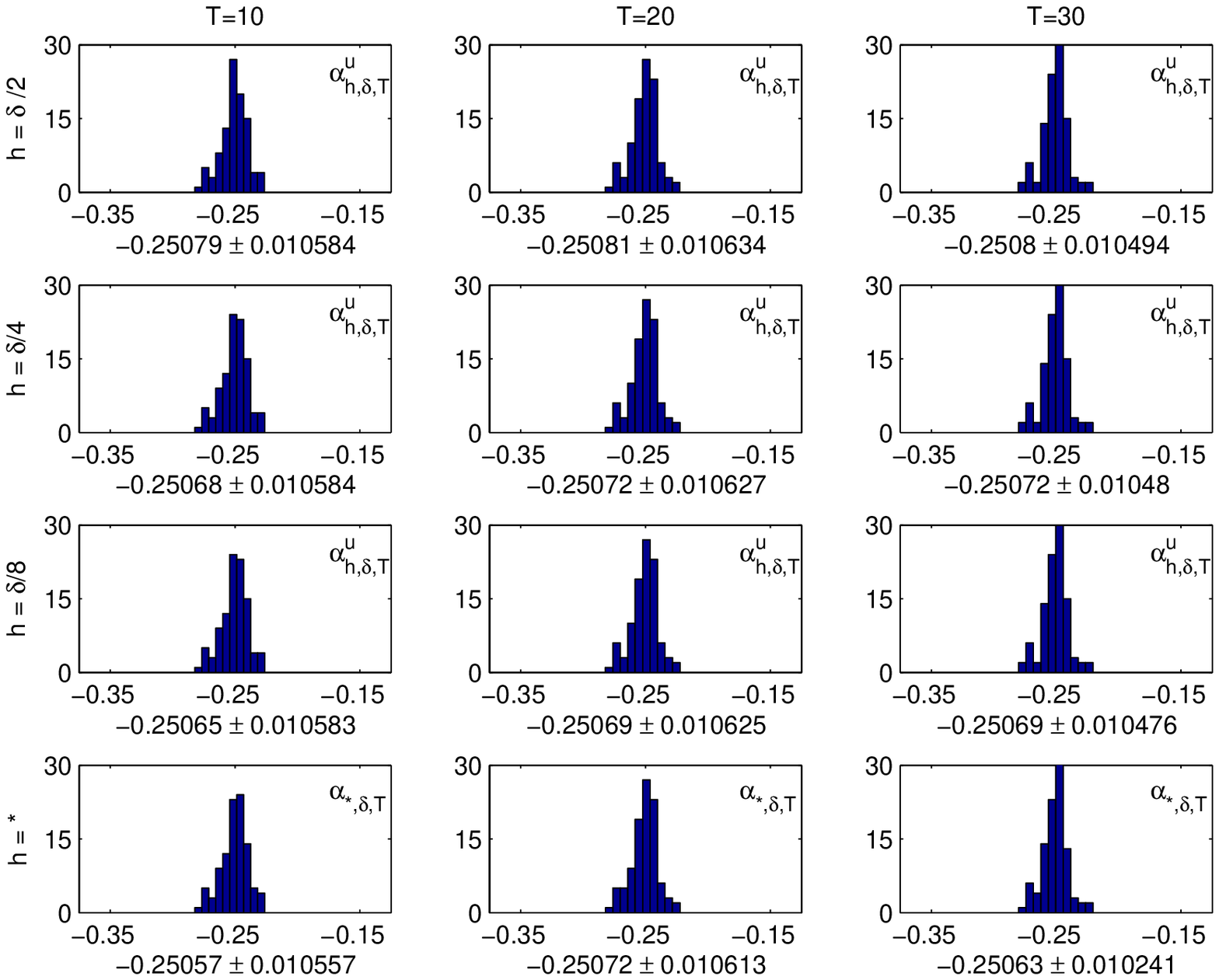} \\
  \includegraphics[width=5in]{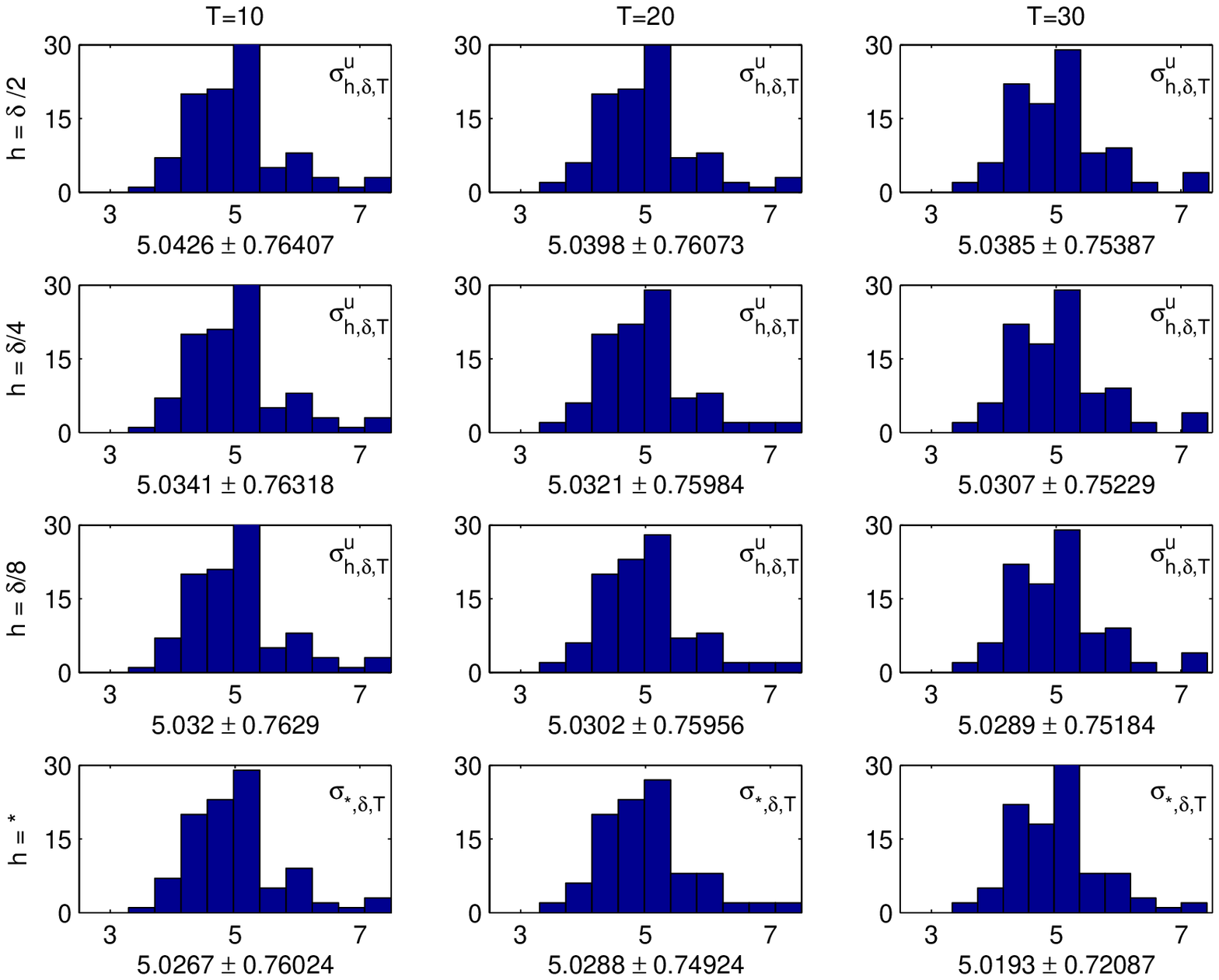}
\end{array}$
\caption{Histograms and confidence limits for the oder-1 QML estimators of $\protect\alpha $ and $%
\protect\sigma $\ computed on uniform $\left( \protect\tau \right)
_{h,T}^{u}$ and adaptive $\left( \protect\tau \right) _{\cdot ,T}$
time discretizations from the Example 2 data with sampling period
$\protect\delta =0.1$ and time interval of length $T$.}
\end{figure}

%% Figure 8b
\begin{figure}
\centering
\includegraphics[width=5in]{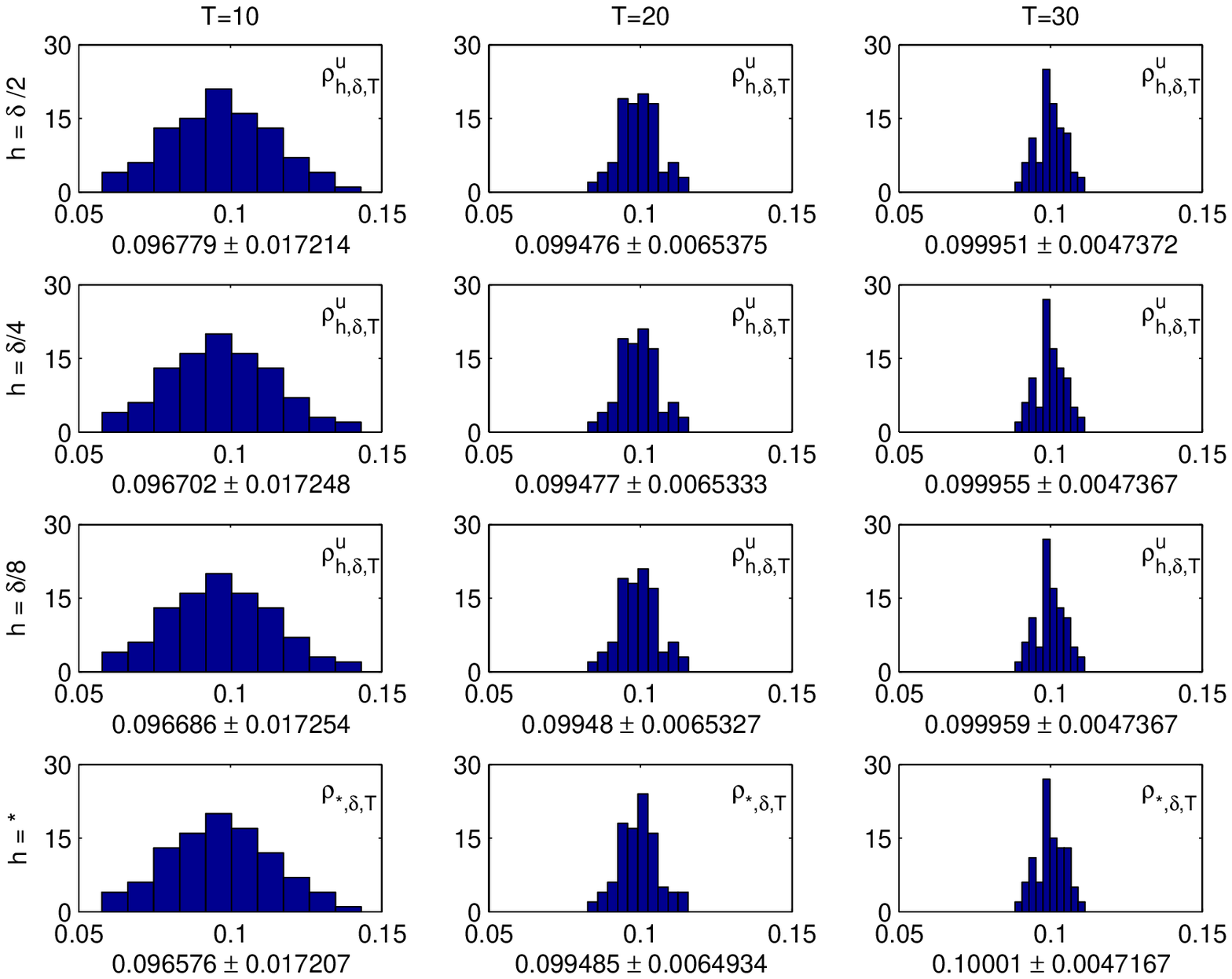}
\caption{Histograms and confidence limits for the oder-1 QML
estimators of $\protect\rho $\ computed on uniform $\left( \protect\tau \right) _{h,T}^{u}$\ and adaptive $%
\left( \protect\tau \right) _{\cdot ,T}$ time discretizations from
the Example 2 data with sampling period $\protect\delta =0.1$ and
time interval of length $T$.}
\end{figure}
\subfiguresend

%% Figure 9
\begin{figure}
\centering
 $\begin{array}{c}
  \includegraphics[width=5in]{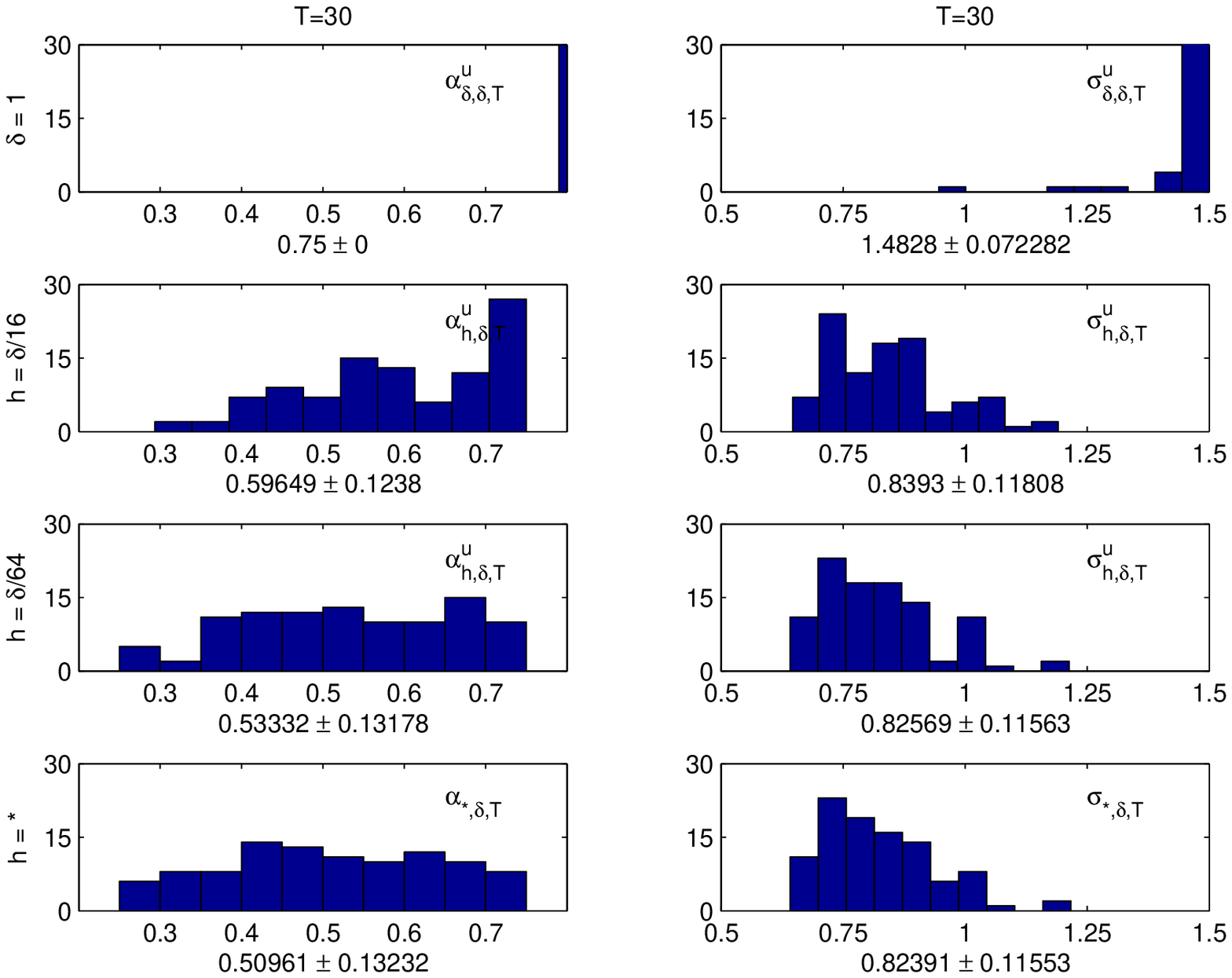} \\
  \includegraphics[width=5in]{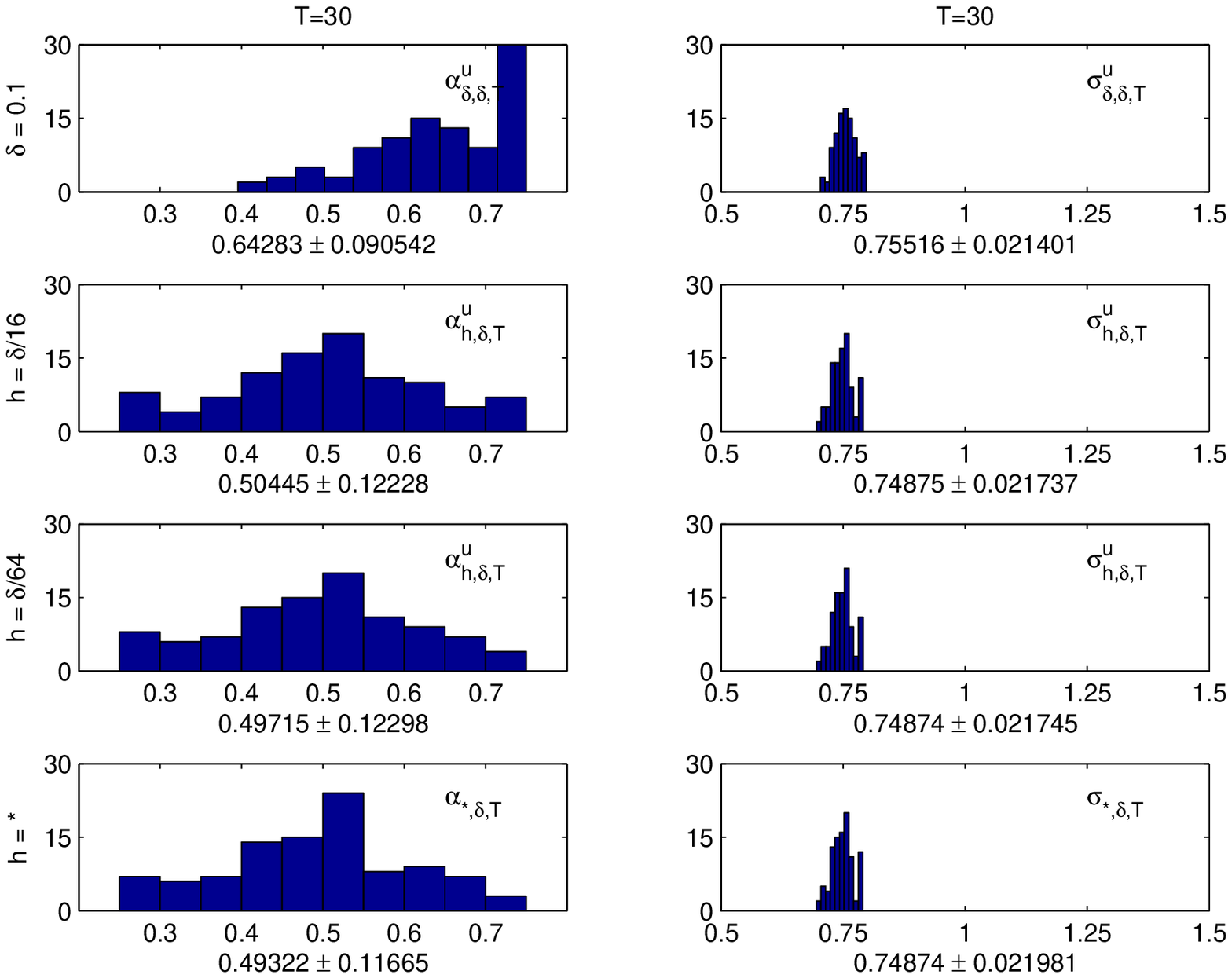}
\end{array}$
\caption{Histograms and confidence limits for the oder-1 QML estimators of $\protect\alpha $ and $%
\protect\sigma $\ computed on uniform $\left( \protect\tau \right)
_{h,T}^{u}$ and adaptive $\left( \protect\tau \right) _{\cdot ,T}$
time discretizations from the Example 3 data with sampling period
$\protect\delta$ and time interval of length $T=30$.}
\end{figure}

%% Figure 10
\begin{figure}
\centering
 $\begin{array}{c}
  \includegraphics[width=5in]{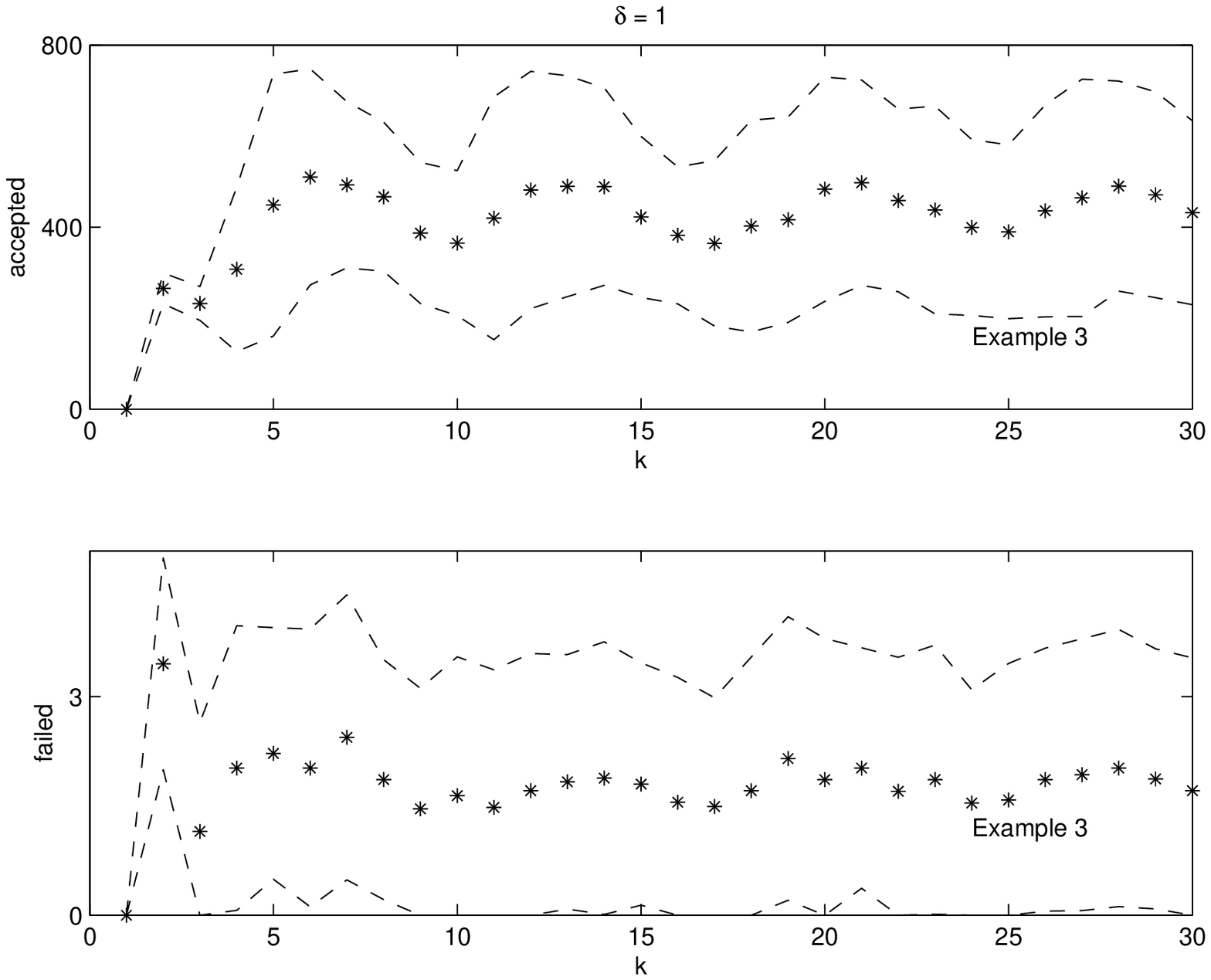} \\
  \includegraphics[width=5in]{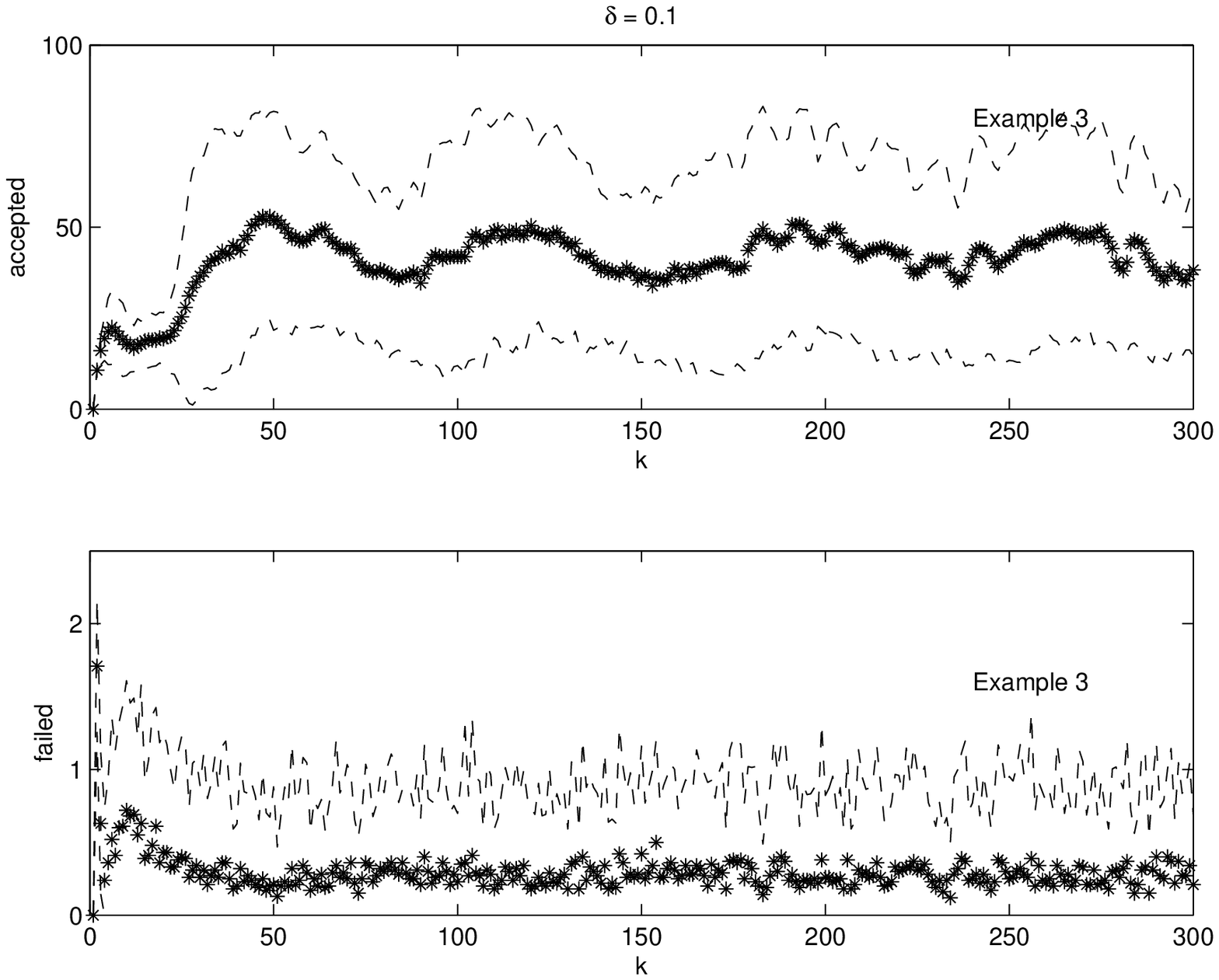}
\end{array}$
\caption{Average (*) and 90\% confidence limits (-) of accepted and
failed steps of the adaptive QML estimator at each $t_{k}\in
\{t\}_{N}$ in the Example 3.}
\end{figure}

%% Figure 11
\begin{figure}
\centering
 $\begin{array}{c}
  \includegraphics[width=5in]{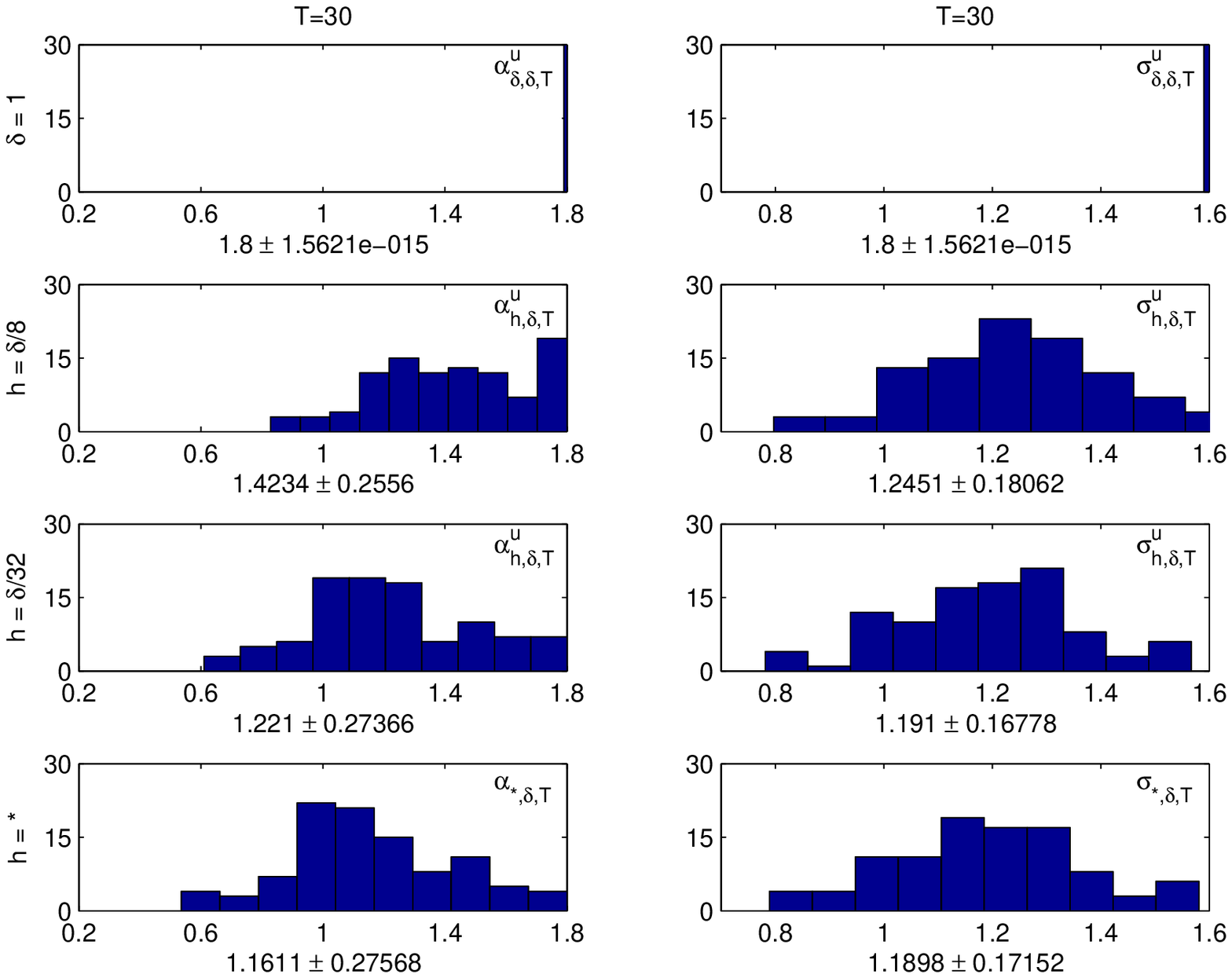} \\
  \includegraphics[width=5in]{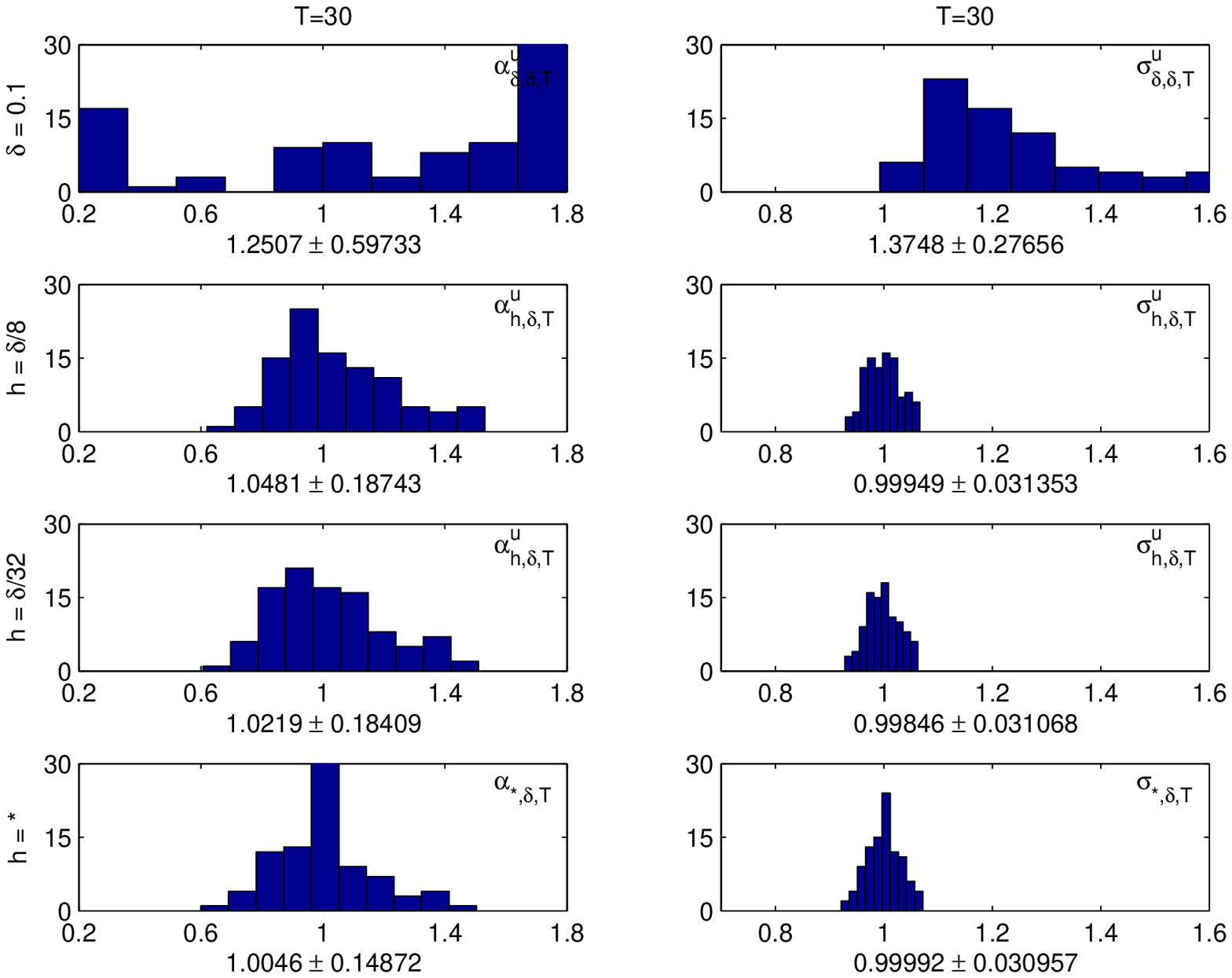}
\end{array}$
\caption{Histograms and confidence limits for the oder-1 QML estimators of $\protect\alpha $ and $%
\protect\sigma $\ computed on uniform $\left( \protect\tau \right)
_{h,T}^{u}$ and adaptive $\left( \protect\tau \right) _{\cdot ,T}$
time discretizations from the Example 4 data with sampling period
$\protect\delta$ and time interval of length $T=30$.}
\end{figure}

%% Figure 12
\begin{figure}
\centering
 $\begin{array}{c}
  \includegraphics[width=5in]{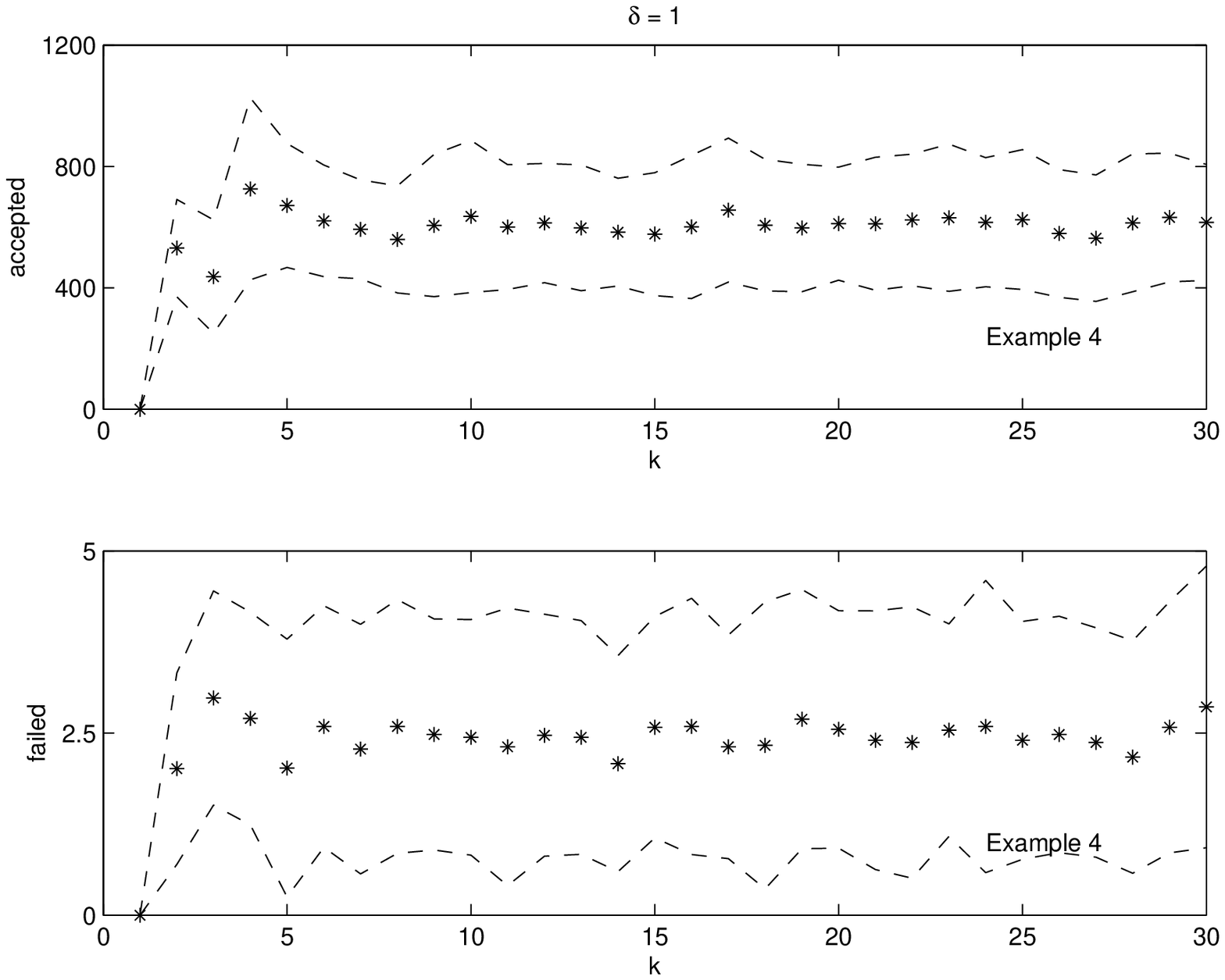} \\
  \includegraphics[width=5in]{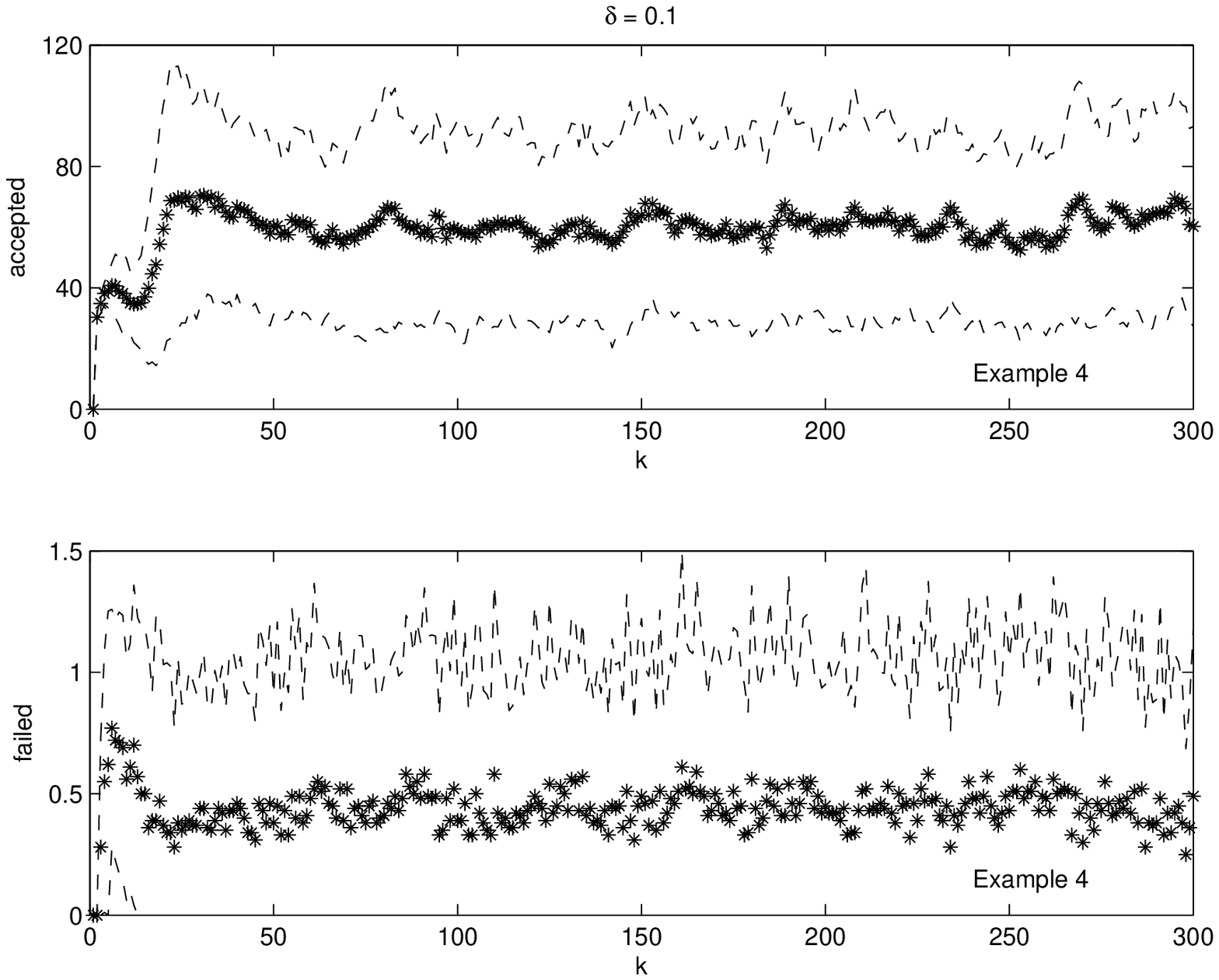}
\end{array}$
\caption{Average (*) and 90\% confidence limits (-) of accepted and
failed steps of the adaptive QML estimator at each $t_{k}\in
\{t\}_{N}$ in the Example 4.}
\end{figure}

\end{document}